\documentclass{amsart}
\usepackage{amssymb, latexsym}
\usepackage{mathtools,color}
\usepackage{enumerate}
\usepackage{footmisc}
\usepackage{framed}

\newcommand\supp{\operatorname{supp}}
\newcommand\card{\operatorname{card}}

\newcommand \cls{|\xi_{st}|:cl}
\newcommand \clth{\pi :cl}
\newcommand \medth{\pi:med}
\newcommand \farth{\pi:far}

\newcommand \bilap{\triangle^{2}}
\newcommand \coef{(t_{0} - t')^{\frac{n}{4}} (t'' -t_{0})^{\frac{n}{4}}}
\newcommand \coefinv{(t_{0} - t')^{-\frac{n}{4}} (t'' -t_{0})^{-\frac{n}{4}}}
\newcommand \p{\partial}

\newcommand \Rec{\mathcal{R}}

\newcommand{\EQQARR}[1]{\begin{equation} \begin{array}{ll} #1 \end{array} \nonumber \end{equation}}
\newcommand{\EQQARRLAB}[1]{\begin{equation} \begin{array}{ll} #1 \end{array}  \end{equation}}
\newcommand{\comp}[1]{{#1}^\complement}

\theoremstyle{plain}
\newtheorem{thm}{Theorem}

\newtheorem{rem}[thm]{Remark}
\newtheorem{prop}[thm]{Proposition}
\newtheorem{lem}[thm]{Lemma}

\newtheorem{defin}{Definition}

\numberwithin{equation}{section} \numberwithin{thm}{section}

\begin{document}

\title[Fourth-Order Schr\"odinger Equations]{A Weak Form Of The Soliton Resolution Conjecture For High-Dimensional
Fourth-Order Schr\"odinger Equations}

\author{Tristan Roy}
\address{Nagoya University}
\email{tristanroy@math.nagoya-u.ac.jp}

\vspace{-0.3in}
\begin{abstract}
We prove a weak form of the soliton resolution conjecture for -$H^{2}$ uniformly bounded in time- solutions of semilinear fourth-order Schr\"odinger equations, in dimensions $n \geq 5$, and with a mass supercritical-energy subcritical power type nonlinearity, by using a strategy devoloped in \cite{taocompact}. More precisely, we prove that the solutions are decomposed into a sum of two terms: a free solution and a nonradiative term that approaches asymptotically
an object that has similar properties to those of a finite sum of solitons. The asymptotic behavior of the nonradiative term is derived from its asymptotic frequency localization and its asymptotic spatial localization. There are two main differences between this paper and \cite{taocompact}. The first one one appears when we prove the asymptotic frequency localization: we fill a gap of regularity by using the better dispersive properties of the high frequency pieces of the free solution. The second one appears when we prove the asymptotic spatial localization. A key estimate depends on the fundamental solution that does not have an explicit form. We overcome the difficulty by introducing a modified fundamental solution and exploiting the symmetries of the characters of the phases.
\end{abstract}

\maketitle
\tableofcontents

\section{Introduction}

In this paper we consider the fourth-order Schr\"odinger equations on $\mathbb{R}^{n}$

\begin{equation}
\begin{array}{ll}
i \partial_{t} u  + \Delta^{2} u  & =  F(u) \\
\end{array}
\label{Eqn:BiSchrod}
\end{equation}
with $F(u)$ a pure power-type nonlinearity, that is $F(u):= \pm |u|^{p-1} u $ and for exponents that are mass-supercritical and
energy-subcritical, that is $1+ \frac{8}{n} < p < 1+ \frac{8}{n-4}$ \footnote{In the sequel we always assume that $p$ lies in this range
unless stated otherwise.}. Fourth-order Schr\"odinger equations have been introduced by
in \cite{karp} and in \cite{karpshag} to take into account the role of small fourth-order dispersion terms
in the propagation of intense laser beams in a bulk medium with Kerr nonlinearity.

These equations have attracted much attention from the community. Sharp dispersive estimates for the biharmonic Schr\"odinger operator have been
obtained in \cite{artzi}. Specific fourth order Schr\"odinger equations have been discussed in
\cite{fibich,guo,hao,segata}. Local well-posedness
for energy subcritical powers (that is $1 < p < 1 + \frac{8}{n-4}$ ) \footnote {for a definition of the concept of energy-subcritical
powers, see for example \cite{taobook} in the context of second-order Schr\"odinger equation and see e.g \cite{pausaderdef1} for
its adaptation to (\ref{Eqn:BiSchrod})} is discussed in \cite{pausaderdef1}. The following theorem is known:

\begin{thm}
Let $(t_{0},u_{0}) \in \mathbb{R} \times H^{2}$. Then

\begin{itemize}

\item \underline{Local existence}: Let $B$ be a bounded subset of $H^{2}$. Then there exist a subset $\tilde{B} \subset H^{2}$, a time
interval $I$ containing $t_{0}$ (called an interval of local existence) such that for all $u_{0} \in B$ there exists a solution
$u: I \rightarrow \tilde{B}$ satisfying  $u(t_{0})=u_{0}$. Furthermore the size of the interval (called the time of local existence) depends on $\| u_{0} \|_{H^{2}}$ and the map
$u_{0} \rightarrow u$ is Lipschitz continuous from $B$ to $\tilde{B}$.

\item  \underline{Uniqueness}: If two solutions $u: I \rightarrow H^{2}$ and $\tilde{u}: I \rightarrow H^{2}$ agree on at least one time in $I$, then
they agree for all time in $I$

\item  \underline{Conservation of the mass}: we have for all time $t \in I$

\begin{equation}
\begin{array}{ll}
\| u(t) \|_{L^{2}} & = \| u_{0} \|_{L^{2}}
\end{array}
\nonumber
\end{equation}

\end{itemize}

By solution $u : I \rightarrow \tilde{B}$ we mean a function $u \in C_{t}^{0} H^{2} (I \times \mathbb{R}^{n})$
that satisfies the Duhamel formula, that is

\begin{equation}
\begin{array}{ll}
u(t) & := e^{i(t-t_{0}) \bilap} u(t_{0}) - i \int_{t_{0}}^{t} e^{i (t-t') \bilap} F(u(t')) \, dt^{'}
\end{array}
\label{Eqn:DuhamelBiSchrod}
\end{equation}
\label{Thm:ToProve}
\end{thm}
This allows to define the maximal interval of existence $I_{max}$, i.e the maximal interval on which $u$ is defined. \\ \\

In the defocusing case (that is $F(u):= -|u|^{p-1} u$), we expect that the solutions exist for all time (i.e $I_{max} =\mathbb{R}$) and that
the solution scatter, i.e they behave asymptotically like a solution to the linear fourth-order
Schr\"odinger equation. The long-time behavior of solutions in this case  has been studied by many
authors: see e.g \cite{miaodef,pausaderdef1,pausaderdef2}. In the focusing case (that is $F(u)= + |u|^{p-1} u$), we do not necessarily expect scattering. For example, it is well known (see e.g \cite{lenz,fibich} ) that there are non trivial solutions of \footnote{Here $s_c$ denotes the critical exponent, i.e
$s_c:= \frac{n}{2} - \frac{4}{p-1}$}

\EQQARR{
\bilap Q  +  Q - |Q|^{p-1} Q =0,
}
which provide solutions

\EQQARRLAB{
u(t,x) := e^{-i t} Q(x)
\label{Eqn:StatSol}
}
which do not scatter. One can also construct solutions that blow-up in finite time: see e.g
\cite{lenz,chao}. We refer to \cite{qing} for scattering results under suitable assumptions. See  \cite{miaoxuzhao,pausaderfocrad} for scattering results
regarding the energy-critical powers (i.e $ p = 1 + \frac{8}{n-4} $). We refer to \cite{pausshao} for scattering results regarding the mass-critical powers
(that is $p= 1 + \frac{8}{n}$).

In this paper, we are interested in the asymptotic behavior of -$H^{2}$ uniformly in time - bounded solutions for mass-supercritical and energy-subcritical
exponents on the maximal time interval of existence. In this case, from a well-known consequence
of Theorem \ref{Thm:ToProve}, the solution exists for all time $T$. In the defocusing case, as we have seen, we expect that the solution
behaves like a free fourth-order Schr\"odinger solution. But in the focusing case, it is believed that the solution divides into two parts
as times goes to infinity. The first one is a radiative part, that is a linear fourth-order Schr\"odinger solution.
The second part approaches a finite sum of stationary solitons (such as (\ref{Eqn:StatSol})) or travelling solitons. By solitons we mean global and
non scattering solutions. This is the \textit{soliton resolution conjecture}: in other words, the only obstacle to scattering is the formation of these solitons. See \cite{soffer} for further discussions regarding this conjecture.

To this end the notion of $G$-precompactness with $J$ components was defined in \cite{taocompact}

\begin{defin}{\textbf{"$G$-precompactness with $J$ components"}, \cite{taocompact}}
Let $J \in \mathbb{N}$. We say that a set $E \subset H^{2}$ is a $G$-precompact set
with $J$ components if there exist a compact set $K \subset H^{2}$ such that for all $f \in E$ one can find
$(x_{1},... x_{J}) \in (\mathbb{R}^{n})^{J}$ and $(h_{1},...,h_{J}) \in K^{J}$ such that

\begin{equation}
\begin{array}{ll}
f(x) & = \sum_{j=1}^{J} h_{j}(x-x_{j})
\end{array}
\nonumber
\end{equation}
\end{defin}

\begin{rem}
Two comments:

\begin{itemize}

\item The orbit of (\ref{Eqn:StatSol}) is $G$-precompact with one component.

\item The notation $G$ corresponds to the action of the translation group $G$ on the compact set $K$ and the generated set is denoted by $GK$. An
equivalent definition is the following one: $E \subset H^{2}$ is a $G$-precompact set with $J$ components if we have $E \subset J(GK)$ with

\begin{equation}
\begin{array}{ll}
J(GK):= & \{ f_{1} + ....+f_{j}; f_{1},..., f_{J} \in GK \}
\end{array}
\nonumber
\end{equation}

\end{itemize}
\end{rem}
With this in mind, we can now state the main result of this paper:

\begin{thm}{\textbf{"Weak form of the soliton resolution conjecture"} }
Let $n \geq 5$ and $u$ solution of (\ref{Eqn:BiSchrod}) with $ 1 + \frac{8}{n} < p < 1 + \frac{8}{n-4}$ and with data $u_{0} \in H^{2}$. Let $I_{max}$ be its maximal time interval of existence. Assume that
$u$ is - $\dot{H}^{2}$ uniformly bounded in time -, that is

\begin{equation}
\begin{array}{ll}
\| u \|_{L_{t}^{\infty} \dot{H}^{2} (I_{max})} & \leq  M
\end{array}
\label{Eqn:Boundu}
\end{equation}
for some $M:=M(\| u_{0} \|_{H^{2}}) < \infty$. Then the solution exists globally in time, that is $I_{max}= \mathbb{R}$. Moreover
there exist $(u_{+},v) \in H^{2} \times H^{2}$  and a $G$-precompact set $K$ with $J$ components such that

\begin{equation}
\begin{array}{ll}
u(t) & = e^{i t \bilap} u_{+} + v(t)
\end{array}
\nonumber
\end{equation}
with

\begin{equation}
\begin{array}{ll}
\lim_{t \rightarrow \infty} dist_{H^{2}}(v(t),K) & = 0.
\end{array}
\nonumber
\end{equation}
We say that $v$ is the \textit{nonradiative part} of the solution and $e^{i t \bilap} u_{+}$ is the \textit{radiative part} (or \textit{dispersive part})
of it.
\label{Thm:WeakSoliton}
\end{thm}

\begin{rem}
Two comments:

\begin{itemize}

\item
Theorem \ref{Thm:WeakSoliton} is consistent with the soliton resolution conjecture. Indeed, we expect the nonradiative part of the solution to approach
a finite sum of solitons. We expect the orbit of a stationary or traveling soliton to be $G$-precompact with one component. We expect
the orbit of the superposition of $J$ solitons to be $G$-precompact with $J$ components.

\item
This result is a weak form of the soliton conjecture. It is weak since it remains to better characterize the $G$-precompact set with $J$ components:
ideally one would like to prove that, in fact, this $G$-precompact set with $J$ components is a finite sum of stationary of travelling solitons.

\end{itemize}
\end{rem}

\begin{rem}
Notice that by combining (\ref{Eqn:Boundu}) with conservation of mass we have in fact
$ \| u \|_{L_{t}^{\infty} H^{2}(I_{max})} < \infty$. Therefore the solution exists globally in time.
\label{rem:inhomhom}
\end{rem}

The following proposition shows that in order to prove that an orbit $f: \mathbb{R} \rightarrow H^{2} $ approaches
a $G$-precompact set with $J$ components it is enough to prove that it is asymptotically bounded, localized in frequency and in space:

\begin{prop}{\textbf{"$G$-precompact set and asymptotic spatial and frequency localization "} \cite{taocompact}}

Let $f: \mathbb{R} \rightarrow H^{2}$.Then the following are equivalent

\begin{itemize}

\item There exists a $G$-precompact set $K \subset H$ with $J$ components such that  \\
$\lim_{t \rightarrow \infty} dist_{H^{2}}(f(t),K)=0$

\item $f$ is asymptotically bounded, that is

\begin{equation}
\begin{array}{ll}
\overline{\lim}_{t \rightarrow \infty} \| f(t) \|_{H^{2}} & < \infty,
\end{array}
\nonumber
\end{equation}
$f$ is asymptotically localized in frequency i.e for any $\epsilon >0$ one can find $\mu >0$ such that

\begin{equation}
\begin{array}{ll}
\overline{\lim}_{t \rightarrow \infty} \| P_{\geq \mu^{-1}} f(t) \|_{H^{2}} & \leq \epsilon \\
\overline{\lim}_{t \rightarrow \infty} \| P_{\leq \mu} f(t) \|_{H^{2}} & \leq \epsilon,
\end{array}
\nonumber
\end{equation}
and $f$ is asymptotically localized in space, i.e there exist $x_{1},..., x_{J}: \mathbb{R} \rightarrow \mathbb{R}^{n}$ for which we have

\begin{equation}
\begin{array}{ll}
\overline{\lim}_{t \rightarrow \infty} \int_{\inf_{1 \leq j \leq J} |x-x_{j}(t)| \geq \mu^{-1}} |f(t,x)|^{2} \, dx & \leq \epsilon^{2} \cdot
\end{array}
\nonumber
\end{equation}

\end{itemize}
\label{prop:EquivCompactSpatialFreq}
\end{prop}

\section{Basic estimates}

In this section we recall some basic estimates that we constantly use throughout the proof of
Theorem \ref{Thm:WeakSoliton} . \\
We first state the dispersive estimates of the free solution:

\begin{prop}{\textbf{"Dispersive Estimates"} \cite{artzi}}

For all $\alpha \in \mathbb{N}^{n}$ we have
\begin{equation}
\begin{array}{ll}
\| D^{\alpha} e^{i t \bilap} f \|_{L^{\infty}} & \lesssim \frac{1}{|t|^{ \frac{n+ |\alpha|}{4}}} \| f \|_{L^{1}}
\end{array}
\label{Eqn:Dispers}
\end{equation}
\end{prop}
We then state the Strichartz estimates:

\begin{prop}{\textbf{"Strichartz estimates"} \cite{pausaderdef1}}

Let $u$ be a solution of $i \partial_{t} u + \triangle^{2} u = G $   on an interval $I=[a,b]$. Let $(q,r)$ $B$- admissible i.e
$\frac{1}{q} + \frac{n}{4r} = \frac{n}{8}$, $2 \leq q,r \leq \infty $, $(q,r) \neq (2,\infty)$. Let $(\tilde{q}, \tilde{r})$ be a bipoint
lying in the dual set of  $B$-admissible points, i.e there exists $(x,y)$ $B$-admissible such that $\frac{1}{\tilde{q}} + \frac{1}{x}=1$ and
$\frac{1}{\tilde{r}} + \frac{1}{y}=1$. Then we have

\begin{itemize}

\item \textbf{"Strichartz estimates with no derivative"}

\begin{equation}
\begin{array}{ll}
\| u \|_{L_{t}^{q} L_{x}^{r} (I)} & \lesssim  \| u (a) \|_{L^{2}} +  \| G \|_{L_{t}^{\tilde{q}} L_{x}^{\tilde{r}} (I) }
\end{array}
\label{Eqn:Strich}
\end{equation}

\item \textbf{"Strichartz estimates with gain of derivative"}

\begin{equation}
\begin{array}{ll}
\| \triangle u \|_{ L_{t}^{q} L_{x}^{r}(I)} & \lesssim \| \triangle u (a) \|_{L^{2}} + \| \nabla G \|_{L_{t}^{2} L_{x}^{\frac{2n}{n+2}}(I)}
\end{array}
\label{Eqn:StrichGain}
\end{equation}
\end{itemize}
\end{prop}

\begin{rem}
The dual inequality of (\ref{Eqn:Strich}) with $G=0$ is

\begin{equation}
\begin{array}{ll}
\| \int_{I} e^{-it \bilap}  h(t) \, dt \|_{L^{2}} & \lesssim  \| h \|_{L_{t}^{\tilde{q}} L_{x}^{\tilde{r}} (I) }
\end{array}
\nonumber
\end{equation}
\label{Rem:Dual}
\end{rem}

\section{Notation}

In this section we set some notation.\\
\\
Let $\alpha$ be a nonegative constant. We denote by $[\alpha]$ the integer part of $\alpha$. We denote by $\alpha+$ (resp.$\alpha-$) a number that
is slightly larger (resp. smaller) respectively than $\alpha$. Let $\beta$ be another nonnegative constant. We say that $\alpha \lesssim  \beta$ (resp. $\alpha  \ll \beta$ ) if there exist a positive constant (resp. small positive constant) such that
$\alpha \leq C \beta$. \\
\\
Let $r_{0}:= \frac{2n}{n-4}-$, $\tilde{r}_{0}$ be such that
$\frac{1}{\tilde{r}_{0}}=\frac{1}{r_{0}}-\frac{1}{n}$, $q_{0}$ be such that $(q_{0},r_{0})$ $B-$ admissible and Q be such that

\begin{equation}
\begin{array}{ll}
\frac{n+2}{2n} & = \frac{p-1}{Q} + \frac{1}{\tilde{r}_{0}} \cdot \\
\end{array}
\nonumber
\end{equation}
Notice that for mass supercritical-energy subcritical exponents $p$ we have $ 2 \leq  Q  < \frac{2n}{n-4}$ and in particular the Sobolev embedding applies, i.e

\begin{equation}
\begin{array}{ll}
\| f \|_{L^{Q}} & \lesssim \| f \|_{H^{2}} \cdot
\end{array}
\label{Eqn:SobH2}
\end{equation}
Let

\begin{equation}
\begin{array}{ll}
\tilde{p} & := \left\{
\begin{array}{l}
\frac{2}{p}, \, p \leq 2 \\
1, \, p > 2
\end{array}
\right.
\end{array}
\nonumber
\end{equation}
Let
\begin{equation}
\begin{array}{ll}
X(I) & = L_{t}^{q_{0}} W^{2,r_{0}}(I) \cap L_{t}^{\infty} L_{x}^{Q}(I) \cdot
\end{array}
\nonumber
\end{equation}
If $E \subset \mathbb{R}^{n}$ is a set then $\mathbf{1}_{E}$ is the characteristic function of $E$. Given $x_{0} \in \mathbb{R}^{n}$ and $R \geq 0$, let

\begin{eqnarray*}
\mathcal{B}(x_{0},R):= \left\{ x \in \mathbb{R}^{n}, |x- x_{0}| < R \right\} \cdot
\end{eqnarray*}
\\
\\
Let $\phi$ be a bump function such that $\phi(\xi)=1$ if $|\xi| \leq 1$ and $\phi(\xi)=0$ if $|\xi| \geq 2$. Let $\psi(\xi):= \phi(\xi) - \phi(2 \xi)$. If $N \in 2^{\mathbb{Z}}$ is a dyadic number then

\begin{equation}
\begin{array}{ll}
\widehat{P_{\leq N} f}(\xi) & := \phi \left( \frac{\xi}{N}  \right) \hat{f}(\xi) \\
\widehat{P_{N} f}(\xi) & := \psi \left( \frac{\xi}{N} \right) \hat{f}(\xi) \\
\widehat{P_{>N} f} (\xi) & := \hat{f}(\xi) - \widehat{P_{\leq N} f}(\xi)
\end{array}
\nonumber
\end{equation}

\section{Organization of the paper and novelties}

We first explain how this paper is organized. We follow the strategy that was developed in \cite{taocompact} in the study of
the semilinear Schr\"odinger equations with mass-supercritical energy-subcritical powers.
In Section \ref{Sec:LocalEst} we prove some local estimates, that is estimates that are only useful on
short time intervals. In Section \ref{Sec:Const} we prove that the solution $u$ can be decomposed into two parts: a free fourth-order
Schr\"odinger solution and a solution $v(t)$ of (\ref{Eqn:BiSchrod}) that shall be considered as the nonradiative part of $u$. We also characterize $v(t)$ as an integral depending on $t$ and as a weak limit of an asymptotic integral: see (\ref{Eqn:vt0}) and (\ref{Eqn:vtinf}). We would like to prove that the orbit of $v(t)$ approaches a $G$-precompact set with $J$ components. To this end it is enough, by Proposition \ref{prop:EquivCompactSpatialFreq}, to prove that $v$ is asymptotically bounded, localized in frequency and in space. In Section \ref{Sec:FreqLoc} we prove that $v$ satisfies the asymptotic frequency localization, by letting (\ref{Eqn:vt0}) interact with (\ref{Eqn:vtinf}). In Section \ref{Sec:SpatLoc} we
prove the asymptotic spatial localization. To this end, we first locate the spots of mass concentration of $v$ at a time $t_{0}$ large enough. Then, we prove that the free fourth-order Schr\"odinger solution with data far from these spots is small. This allows to estimate several quantities far away
from these spots. We prove, by a perturbation argument, that a Strichartz norm of the solution is locally small. Then the $L^{2}$ norm of $v$  is estimated by letting again (\ref{Eqn:vt0}) interact with (\ref{Eqn:vtinf}). If the interaction is large then we use the smallness of this Strichartz norm  and if the interaction is large, then we use dispersive estimates. This proves a partial spatial localization, partial since the points of concentration (and their number) depend on the size of the decay of the $L^{2}$ norm of $v$ away from them. In Subsection \ref{subsec:finalspatial}, we prove that the
partial spatial localization can be upgraded to final spatial localization.\\
\\
We now discuss the main novelties of this paper.\\
In Section \ref{Sec:FreqLoc}, one has to make asymptotically the high frequency component of the nonradiative part small
by letting the high frequency part of (\ref{Eqn:vt0}) interact with that of (\ref{Eqn:vtinf}). Unfortunately the function
$F$ is only $C^{1}$ whereas we have to control quantities in $H^{2}$.  Fortunately, since we know that the Paley-Littlewood pieces $P_{K}$ of the
free solution have better dispersion properties locally as $K$ goes to infinity (a phenomenon that, to our knowledge, was observed in \cite{pausaderdef1}), we
can fill this gap of regularity by using them extensively on a large portion of the interaction, and, on the small remaining part, we use the local estimates proved in Section \ref{Sec:LocalEst}.\\
A key estimate (namely (\ref{Eqn:EstK})) to prove the asymptotic spatial localization depends on the interaction of two fundamental
solutions evaluated at two different points: see (\ref{Eqn:DefK}). Unfortunately the fundamental solution of the fourth-order
Schr\"odinger equation does not have an explicit formula (unlike the Schr\"odinger equation), which makes the estimate delicate to prove. In order to
overcome the difficulty, we use the following strategy:

\begin{itemize}

\item we factor a well-chosen phase out of the fundamental solution and we introduce a modified fundamental solution: see
(\ref{Eqn:Phaseout}) and (\ref{Eqn:DeftildeI}).

\item we estimate the derivatives of the modified fundamental solution by weights: see (\ref{Eqn:EsttildeI}). The estimates
are mostly derived by integration by parts of the phase in the polar coordinates $(r,\theta)$. We mention that the
polar coordinates were already used in \cite{artzi}, where the authors performed integration by parts w.r.t
$\theta$ followed by integration by parts w.r.t $r$. Here, in order to use the oscillations of the phase to the most, we first determine the region for which the integration by parts w.r.t $r$ yields a better decay than that w.r.t  $\theta$. We then introduce an homogeneous function of degree zero (in the spirit of \cite{stein}, p 345) in order to emphasize this region. In this region we integrate by parts w.r.t $r$. In the complement of this region we integrate by parts w.r.t $\theta$.

\item we estimate the interaction of two fundamental solutions by using the estimates of the derivatives of the modified fundamental solution
and the bipolar coordinates $(\rho,\sigma)$, taking advantage of the symmetries of the character of the phase. When we pass to the
bipolar coordinates, the region of integration $\Rec$ is more bounded and the integrand has less regularity as we approach the boundary $\p \Rec$. So in one region
of the plane $(\rho,\sigma)$ we integrate by parts w.r.t $\rho':= \rho + \sigma$ a well-chosen amount of times to kill the singularities of the integrand. In the other region, the procedure described above does not kill the singularities. Instead we proceed as follows. We choose not to integrate by parts w.r.t  the bipolar coordinates in the subregion close to $\p \Rec$ and we estimate directly the integrals involved in this subregion, provided that they are integrable. This integrability holds for high dimensions thanks to the weights of the modified fundamental solution. It barely fails for lower dimensions thanks again to the weights. So one integrates by parts the phase just a few times w.r.t. $\rho$, or $\sigma$  to get integrability. In the region far from the boundary, once again we determine the subregion for which the integration by parts w.r.t $\rho$ yields a better decay than that w.r.t to $\sigma$ ; we then introduce an homogeneous function of degree zero; we integrate by parts w.r.t $\rho$ in this subregion and in the complement of this subregion we
integrate by parts w.r.t $\sigma$.
\end{itemize}

$\textbf{Acknowledgements}:$ The author would like to thank Terence Tao for suggesting him this problem. He also would like to thank Benoit Pausader
for interesting discussions regarding the fourth-order Schr\"odinger equations.

\section{Local estimates}
\label{Sec:LocalEst}

In this section we prove some local estimates.

\begin{prop}{\textbf{"Local estimates"}}

For all $(q,r)$ $B$-admissible, there exists $\alpha >0$ such that if
 $m \in \{0,1,2 \}$ then

\begin{equation}
\begin{array}{ll}
\| D^{m} u \|_{L_{t}^{q} L_{x}^{r}(I)} & \lesssim \langle |I| \rangle^{\alpha} \cdot
\end{array}
\label{Eqn:LocalEst}
\end{equation}
\label{Prop:LocalEst}
\end{prop}

\begin{proof}

Let $I=[a,b]$. From (\ref{Eqn:Boundu}), (\ref{Eqn:StrichGain}), and H\"older in time

\begin{equation}
\begin{array}{ll}
\| \triangle u \|_{L_{t}^{q} L_{x}^{r} (I)} & \lesssim \| \triangle u(a) \|_{L^{2}} +
\| \nabla F(u) \|_{ L_{t}^{2} L_{x}^{\frac{2n}{n+2}} (I) } \\
& \lesssim 1 + |I|^{\alpha} \| u \|^{p-1}_{L_{t}^{\infty} L_{x}^{Q} (I)}
\| \nabla  u \|_{L_{t}^{q_{0}} L_{x}^{\tilde{r}_{0}} (I)} \\
& \lesssim 1 + |I|^{\alpha} \| \triangle u \|_{ L_{t}^{q_{0}} L_{x}^{r_{0}} (I)}
\end{array}
\label{Eqn:BoundH2Local}
\end{equation}
By conservation of mass, (\ref{Eqn:Boundu}), (\ref{Eqn:Strich}), (\ref{Eqn:SobH2}), and Remark \ref{rem:inhomhom} we have

\begin{equation}
\begin{array}{ll}
\| u \|_{L_{t}^{q} L_{x}^{r} (I)} & \lesssim \| u (a) \|_{L^{2}} + \| |u|^{p-1} u  \|_{L_{t}^{2} L_{x}^{\frac{2n}{n+4}} (I) } \\
& \lesssim 1  + |I|^{\alpha} \| u \|^{p-1}_{L_{t}^{\infty} L_{x}^{Q} (I) } \| u \|_{ L_{t}^{q_{0}} L_{x}^{r_{0}} (I)} \\
& \lesssim 1  + |I|^{\alpha}  \| u \|_{L_{t}^{q_{0}} L_{x}^{r_{0}} (I)}
\end{array}
\label{Eqn:BoundL2Local}
\end{equation}
A continuity argument (with $(q,r):=(q_{0},r_{0})$)  shows that there exists $ 0 < \epsilon \ll 1$ such that if $|I| \leq \epsilon$ then
(\ref{Eqn:LocalEst}) holds if $m=0,2$. With this in mind, we see that, by interpolation, that (\ref{Eqn:LocalEst})
also holds if $m=1$. In the general case, if $(q,r)$ is $B-$ admissible, then we
reapply (\ref{Eqn:BoundH2Local}) and (\ref{Eqn:BoundL2Local}), taking into account that (\ref{Eqn:LocalEst}) was already proved for
the particular case $(q,r):=(q_{0},r_{0})$. Now let $I$ be an arbitrarily large interval. We divide $I$ into subintervals $|J| =\epsilon$ (except maybe the last one) and we apply to each of these subintervals (\ref{Eqn:LocalEst}); by summation we see that (\ref{Eqn:LocalEst}) holds on $I$.

\end{proof}

\section{Construction of the nonradiative part}
\label{Sec:Const}

In this section we prove that the solution can be divided into two parts: the radiative part
 and the nonradiative part. We give elementary properties of the nonradiative part.

\begin{prop}{\textbf{"Construction and properties of the non radiative part"}}
There exists a unique decomposition

\begin{equation}
\begin{array}{ll}
u(t) & = e^{i t \bilap} u_{+} + v(t)
\end{array}
\label{Eqn:Decomp}
\end{equation}
such that

\begin{equation}
\begin{array}{ll}
e^{-it \bilap} u(t) & \rightharpoonup_{t \rightarrow \infty} u_{+}
\end{array}
\label{Eqn:WeakConvUplus}
\end{equation}
and

\begin{equation}
\begin{array}{ll}
e^{-it \bilap} v(t) & \rightharpoonup_{|t| \rightarrow \infty} 0.
\end{array}
\end{equation}
Moreover

\begin{equation}
\begin{array}{ll}
\| u_{+} \|_{H^{2}} & \lesssim  1 \\
\| v(t) \|_{H^{2}} & \lesssim 1 \\
\lim_{|t| \rightarrow \infty} \| u(t) \|^{2}_{H^{2}} - \| v(t) \|^{2}_{H^{2}} - \| e^{i t \triangle^{2}} u_{+} \|_{H^{2}} & =0,
\end{array}
\label{Eqn:BounduplusvH2}
\end{equation}

\begin{equation}
\begin{array}{ll}
v(t) & = e^{it \bilap} (u(0)- u_{+}) - i \int_{0}^{t} e^{i(t-t^{'}) \bilap} F(u(t^{'})) \, dt^{'},
\end{array}
\label{Eqn:vt0}
\end{equation}
and

\begin{equation}
\begin{array}{ll}
i \int_{t}^{T} e^{i(t-t^{'}) \bilap} F(u(t^{'})) \, dt^{'} & \rightharpoonup_{T \rightarrow \infty} v(t).
\end{array}
\label{Eqn:vtinf}
\end{equation}

\label{Prop:WeakBoundConst}
\end{prop}

\begin{proof}
We shall only prove the existence of this decomposition, since the uniqueness along with the estimates are straightforward application of arguments explained
in \cite{taocompact}. Let $\phi \in S(\mathbb{R}^{n})$ and $t_{2} \geq t_{1}$. Let $\epsilon >0$. We have for $t_{1}$ large

\begin{equation}
\begin{array}{ll}
\langle e^{-i t_{1} \bilap} u(t_{1})- e^{-i t_{2} \bilap} u(t_{2}), \phi \rangle_{H^{2}} & = \sum_{m \in \{ 0,1,2 \}} i X_{m}
\end{array}
\nonumber
\end{equation}
with

\begin{equation}
\begin{array}{ll}
 X_{m} & := \langle  D^{m} \int_{t_{1}}^{t_{2}}  e^{-it^{'} \bilap} F(u(t^{'})) \, dt^{'}, \, D^{m} \phi  \rangle_{L^{2}} \\
 \end{array}
 \nonumber
 \end{equation}
One would like to compute $D^{2}F(u(t'))$ but it is not possible since $F \notin C^{2}$.  Instead we proceed as follows

 \begin{equation}
 \begin{array}{ll}
X_{m} & = \int_{t_{1}}^{t_{2} } \langle  F(u(t^{'})), \, D^{2m} e^{it^{'} \bilap} \phi \rangle_{L^{2}} \, dt^{'} \\
 & \lesssim  \int_{t_{1}}^{t_{2} }  \| F(u(t^{'})) \|_{L^{\tilde{p}}} \| D^{2m} e^{it^{'} \bilap} \phi \|_{L^{\tilde{p}^{'}}} \, dt^{'} \\
 & \lesssim \epsilon,
\end{array}
\nonumber
\end{equation}
where the last inequality follows from interpolation between (\ref{Eqn:Dispers}) and the trivial estimate
$ \| D^{2m} e^{ i t' \bilap} \phi \|_{L^{2}} =  \| D^{2m} \phi \|_{L^{2}}$, and the following equality

\begin{equation}
\begin{array}{ll}
\| F(u(t^{'})) \|_{L^{\tilde{p}}}  & \lesssim
\left\{
\begin{array}{l}
\| u(t^{'})\|^{p}_{L^{2}}, \, p  \leq 2 \\
\| u(t^{'}) \|^{p}_{L^{p}}, \, p > 2
\end{array}
\right. \\
& \lesssim 1,
\end{array}
\label{Eqn:BoundFu}
\end{equation}
that is derived from Remark \ref{rem:inhomhom}. This proves that there exists $u_{+} \in H^{2}$ such that
$e^{-it \bilap} u(t) \rightharpoonup_{t \rightarrow \infty} u_{+} $. We define $v(t):= u(t) - e^{it \bilap} u_{+}$.

\end{proof}

\section{Asymptotic frequency localization}
\label{Sec:FreqLoc}

In this section we prove the asymptotic frequency localization of the nonradiative part of the solution.

\begin{prop}{\textbf{"Asymptotic frequency localization"}}
The non radiative part $v$ of the solution $u$ is asymptotically localized in the frequency domain.

\label{Prop:FreqLocal}
\end{prop}

\begin{proof}
It is enough to prove that, given $\epsilon > 0$, there exists $t_{\epsilon}  \geq 0$ such that for $t \geq  t_{\epsilon}$

\begin{equation}
\begin{array}{ll}
\| P_{\geq N} v(t) \|_{H^{2}} & \lesssim N^{-\eta} + \epsilon
\end{array}
\label{Eqn:HighFreqBound}
\end{equation}
for $N \geq 1$ and

\begin{equation}
\begin{array}{ll}
\| P_{\leq N} v(t) \|_{H^{2}} & \lesssim N^{\eta} + \epsilon
\end{array}
\label{Eqn:LowFreqBound}
\end{equation}
for $N \leq 1$. Let

\begin{equation}
Q_{N}:=
\left\{
\begin{array}{ll}
P_{\leq N} & \, N \leq 1 \\
P_{\geq N} & \, N \geq 1
\end{array}
\right.
\nonumber
\end{equation}
We choose  $u_{\epsilon} \in \mathcal{S}(\mathbb{R}^{n})$ such that

\begin{equation}
\begin{array}{ll}
u(0)-u_{+} & = u_{\epsilon} + O_{H^{2}}(\epsilon^{2}).
\end{array}
\nonumber
\end{equation}
By (\ref{Eqn:vt0}) and (\ref{Eqn:vtinf})

\begin{equation}
\begin{array}{ll}
i \int_{t}^{T} e^{i(t-t^{'}) \bilap} Q_{N} F(u(t^{'})) \, dt^{'} \rightharpoonup_{T \rightarrow \infty} Q_{N}v(t),
\end{array}
\label{Eqn:FirstRepn}
\end{equation}

\begin{equation}
\begin{array}{ll}
Q_{N}v(t) &  = e^{i t \bilap} Q_{N} u_{\epsilon} - i \int_{0}^{t} e^{i(t-t^{''}) \bilap} Q_{N} F(u(t^{''})) \, dt^{''} + O_{H^{2}} (\epsilon),
\end{array}
\label{Eqn:SecondRepn}
\end{equation}
Hence we see that for $T$ large enough

\begin{equation}
\begin{array}{ll}
\| Q_{N} v(t) \|^{2}_{H^{2}} & \lesssim |Z_{1}| + |Z_{2}| + O_{H^{2}}(\epsilon^{2})
\end{array}
\nonumber
\end{equation}
with

\begin{equation}
\begin{array}{ll}
Z_{1}:= & \langle \int_{t}^{T} e^{i(t-t^{'}) \bilap} Q_{N} F(u(t^{'})) \, dt^{'}, \, Q_{N} e^{it \bilap} u_{\epsilon} \rangle_{H^{2}}, \; \text{and} \\
Z_{2}:= & \langle \int_{t}^{T} \int_{0}^{t}  e^{i(t-t^{'}) \bilap} Q_{N} F(u(t^{'})), \ e^{i(t-t^{''}) \bilap} Q_{N} F(u(t^{''}))  \, dt^{''} \, dt^{'}
\rangle_{H^{2}} \cdot
\end{array}
\end{equation}

We first deal with $Z_{1}$. Using (\ref{Eqn:BoundFu}) and interpolating between (\ref{Eqn:Dispers}) and the $L^{2}$ estimate (see previous section), we see that $Z_{1}$ is the sum of terms of the form $X_{1,m}$ ($ 0 \leq m \leq 2$) with

\begin{equation}
\begin{array}{ll}
X_{1,m} & :=  \int_{t}^{T} \langle Q_{N} F(u(t^{'})), \, D^{2m} e^{i t^{'} \bilap} Q_{N} u_{\epsilon}  \rangle_{L^{2}} \, dt^{'} \\
& \lesssim
\int_{t}^{T} \frac{ \| F(u(t^{'})) \|_{L^{\tilde{p}}} } { \left( t^{'} \right)^{ \frac{n+2m}{4} \left( 1 - \frac{2}{\tilde{p}^{'}}  \right) }} \, dt^{'}, \\
& \lesssim \epsilon^{2}
\end{array}
\nonumber
\end{equation}
for $t$ large enough. \\
We then deal with $Z_2$. First we list some estimates. Recall the localized
dispersive estimate \cite{pausaderdef1} (for $K \in 2^{\mathbb{Z}}$) \footnote{
The author would like to thank Benoit Pausader  for suggesting him to use
(\ref{Eqn:DispImpr}). This argument allowed him to prove Theorem \ref{Thm:WeakSoliton}
for $n \geq 5$ (in the previous version Theorem \ref{Thm:WeakSoliton} was proved for $n>8$.)}

\begin{equation}
\begin{array}{ll}
\| e^{i t \bilap} P_{K} f  \|_{L^{\infty}} & \lesssim \frac{K^{n}}{ ( K^{4} t )^{\frac{n}{2}}} \| f \|_{L^{1}} \cdot
\end{array}
\label{Eqn:DispImpr}
\end{equation}
Recall the localized Strichartz estimate, proved in \cite{pausaderdef1}

\begin{equation}
\begin{array}{ll}
\left\| \int  P_{K} e^{i (t-s) \bilap} h(s) \, ds \right\|_{L_{t}^{2} L_{x}^{\frac{2n}{n-2}}} & \lesssim K^{-2} \| P_{K} h \|_{L_{t}^{2} L_{x}^{\frac{2n}{n+2}}} \cdot
\end{array}
\label{Eqn:Paus}
\end{equation}
In view of Remark \ref{rem:inhomhom} we have

\begin{equation}
\begin{array}{ll}
\| \nabla F(u(t')) \|_{L^{\tilde{p}}}  & \lesssim \left\{
\begin{array}{l}
\| u(t') \|^{p-1}_{L^{2}} \| \nabla u(t')  \|_{L^{2}}, \, p \leq 2 \\
\| u(t') \|^{p-1}_{L^{2(p-1)}}  \| \nabla u(t')  \|_{L^{2}}, \, p > 2
\end{array}
\right. \\
& \lesssim 1
\end{array}
\nonumber
\end{equation}
We are now in position to estimate $Z_{2}$. \\
First we assume that $N \leq 1$. Integrating by part we see that $Z_{2}$ is the sum of terms of the form $X_{2,m}$ with

\begin{equation}
\begin{array}{ll}
X_{2,m} & := \int_{t}^{T} \int_{0}^{t} \langle D^{m} P_{\leq N}  F(u(t^{'})), \, D^{m} P_{\leq N}  e^{i (t^{'} - t^{''}) \bilap}  F(u(t^{''})) \rangle_{L^{2}}
\, dt^{''} \, dt^{'}.
\end{array}
\nonumber
\end{equation}
Performing a Paley-Littlewood decomposition we see that  $ X_2 = \sum_{K \leq N} \int_{t}^{T} \int_{0}^{t} X_{2,K,m} \; d t' \; d t'' $ with

\EQQARR{
X_{2,K,m} := \langle D^{m} P_{K} F(u(t^{'})), \, D^{m}  e^{i (t^{'} - t^{''})  \bilap } P_{K} F(u(t^{''})) \rangle_{L^{2}}
}
One one hand, we get from (\ref{Eqn:BoundFu})

\begin{equation}
\begin{array}{ll}
|X_{2,K,m}| & \lesssim \| F(u(t^{'})) \|_{L^{\tilde{p}}}
\|  e^{i (t^{'} - t^{''})  \bilap } P_{K} F(u(t^{''}))  \|_{L^{\tilde{p}^{'}}} \\
& \lesssim \left(  \frac{1}{ K^{n} |t^{'} - t^{''}|^{\frac{n}{2}}} \right)^{ 1 -\frac{2}{\tilde{p}^{'}} }
\end{array}
\nonumber
\end{equation}
On the other hand (by Bernstein)

\begin{equation}
\begin{array}{ll}
| X_{2,K,m} |  & \lesssim
\| P _{K} F(u(t^{'})) \|_{L^{2}} \| P_{K} F(u(t^{''})) \|_{L^{2}} \\
& \lesssim K^{ 2 n \left( \frac{1}{_{\tilde{p}}} - \frac{1}{_{2}}  \right)}  \| F(u(t^{'})) \|_{L^{\tilde{p}}} \| F(u(t^{''})) \|_{L^{\tilde{p}}} \\
& \lesssim   K^{ 2 n \left( \frac{1}{_{\tilde{p}}} - \frac{1}{_{2}}  \right)}.
\end{array}
\nonumber
\end{equation}
Therefore, since $n \geq 5$, we conclude that

\begin{equation}
\begin{array}{ll}
Z_{2} & \lesssim \sum_{K \leq N}   \int_{t}^{T} \int_{0}^{t}
\min \left( \left(  \frac{1}{ K^{n }|t' - t''|^{\frac{n}{2}}}  \right)^{ 1 -\frac{2}{\tilde{p}^{'}}}, K^{ 2 n \left( \frac{1}{_{\tilde{p}}} - \frac{1}{_{2}}  \right)}  \right) \, dt' \, dt'', \\
& \lesssim N^{2 \eta},
\end{array}
\end{equation}
for some $\eta>0$.

Next we assume that $N \geq 1$. High frequencies are more complicated to deal with: indeed $F \in C^{1}$ (so we can only expect
to control norms involving the gradient of $F(u)$) whereas the the expressions involved lie in $H^{2}$. The idea here is again to
perform a Paley-Littlewood decomposition. Since the dispersive estimates for the high frequency Paley-Littlewood
pieces $P_{K}$ ($K \geq N$) are better as $K$ goes to infinity, we use them on a larger portion of the area of interaction between
(\ref{Eqn:FirstRepn}) and (\ref{Eqn:SecondRepn}) and, since the size of the remaining part of the interaction is small one can exploit it by using basic inequalities such as H\"older-in-time. Integrating by part and performing a Paley-Littlewood decomposition we see that $ Z_{2}  = \sum_{K \geq N} Z_{2,1,K} +  Z_{2,2,K} $ with

\begin{equation}
\begin{array}{ll}
Z_{2,1,K} & := \int_{t}^{T} \int_{0}^{t} \mathbf{1}_{|t^{'} - t^{''}| \geq K^{-\alpha}}
\langle \tilde{P}_{K} F(u(t^{'})), \, e^{i(t^{'} - t^{''}) \bilap} \tilde{P}_{K} F(u(t^{''})) \rangle_{H^{2}} \, dt^{''} \, dt^{'},
\end{array}
\nonumber
\end{equation}

\begin{equation}
\begin{array}{ll}
Z_{2,2,K} & := \int_{t}^{T} \int_{0}^{t} \mathbf{1}_{|t^{'} - t^{''}| < K^{-\alpha}}
\langle \tilde{P}_{K} F(u(t^{'})), \, e^{i(t^{'} - t^{''}) \bilap} \tilde{P}_{K} F(u(t^{''})) \rangle_{H^{2}} \, dt^{''} \, dt^{'},
\end{array}
\nonumber
\end{equation}
$ 0 < \alpha < \frac{n \left( 1 - \frac{2}{\tilde{p}^{'}} \right) -2 }{\frac{n}{2} \left( 1 - \frac{2}{\tilde{p}^{'}} \right) -2 }$, and
$\tilde{P}_{K}$ a frequency localized operator at $|\xi| \sim K$ (like $P_{K}$). $Z_{2,1,K}$ can be written as a sum of terms $X_{m}$ of the form ($m \leq 2$)

\begin{equation}
\begin{array}{ll}
X_{1,m}  & := K^{m} \int_{t}^{T} \int_{0}^{t} \mathbf{1}_{|t^{'} - t^{''}| \geq K^{-\alpha}}
\langle \nabla \tilde{P}_{K} F(u(t^{'})), \, e^{i(t^{'} - t^{''}) \bilap} \nabla \tilde{P}_{K} F(u(t^{''})) \rangle_{L^{2}} \, dt^{''} \, dt^{'}.
\end{array}
\nonumber
\end{equation}
Interpolating between (\ref{Eqn:DispImpr}) and the $L^{2}$ estimate

\begin{equation}
\begin{array}{ll}
|X_{1,m}| & \lesssim K^{2} \int_{t}^{T} \int_{0}^{t} \mathbf{1}_{|t^{'} - t^{''}| \geq K^{-\alpha}} \| \nabla F(u(t^{'})) \|_{L^{\tilde{p}}}
\| e^{i(t^{'} - t^{''}) \bilap} \nabla \tilde{P}_{K} F(u(t^{''})) \|_{L^{\tilde{p}^{'}}} \, dt^{''} \, dt^{'} \\
& \lesssim_{M} \int_{t}^{T} \int_{0}^{t}  K^{2} \mathbf{1}_{|t^{'} - t^{''}| \geq K^{-\alpha}}
\left( \frac{1} { |t' -  t^{''}|^{\frac{n}{2}} K^{n}} \right)^{1 - \frac{2}{\tilde{p}^{'}}} \, dt^{''} \, dt^{'} \\
& \lesssim_{M} K^{-2\eta},
\end{array}
\nonumber
\end{equation}
for some $\eta >0$, since $n \geq 5$. $Z_{2,2,K}$ can be written as a sum of terms $X_{2,m}$ ($m \leq 2$)

\begin{equation}
\begin{array}{ll}
X_{2,m} & := K^{m} \int_{t}^{t+K^{-\alpha}} \int_{t- K^{-\alpha}}^{t} \mathbf{1}_{|t^{'}-t^{''}| < K^{- \alpha}} \nabla \tilde{P}_{K} F(u(t^{'})) \nabla e^{i(t^{'} - t^{''}) \bilap} \tilde{P}_{K} F(u(t^{''}))
\, dt^{''} \, dt^{'}
\end{array}
\end{equation}
We have, by (\ref{Eqn:Paus}) and Proposition \ref{Prop:LocalEst} (and its proof)

\begin{equation}
\begin{array}{ll}
| X_{2,m} | & \lesssim  \| \nabla \tilde{P}_{K} F(u) \|_{L_{t}^{2} L_{x}^{\frac{2n}{n+2}} ([t,t+K^{-\alpha}]) }
\| \nabla \tilde{P}_{K} F(u) \|_{L_{t}^{2} L_{x}^{\frac{2n}{n+2}} ([t-K^{-\alpha},t]) } \\
& \lesssim K^{-2 \eta} \left( \| u \|^{p-1}_{L_{t}^{\infty} L_{x}^{Q} ([t,t+K^{-\alpha}])}
\| \nabla u \|_{L_{t}^{q_{0}} L_{x}^{\tilde{r}_{0}} ([t-K^{-\alpha},t])} + \| u \|^{p-1}_{L_{t}^{\infty} L_{x}^{Q} ([t,t+K^{-\alpha}])}
\| \nabla u   \|_{L_{t}^{q_{0}} L_{x}^{r_{0}} ([t,t+K^{-\alpha}]) } \right) \\
& \lesssim K^{-2 \eta},
\end{array}
\nonumber
\end{equation}
for some $\eta > 0$.

\end{proof}

\section{Asymptotic spatial localization }
\label{Sec:SpatLoc}

In this section we prove the asymptotic spatial localization. The asymptotic spatial localization relies upon the asymptotic partial spatial localization
and the asymptotic final spatial localization.

\subsection{Asymptotic partial spatial localization}

In this subsection we prove the asymptotic partial spatial localization property of the nonradiative part. First we define some constants and we set up
the framework.

Let $1 \gg \mu_{1}$ be a fixed constant and let  $ \mu_{1} \gg \mu_{2} \gg \mu_{3}
\gg  \mu_{4} \gg \mu_{5} \gg \mu_{6} >0$ be constants depending on $\mu_{1}$
that are chosen such that all the inequalities in this section are true. Let also $c$, $C$ denote a small, large constant whose value can change
from one line to the other one.

We use the decomposition (\ref{Eqn:Decomp}), (\ref{Eqn:HighFreqBound}), and (\ref{Eqn:LowFreqBound}) to get (with $N \gtrsim 1$)

\begin{equation}
\begin{array}{ll}
u(t) & = e^{it \bilap} u_{+} + v_{N}(t) + O_{H^{2}}(N^{-\eta}),
\end{array}
\nonumber
\end{equation}
with $v_{N}(t):= P_{ \frac{1}{N}  \leq  \cdot \leq N} v(t)$. Let
$ A:= \{ x \in \mathbb{R}^{n}, \, |v_{N}(t,x)| \geq \mu^{c}_{3}  \} $. Given $t >0$,  we
construct inductively a sequence of points $\{ x_{j}(t,\mu_{3}) \}_{1 \leq j \leq \bar{J}}$ in the following fashion:

\begin{enumerate}

\item Initially let $X:=[\emptyset]$ and $j=1$
\item If $A \neq \emptyset$ , then choose $y \in A$, let $x_{1}(t,\mu_{3}):=y$ and $X:=[x_{1}(t,\mu_{3})]$;
\item While $A \cap B_{j} \neq \emptyset $ with
$B_{j}:= \cap_{ 1 \leq k \leq j} \{ x \in \mathbb{R}^{n}, \, |x-x_{k}(t,\mu_{3})| \geq \frac{1}{2 \mu_{3}} \} $ do the following

\begin{itemize}
\item  choose $y \in A \cap B_{j}$
\item let $j:=j+1$
\item let $x_{j}(t,\mu_{3}):=y$ and let $X:=[X,x_{j}(t,\mu_{3})]$
\end{itemize}

\end{enumerate}
This construction is finite: let us see why. We see that for $ 1 \leq j \leq \card{(X)}$ $v_{ N}(t,x_{j}(t,\mu_{3}))$
is a sum of two elements of the form

\begin{equation}
\tilde{v}_{M}(t,x_{j}(t,\mu_{3})) := M^{n} \int \tilde{\phi}(M(x_{j}(t,\mu_{3})-y)) v(t,y) \, dy,
\nonumber
\end{equation}
with $\bar{N} \in \left\{ \frac{1}{N},N \right\}$ and $\widehat{\tilde{\phi}}$ a localized bump around  $|\xi| \lesssim 1$. Using the fast
decay of $\check{\phi}$ on the region $|\bar{N}(x_{j}(t,\mu_{3})-y)| \geq R$ we see that for $\gamma$ arbitrarily large

\begin{equation}
\begin{array}{ll}
\left| \tilde{v}_{\bar{N}}(x_{j} \left( t,\mu_{3}),t \right) \right| & \lesssim \bar{N}^{\frac{n}{2}}
R^{\frac{n}{2}} \int_{|y-x_{j}(t,\mu_{3})| \leq \frac{R}{\bar{N}}} |v(t,y)|^{2} \, dy + \frac{\bar{N}^{\frac{n}{2}}}{R^{\gamma}}
\end{array}
\label{Eqn:ExpProj}
\end{equation}
Letting $ N \sim \mu_{2}^{-1}$ and $R \sim \mu_{3}^{-c}$, we see that

\begin{equation}
\begin{array}{ll}
\int_{|x-x_{j}(t,\mu_{3})| \leq \frac{1}{4 \mu_{3}}} |v(t,x)|^{2} \, dx & \gtrsim \mu_{3}^{C} \cdot
\end{array}
\end{equation}
By (\ref{Eqn:BounduplusvH2}), we see that $\card{(X)} \leq J := J(\mu_{3}) \leq \mu_{3}^{-C}$. In order to make the cardinal constant we set for $j$, $ \card{(J)} < j \leq J(\mu_{3})$, $x_{j}(t,\mu_{3}):=x_{\tilde{J}}(t,\mu_{3})$.
By  maximality we have

\begin{equation}
\begin{array}{l}
|v_{N}(t,x)| < \mu^{c}_{3}, \, if \, \,  inf_{1 \leq j \leq J} |x-x_{j}(t,\mu_{3})| \geq \frac{1}{2 \mu_{3}}
\end{array}
\label{Eqn:BoundLinfty}
\end{equation}
Notice that, since $\mu_{3}$ is a function of $\mu_{1}$ we write $x_{j}(t,\mu_{1})$, $J(\mu_{1})$ instead of $x_{j}(t,\mu_{3})$, $J(\mu_{3})$
respectively in the sequel. We are now in position to state the asymptotic partial spatial localization property:

\begin{prop}{\textbf{" Asymptotic partial spatial localization"}}
We have

\begin{equation}
\begin{array}{ll}
\overline{\lim}_{t \rightarrow \infty} \int_{\inf_{1 \leq j \leq J} | x- x_{j}(t,\mu_{1})| \geq  \mu^{-1}_{6}} |v(t,x)|^{2} \, dx & \lesssim
\mu^{c}_{1} \cdot
\end{array}
\end{equation}
\label{prop:spat}
\end{prop}

\begin{rem}
Notice that the number of points where the mass concentrates depends on the parameter $\mu_{1}$. This is not totally consistent with the soliton
resolution conjecture, since, as $t \rightarrow \infty$, the number of solitons (i.e the spots where the mass concentrates) should not depend on $\mu_1$.
This is why we have called this proposition asymptotic ``partial'' spatial localization. Notice also that the points depend on $\mu_{1}$.
These constraints will be removed in Proposition \ref{prop:finalspat}.
\end{rem}

Let $t_{0}$ be large enough such that all the inequalities in this subsection are true. In the sequel, in order to avoid too much notation, we
forget $t_{0}$ and $\mu_{1}$ to set $D(x):=\inf_{1 \leq  j \leq J} |x-x_{j}(t_{0},\mu_{1})|$ and $x_{j}:=x_{j}(t_{0},\mu_{1})$.
Let $I=[t_{0} -\mu_{1}^{-1}, t_{0}+ \mu_{1}^{-1}]$. \\
Given $\mu>0$, let $\chi_{\mu}$ be a smooth function such that $\chi_{\mu}(x)=1$ if $D \leq \mu^{-1}$, $\chi_{\mu}(x)=0$ if $D \geq 2 \mu^{-1}$
and satisfies $| \partial^{\alpha} \chi_{\mu}(x)| \lesssim_{k} \mu^{k}$ if $|\alpha|=k$. This function can be obtained, for example, by convolution
of the characteristic function on $D \leq 1.1 \mu^{-1}$ with an approximate of the identity of size $0.1 \mu^{-1}$. \\
We first prove that the linear flow away from the points of concentration is small as time goes to infinity:

\begin{lem}
We have

\begin{equation}
\begin{array}{ll}
\overline{\lim}_{t_{0} \rightarrow \infty}  \| e^{i(t-t_{0}) \bilap} (1- \chi_{ \frac{\mu_{4}}{4}} ) u(t_{0}) \|_{X(I)} & \lesssim  \mu_{2}^{c}
\end{array}
\end{equation}
\label{Lem:LinearPart}
\end{lem}

\begin{proof}
We use the decomposition (\ref{Eqn:Decomp}) at time $t_{0}$. By density of the Schwarz functions in $H^{2}$ we can find $u_{\mu_{2}} \in \mathcal{S}(\mathbb{R}^{n})$ such that

\begin{equation}
\begin{array}{ll}
u_{+} & = u_{\mu_{2}} +  O_{H^{2}}( \mu_{2}^{c}) \cdot
\end{array}
\nonumber
\end{equation}
Combining (\ref{Eqn:Strich}), (\ref{Eqn:StrichGain}) with Proposition \ref{Prop:FreqLocal} we see that

\begin{equation}
\begin{array}{ll}
e^{i(t-t_{0}) \bilap}(1- \chi_{ \frac{\mu_{4}}{4}} )u(t_{0}) & =  e^{i(t-t_{0}) \bilap} (1-\chi_{ \frac{\mu_{4}}{4}}) \; e^{i t_{0} \bilap} u_{\mu_{2}} \\
& \\
& +  e^{i(t-t_{0}) \bilap} (1- \chi_{\frac{\mu_{4}}{4}}  ) \; P_{\leq  10 \mu_{2}^{-1}}  v_{\mu_{2}^{-1}}(t_{0}) + O_{X(I)}(\mu_{2}^{c}) \\
& \\
& =   e^{i(t-t_{0}) \bilap} (1-\chi_{ \frac{\mu_{4}}{4}} ) \; e^{i t_{0} \bilap} u_{\mu_{2}} \\
& \\
& +  e^{i(t-t_{0}) \bilap} (1- \chi_{ \frac{\mu_{4}}{4}} )  \; P_{ \leq  10 \mu_{2}^{-1}} \left(
\chi_{ \mu_{4}} \; v_{\mu_{2}^{-1}}(t_{0}) \right) \\
& \\
&  +   e ^{i(t-t_{0}) \bilap} ( 1 - \chi_{ \frac{\mu_{4}}{4}} )  \; P_{ \leq 10 \mu_{2}^{-1}} \left( (1 - \chi_{\mu_{4}}) \;
v_{\mu_{2}^{-1}}(t_{0}) \right) +
O_{X(I)} (\mu_{2}^{c}) \\
\end{array}
\nonumber
\end{equation}
Next we show the following results:

\underline{Result 1}: We have

\begin{equation}
\begin{array}{ll}
\| e^{i(t-t_{0}) \bilap} (1- \chi_{\frac {\mu_{4}}{4}}) e^{i t_{0} \bilap} u_{\mu_{2}} \|_{X(I)} & \lesssim \mu_{2}
\end{array}
\end{equation}

\begin{proof}
From (\ref{Eqn:Dispers}) we see that it is enough to prove that

\begin{equation}
\begin{array}{ll}
\| e^{i(t-t_{0}) \bilap} \chi_{\frac{\mu_{4}}{4}} e^{i t_{0} \bilap} u_{\mu_{2}} \|_{X(I)} & \lesssim \mu_{2}
\end{array}
\label{Eqn:WeightEst}
\end{equation}
Let $g_{t_{0}}(x) := e^{i t_{0} \bilap} u_{\mu_{2}}(x)$. We write

\begin{equation}
\begin{array}{ll}
e^{i(t-t_{0}) \bilap} \chi_{\frac{\mu_{4}}{4}} g_{t_{0}} & =  e^{i(t-t_{0}) \bilap} P_{\leq  1} \chi_{ \frac{\mu_{4}}{4}} \; g_{t_{0}} +
\sum_{N > 1} e^{i(t-t_{0}) \bilap} \tilde{P}_{N} P_{N} \chi_{ \frac{\mu_{4}}{4}} \; g_{t_{0}}
\end{array}
\end{equation}
with $\tilde{P}_{N}  := P_{>2N} - P_{ \leq \frac{N}{2}} $. Let $\tilde{\phi}$ be such that
$\widehat{\tilde{P}_{N} f}(\xi) := \tilde{\phi} \left( \frac{\xi}{N} \right) \hat{f}(\xi) $.
Let $\bar{p}=r_{0},Q$. We have

\begin{itemize}

\item

\begin{equation}
\begin{array}{ll}
\| e^{i(t-t_{0}) \bilap} \tilde{P}_{N} f \|_{L^{\bar{p}}} & \lesssim N^{4n} |I|^{n} \| f \|_{L^{\bar{p}}}
\end{array}
\label{Eqn:BoundFundN}
\end{equation}
Indeed, the kernel $\tilde{K}_{N}$ of $e^{i(t-t_{0}) \bilap} \tilde{P}_{N}$ is

\begin{equation}
\begin{array}{ll}
\tilde{K}_{N}(x,y) & = N^{n} \int e^{i(t-t_{0}) N^{4} |\xi|^{4}} e^{iN(x-y) \cdot \xi} \tilde{\psi}(\xi) \, d \xi
\end{array}
\nonumber
\end{equation}
Stationary phase \cite{stein} yields

\begin{equation}
\begin{array}{ll}
\tilde{K}_{N}(x,y) & = \left\{
\begin{array}{l}
O(N^{n}), \, |x-y| \lesssim N^{3} |I| \\
O \left( \frac{N^{n}}{\langle N(x-y)\rangle^{n+1}}   \right), \, |x-y| \gg N^{3} |I|,
\end{array}
\right.
\end{array}
\nonumber
\end{equation}
The conclusion follows from Schur's lemma.

\item  $  \| P_{N} \chi_{ \frac{\mu_{4}}{4}} \langle D \rangle^{-100n} f  \|_{L^{p}} \lesssim \frac{(\mu_{3})^{-C}}{N^{C}} \|f \|_{L^{p}}$.
Indeed the kernel $K_{N}$ of this operator is

\begin{equation}
\begin{array}{ll}
K_{N}(x,y) & = \int \int \int \psi \left( \frac{\xi}{N} \right)

 \frac{ \chi_{\frac{ \mu_{4}}{4}} (z)} {\langle \eta \rangle^{100n}}
e^{i \left( \xi \cdot x - \eta \cdot y - (\xi - \eta) \cdot z \right)} \, d \eta \, d \xi \, dz
\end{array}
\nonumber
\end{equation}
We see from integration by part w.r.t  $\xi$ and $\eta$ of the phase that

\begin{equation}
\begin{array}{ll}
|K_{N}(x,y)| & \lesssim \frac{(\mu_{3})^{-C} N^{n} }{ |x - z|^{n+1}}, \; \text{and} \\
|K_{N}(x,y)| & \lesssim \frac{(\mu_{3})^{-C} N^{n} }{ |y - z|^{n+1}} \; \cdot
\end{array}
\nonumber
\end{equation}
We also have

\begin{equation}
\begin{array}{ll}
|K_{N}(x,y)| & \lesssim \frac{(\mu_{3})^{-C}}{N^{90n}} \cdot
\end{array}
\nonumber
\end{equation}
Indeed if $|\eta| \lesssim N$ then this follows by integration by parts of the phase w.r.t  $z$; if not we
bound pointwise $K_{N}$. The conclusion follows from Schur's lemma.

\item A straightforward modification of the proof of (\ref{Eqn:BoundFundN}) shows that \\
$\| e^{i(t-t_{0}) \bilap} P_{\leq 1} f \|_{L^{p}} \lesssim |I|^{n} \| f \|_{L^{p}}$.

\end{itemize}
Now, using these these operator norm bounds, summing over $N$, using the dispersive bound
$ \| \langle D \rangle^{100n} e^{it_{0} \bilap} u_{\mu_{2}} \|_{L^{p}}  \lesssim \frac{1}{t_{0}^{\frac{n}{4}
\left( 1 - \frac{2}{p} \right) }}
\| \langle D \rangle^{100n} u_{\mu_{2}} \|_{L^{p^{'}}}$ and H\"older in time, we see that (\ref{Eqn:WeightEst}) holds.

\end{proof}

\underline{Result 2}: We have

\begin{equation}
\begin{array}{ll}
\left\| e^{i(t-t_{0}) \bilap} (1- \chi_{ \frac{\mu_{4}}{4}})  P_{ \leq  10 \mu_{2}^{-1}} \left( ( 1- \chi_{ \mu_{4}}) v_{\mu_{2}^{-1}}(t_{0}) \right) \right\|_{X(I)} & \lesssim \mu_{2}^{c}
\end{array}
\nonumber
\end{equation}

\begin{proof}
It is enough to prove that for $\bar{p}=Q,r_{0}$ and for $l=0,1,2$
\begin{equation}
\begin{array}{ll}
\| D^{l} e^{i(t-t_{0}) \bilap} (1-\chi_{\frac{ \mu_{4}}{4}}) P_{ \leq 10 \mu_{2}^{-1} }  f \|_{L^{\bar{p}}} & \lesssim \mu_{2}^{-C} |I|^{n}
\| f \|_{L^{\bar{p}}}
\end{array}
\label{Eqn:Est}
\end{equation}
Indeed it is enough to combine (\ref{Eqn:Est})  with the interpolation inequality (for some $ 0< \theta <1$)

\begin{equation}
\begin{array}{ll}
\| (1- \chi_{\mu_{4}}) v_{\mu_{2}^{-1}}(t_{0}) \|_{L^{p}} & \lesssim \| (1- \chi_{\mu_{4}})  v_{\mu_{2}^{-1}}(t_{0}) \|^{\theta}_{L^{2}}
\| (1- \chi_{\mu_{4}})  v_{\mu_{2}^{-1}}(t_{0}) \|^{1-\theta}_{L^{\infty}} \\
& \lesssim \mu^{c}_{3},
\end{array}
\end{equation}
the last inequality following from (\ref{Eqn:BounduplusvH2}) and (\ref{Eqn:BoundLinfty}). Using the triangle inequality
we have to estimate two terms. We shall only prove the following estimate

\begin{equation}
\begin{array}{ll}
\| D^{l} e^{i(t-t_{0}) \bilap} \chi_{\mu_4} P_{\leq 10 \mu_{2}^{-1} }  f \|_{L^{\bar{p}}} & \lesssim \mu_{2}^{-C} |I|^{n} \| f \|_{L^{\bar{p}}},
\end{array}
\label{Eqn:Toprove}
\end{equation}
since the other estimate is easier to prove (and therefore left to the reader). We write

\begin{equation}
\begin{array}{ll}
D^{l} e^{i(t-t_{0}) \bilap} \chi_{\mu_4} P_{\leq 10 \mu_{2}^{-1}}  f  & =
D^{l} e^{i(t-t_{0}) \bilap} \;  \tilde{P}_{\leq 640 \mu_{2}^{-1}}
\chi_{\mu_4} \; P_{\leq  10 \mu_{2}^{-1} } f  \\
& \\
& +   \sum_{N>1,N \in 2^{\mathbb{N}}} D^{l} e^{i(t-t_{0}) \bilap} \; \tilde{P}_{640 N \mu_{2}^{-1}} \; P_{640 N \mu_{2}^{-1}}
\; \chi_{\mu_4} \; P_{\leq 10 \mu_{2}^{-1}} f \cdot
\end{array}
\nonumber
\end{equation}
A straightforward modification of the proof of (\ref{Eqn:BoundFundN}) shows that

\begin{equation}
\begin{array}{ll}
\| D^{l} e^{i(t-t_{0}) \triangle^{2}} \tilde{P}_{640 N \mu_{2}^{-1}} f \|_{L^{p}} & \lesssim N^{l+4n} \mu_{2}^{-(4n+l)} |I|^{n} \| f \|_{L^{p}} \\
\| D^{l} e^{i(t-t_{0}) \triangle^{2}} \tilde{P}_{\leq 640 \mu_{2}^{-1}} f \|_{L^{p}} & \lesssim \mu_{2}^{-(4n+l)} |I|^{n} \| f \|_{L^{p}}
\end{array}
\nonumber
\end{equation}
Therefore it remains to show that

\begin{equation}
\begin{array}{ll}
\| P_{640 N \mu_{2}^{-1}}  \chi_{\mu_4} P_{\leq  10 \mu_{2}^{-1}} f \|_{L^{p}} & \lesssim \frac{\mu_{2}^{C}}{N^{C}} \| f \|_{L^{p}} \cdot
\end{array}
\nonumber
\end{equation}
The kernel $K_{N}$ of $ P_{640 N \mu_{2}^{-1}} \chi_{\mu_4} P_{\leq  10 \mu_{2}^{-1}} $ is

\begin{equation}
\begin{array}{ll}
K_{N}(x,y) & = \int \int \int \psi \left( \frac{\xi}{ 640 N \mu_{2}^{-1}} \right) \chi_{\mu_4}(z)
\phi \left( \frac{\eta}{ 10 \mu_{2}^{-1}}  \right) e^{i(\xi \cdot x - \eta \cdot y - (\xi- \eta) \cdot z )} \, d \xi \, d \eta \, dz \cdot
\end{array}
\nonumber
\end{equation}
By integrating by parts the phase w.r.t  $\xi$, $\eta$ and $z$ we see that

\begin{equation}
\begin{array}{ll}
|K_{N}(x,y)| & \lesssim \frac{1}{ (N \mu_{2}^{-1})^{100n} \langle x - z \rangle^{100n} \langle y - z  \rangle^{100n}} \cdot
\end{array}
\nonumber
\end{equation}

The conclusion follows from the application of Schur's lemma. \\

\end{proof}

\underline{Result 3}: We have

\begin{equation}
\begin{array}{ll}
\|  e ^{i(t-t_{0}) \bilap} ( 1 - \chi_{\frac{\mu_4}{4}} ) P_{ \leq 10 \mu_{2}^{-1}} ( \chi_{\mu_4} \;
v_{\mu_{2}^{-1}}(t_{0})) \|_{X(I)} & \lesssim \mu_{2}^{c}
\end{array}
\nonumber
\end{equation}

\begin{proof}
We see  from (\ref{Eqn:Strich}) that it is enough to prove that, for $m \in \{0,1,2 \}$, we have

\begin{equation}
\begin{array}{ll}
\left\| T_{m} \left[  \left( (1-\chi_{\frac{\mu_4}{4}} ) P_{\leq 10 \mu_{2}^{-1}} ( \chi_{\mu_{4}} \;  v_{\mu_{2}^{-1}}(t_{0}) ) \right) \right]  \right\|_{L^{2}} & \lesssim \mu_{2}^{c}
\end{array}
\label{Eqn:Ineqm}
\end{equation}
for all $t \in I$. Here $T_{0}:= Id$, $T_{1}:= \nabla$, and $T_{2} := \triangle$. We shall prove (\ref{Eqn:Ineqm}) for $m=0$, since the other cases ($m=1,2$) can be easily derived from this case, using the Leibnitz rule. By Minkowski's inequality and (\ref{Eqn:BounduplusvH2}), we have

\begin{equation}
\begin{array}{l}
\| (1-\chi_{ \frac{\mu_{4}}{4}}) P_{\leq 10 \mu_{2}^{-1}} ( \chi_{\mu_{4}} \;  v_{\mu_{2}^{-1}}(t_{0})) \|_{L^{2}} \\
\lesssim   \left\| \int \chi_{\mu_{4}}(x-y) \;
 v_{\mu_{2}^{-1}}(t_{0},x-y) \; \mu^{-n}_{2} \check{\phi} \left( 10 \mu_{2}^{-1} y  \right) \, dy  \right\| _{ L^{2}(D \geq  4 \mu_{4}^{-1} )} \\
\lesssim \| v_{\mu_{2}^{-1}} \|_{L^{2}} \mu^{-n}_{2} \left\| \check{\phi} \left( 10 \mu^{-1}_{2} y  \right) \right\|_{L^{1} (|y| \gtrsim \mu_{4}^{-1}) } \\
\lesssim \mu_{2}^{c}
\end{array}
\nonumber
\end{equation}

\end{proof}

\end{proof}

Next we prove the following lemma:

\begin{lem}
We have

\begin{equation}
\begin{array}{ll}
\overline{\lim}_{t_{0} \rightarrow \infty} \| \mathbf{1}_{D \geq \mu_{5}^{-1}} u \|_{L_{t}^{q_{0}} L_{x}^{r_{0}} (I \times \mathbb{R}^{n}) } & \lesssim \mu_{2}^{c}
\end{array}
\label{Eqn:DecayOutConc}
\end{equation}

\end{lem}

\begin{proof}

Let $\tilde{u}$ be the solution of (\ref{Eqn:BiSchrod}) with data $\tilde{u}(t_{0}):= \chi_{\frac{ \mu_{4}}{4}} u(t_{0})$. Then, by Lemma \ref{Lem:LinearPart}, Proposition \ref{prop:Perturb}, Remark \ref{rem:inhomhom}, (\ref{Eqn:SobH2}) and (\ref{Eqn:LocalEst}) we see that, if $t_{0}$ is large enough, then

\begin{equation}
\begin{array}{l}
\| u - \tilde{u} \|_{L_{t}^{\infty} L_{x}^{2}(I)} \lesssim_{|I|} 1, \\
\| u - \tilde{u} \|_{L_{t}^{q_{0}} W^{2,r_{0}}(I) \cap L_{t}^{\infty} L_{x}^{Q}(I) } \lesssim \mu^{c}_{2}, \\
\| \tilde{u} \|_{L_{t}^{q_{0}} W^{2,r_{0}} (I)}  \lesssim_{\mu_{1}} 1, \, and \\
\| \tilde{u} \|_{ L_{t}^{\infty} L_{x}^{Q} (I)}  \lesssim 1.
\end{array}
\label{Eqn:Boundtildeu}
\end{equation}
Let $\omega$ be equal to the convolution of an approximate of the identity of size $0.1 \mu_{4}$ and

\begin{equation}
\bar{\phi} \left( \frac{D}{0.9 \mu_{5}^{-1} -  (\mu_{4}^{-1} + 0.1 \mu_{5}^{-1})}
+ \frac{ 0.9 \mu_{5}^{-1} - 2 (\mu_{4}^{-1} + 0.1 \mu_{5}^{-1})}{ 0.9 \mu_{5}^{-1} - (\mu_{4}^{-1} + 0.1 \mu_{5}^{-1})} \right),
\nonumber
\end{equation}
(Here $\bar{\phi}: \mathbb{R}^{n} \rightarrow \mathbb{R}$ a smooth function such that
$\bar{\phi}(x)= \mu_{2}^{-c}$ if $|x| \geq 2$ and $\bar{\phi}(x) = 1$ if $|x| \leq 1$).
Observe that $\omega(x) = \mu_{2}^{-c}$ if $D \geq  \mu_{5}^{-1}$, $\omega(x) =1$ if $D \leq \mu_{4}^{-1}$,
$\| \omega \tilde{u}(t_{0}) \|_{L^{2}} \lesssim \| u(t_{0}) \|_{L^{2}} $.\\








A computation (using (\ref{Eqn:BiSchrod})) shows that

\begin{equation}
i \partial_{t} (\omega \tilde{u}) + \bilap (\omega \tilde{u}) =
\omega F(\tilde{u}) +  \bilap \omega \tilde{u} + 2 \nabla (\triangle w) \cdot \nabla \tilde{u}
+ 4 \nabla \cdot (\nabla w \triangle \tilde{u}) - 2 \triangle w \triangle \tilde{u}
+ 2 \triangle (\nabla w \cdot \nabla \tilde{u})
\nonumber
\end{equation}
Let $Y(J):=L_{t}^{q_{0}} L_{x}^{r_{0}} (J) \cap L_{t}^{\infty} L_{x}^{2} (J)$.  From (\ref{Eqn:Strich}) and (\ref{Eqn:StrichGain}) we see that
for any subinterval $ J=[a,b] \subset I$

\begin{equation}
\begin{array}{ll}
\| \omega \tilde{u} \|_{Y(J)} &  \lesssim  \| \omega \tilde{u}(a) \|_{L^{2}} +
\| \omega F(\tilde{u}) \|_{L_{t}^{2} L_{x}^{\frac{2n}{n+4}} (J)} +
\| \bilap \omega \tilde{u} \|_{L_{t}^{1} L_{x}^{2} (J)} +
\| \nabla (\triangle w) \cdot \nabla \tilde{u} \|_{L_{t}^{1}L_{x}^{2}(J)} \\
& + \| \triangle w  \triangle \tilde{u} \|_{L_{t}^{1} L_{x}^{2}(J)} +
\| \nabla (\nabla w \cdot \nabla \tilde{u}) \|_{L_{t}^{2} L_{x}^{\frac{2n}{n+2}}(J)} +
\| \nabla w \triangle \tilde{u} \|_{L_{t}^{2} L_{x}^{\frac{2n}{n+2}}(J)} \\
& \lesssim \| \omega \tilde{u}(a) \|_{L^{2}} + |J|^{c}  \| \omega \tilde{u} \|_{L_{t}^{q_{0}} L_{x}^{r_{0}} (J)}
\| \tilde{u} \|^{p-1}_{L_{t}^{\infty} L_{x}^{Q}(J)} + \mu_{5}^{c} \\
& \lesssim  \| \omega \tilde{u}(a) \|_{L^{2}} + |J|^{c} \| \omega \tilde{u} \|_{Y(J)} + \mu_{5}^{c}
\end{array}
\label{Eqn:Comput}
\end{equation}
Using (\ref{Eqn:Comput}) on time intervals of small size and iterating, we see that
$\| \omega \tilde{u} \|_{L_{t}^{q_{0}} L_{x}^{r_{0}} (I)} \lesssim_{\mu_{1}} 1$ and consequently  $ \| \mathbf{1}_{D \geq \mu_{5}^{-1}} \tilde{u} \|_{L_{t}^{q_{0}} L_{x}^{r_{0}} (I)}  \lesssim
\mu_{2}^{c} $. Hence (\ref{Eqn:DecayOutConc}) holds.

\end{proof}

This implies the decay of the $L^{2}$ norm of the inhomogeneous part of the solution far away from the $\bar{x}$ :

\begin{lem}
For every subinterval $I^{'} \subset I$, we have

\begin{equation}
\begin{array}{ll}
\| (1- \chi_{\mu_{5}^{2}}) \int_{I^{'}} e^{i(t_{0} - t^{'}) \bilap} F(u(t^{'})) \, dt^{'} \|_{L^{2}(\mathbb{R}^{n})} & \lesssim \mu^{c}_{2}
\end{array}
\end{equation}
\end{lem}

\begin{proof}
We write

\begin{equation}
\begin{array}{ll}
F(u(t^{'})) & = F(u(t^{'}) \chi_{\mu_{5}} + u(t^{'})(1- \chi_{\mu_{5}}) )  \\
& = F(u(t^{'})  \chi_{\mu_{5}} ) + O \left( \mathbf{1}_{D \geq   \mu_{5}^{-1}} |u(t^{'})|^{p} \right) \\
& = P_{\geq \mu_{2}^{-1}} F( u(t^{'})) \tilde{\chi}_{\mu_{5}} + P_{\leq \mu_{2}^{-1}} F( u(t^{'}))  \tilde{\chi}_{\mu_{5}}
+ O \left( \mathbf{1}_{ D \geq  \mu_{5}^{-1}} |u(t^{'})|^{p}  \right),
\end{array}
\nonumber
\end{equation}
with $\tilde{\chi}_{\mu_{5}}$ denoting a smooth function that behaves like $\chi_{\mu_{5}}$.\\
In order to control the high frequency term, we use the fact that we work in $H^{2}$ and so we can expect some gain. By (\ref{Eqn:Strich}), it
is enough to control the following term

\begin{equation}
\begin{array}{ll}
\| (1- \chi_{\mu_{5}^{2}}  ) P_{ \geq \mu^{-1}_{2}} F( u(t')) \tilde{\chi}_{\mu_{5}}  \|_{L_{t}^{2} L_{x}^{\frac{2n}{n+4}} (I')} & \lesssim  \mu_{2}
\| \nabla ( F(u(t^{'})) \tilde{\chi}_{\mu_{5}})  \|_{L_{t}^{2} L_{x}^{\frac{2n}{n+4}} (I') } \\
& \lesssim \mu_{2} \left( \| F(u(t^{'})) \|_{L_{t}^{2} L_{x}^{\frac{2n}{n+4}} (I')} +  \| \nabla F(u(t^{'})) \|_{L_{t}^{2} L_{x}^{\frac{2n}{n+4}} (I') } \right) \\
& \lesssim \mu_{2} \mu^{-\alpha}_{1} \| u \|^{p-1}_{L_{t}^{\infty} L_{x}^{Q} (I)}
( \| u \|_{L_{t}^{q_{0}} L_{x}^{r_{0}} (I^{'})} + \| \nabla u \|_{L_{t}^{q_{0}} L_{x}^{r_{0}} (I^{'}) } ) \\
& \lesssim \mu_{2}^{c},
\end{array}
\label{Eqn:L2Bound1}
\end{equation}
the last inequality following from (\ref{Eqn:SobH2}) (\ref{Eqn:LocalEst}), and Remark \ref{rem:inhomhom}.
In order to control the third term, we use (\ref{Eqn:DecayOutConc})

\begin{equation}
\begin{array}{ll}
\| \mathbf{1}_{D \geq  \mu_{5}^{-1}} |u(t^{'})|^{p}  \|_{L_{t}^{2} L_{x}^{\frac{2n}{n+4}} (I^{'})} & \lesssim \mu^{- C}_{1}  \| \mathbf{1}_{D \geq  \mu_{5}^{-1}} u(t^{'}) \|_{L_{t}^{q_{0}} L_{x}^{r_{0}} (I)}
\| u(t^{'}) \|^{p-1}_{L_{t}^{\infty} L_{x}^{Q} (I)} \\
& \lesssim \mu_{2}^{c}
\end{array}
\label{Eqn:L2Bound2}
\end{equation}
In order to control the second term we use the fact that the medium frequencies of the solution have (locally) an almost finite speed of propagation.
More precisely the kernel $K(x,y)$ of the operator $ (1- \chi_{\mu_{5}^{2}}) e^{i(t-t_{0}) \bilap} P_{\leq \mu_{2}^{-1}} \tilde{\chi}_{\mu_{5}} $ is

\begin{equation}
\begin{array}{ll}
K(x,y) & : = (1- \chi_{\mu_{5}^{2}} (x)) \mu_{2}^{n} \int e^{i ( (t-t_{0}) |\xi|^{4} + (x-y) \cdot \xi )}
\phi \left(\mu_{2} \xi \right)  \, d \xi \,  \tilde{\chi}_{\mu_{5}}(y)
\end{array}
\nonumber
\end{equation}
Now it is not difficult to see that if $(x,y) \in
\{1 - \chi_{\mu_{5}^{2}} >0 \} \times \{ \tilde{\chi}_{\mu_{5}} >0 \} $
then, $ |\nabla(\Psi(\xi))| \gtrsim |x-y| $, with
$\Psi(\xi):=(t-t_{0}) |\xi|^{4} + (x-y) \cdot \xi$. Therefore, by stationary phase, we see
that

\begin{equation}
\begin{array}{ll}
|K(x,y)| & \lesssim  \frac{\mu_{5}^{c}} {  \langle  x-y \rangle^{n+1}} \cdot  \\
\end{array}
\nonumber
\end{equation}
By Schur's lemma and by Minkowski inequality, we see that it is bounded from $L^{2}$ to $L^{2}$ and from
$L^{1}$ to $L^{2}$ (with norm $O( \mu_{5}^{c})$). By interpolation,

\begin{equation}
\begin{array}{ll}
\| (1- \chi_{\mu_{5}^{2}}) e^{i(t-t_{0}) \bilap} P_{\leq \mu^{-1}_{2}} \tilde{\chi}_{\mu_{5}}  \|_{L^{\tilde{p}} \rightarrow L^{2} } & \lesssim
\mu_{5}^{c}
\end{array}
\label{Eqn:L2Bound3}
\end{equation}
The conclusion then follows from (\ref{Eqn:BoundFu}) and Remark
\ref{Rem:Dual}.

\end{proof}

Next we prove a result very similar to Proposition \ref{prop:spat}:

\begin{lem}
We have
\begin{equation}
\begin{array}{ll}
\overline{\lim}_{t_{0} \rightarrow \infty} (1-\chi_{\mu_{5}^{3}}) v(t_{0}) & = O_{L^{2}} (\mu^{c}_{1})
\end{array}
\end{equation}
\label{lem:Decayv}
\end{lem}

\begin{proof}
Let $\tilde{\chi}_{\mu_{5}^{3}}:= 1 - \chi_{\mu_{5}^{3}}$. The local-in-space Duhamel bound that is proved in (\ref{Eqn:DecayOutConc}) allows to limit the interaction between (\ref{Eqn:vt0}) and (\ref{Eqn:vtinf}) as we shall see. Indeed, using Duhamel formula and (\ref{Eqn:DecayOutConc}) we see that

\begin{equation}
\begin{array}{ll}
\tilde{\chi}_{\mu_{5}^{3}} v(t_{0}) & = \tilde{\chi}_{\mu_{5}^{3}} e^{-i \mu^{-1}_{1} \bilap}
v(t_{0} + \mu^{-1}_{1}) + O_{L^{2}} (\mu^{c}_{2}) \\
\tilde{\chi}_{\mu_{5}^{3}} v(t_{0}) & = \tilde{\chi}_{\mu_{5}^{3}} e^{i \mu^{-1}_{1} \bilap} v(t_{0} - \mu^{-1}_{1}) + O_{L^{2}} (\mu^{c}_{2})
\end{array}
\nonumber
\end{equation}
We compute

\begin{equation}
\begin{array}{ll}
\| \tilde{\chi}_{\mu_{5}^{3}} v(t_{0}) \|^{2}_{L^{2}} & = < \tilde{\chi}_{\mu_{5}^{3}} e^{-i \mu^{-1}_{1} \bilap} v(t_{0} + \mu^{-1}_{1}),
\tilde{\chi}_{\mu_{5}^{3}} e^{i \mu^{-1}_{1} \bilap} v(t_{0}- \mu^{-1}_{1}) >_{L^{2}}
+ O_{L^{2}}(\mu^{c}_{2}) \\
& = < e^{-i \mu^{-1}_{1} \bilap} v(t_{0} + \mu^{-1}_{1}), \tilde{\chi}^{2}_{\mu_{5}^{3}} e^{-i \mu^{-1}_{1} \bilap} v(t_{0} - \mu^{-1}_{1}) >_{L^{2}} + O_{L^{2}}(\mu^{c}_{2})
\end{array}
\nonumber
\end{equation}
We will only deal with the case $t_{0}>0$. By applying (\ref{Eqn:vt0}) for $v( t_{0}- \mu^{-1}_{1})$ (using the approximation $u(0)-u_{+} = \psi + O_{L^{2}} (\mu_{2})$ with $\psi \in \mathcal{S}(\mathbb{R}^{n})$, (\ref{Eqn:vtinf}) for $v (t_{0} + \mu^{-1}_{1})$, we see that is it is enough to prove that

\begin{equation}
\begin{array}{ll}
\langle e^{-i \mu^{-1}_{1} \bilap} v(t_{0}+ \mu^{-1}_{1}),  (1- \tilde{\chi}^{2}_{\mu_{5}^{3}}) e^{i t_{0} \bilap} \psi \rangle_{L^{2}} & \lesssim \mu^{c}_{1},
\end{array}
\label{Eqn:FirstEst}
\end{equation}

\begin{equation}
\begin{array}{ll}
\left| \int_{t_{0} + \mu^{-1}_{1}}^{\infty} \langle e^{i(t_{0} - t^{'}) \bilap} F(u(t^{'})), \, e^{i t_{0} \bilap} \psi \rangle_{L^{2}} \, dt^{'} \right|
& \lesssim \mu^{c}_{1}
\end{array}
\label{Eqn:SecondEst}
\end{equation}
and

\begin{equation}
\begin{array}{ll}
\int_{t_{0}+ \mu^{-1}_{1}}^{\infty} \int_{0}^{t_{0} - \mu^{-1}_{1}} | < e^{i(t_{0} - t^{'}) \bilap} F(u(t^{'})), \, \tilde{\chi}^{2}_{\mu_{5}^{3}}
e^{i (t_{0} - t^{''}) \bilap} F(u(t^{''})) >_{L^{2}} |  \, dt^{'} \, dt^{''} & \lesssim \mu^{c}_{1}
\end{array}
\label{Eqn:ThirdEst}
\end{equation}
Now we prove (\ref{Eqn:FirstEst}). By \cite{artzi}, we see that the kernel $K(x,y)$ of
$ (1- \tilde{\chi}^{2}_{\mu_{5}^{3}})  e^{i t_{0} \bilap} \langle y \rangle^{-100n}  $ satisfies

\begin{equation}
\begin{array}{ll}
|K(x,y)| & \lesssim \frac{1- \tilde{\chi}^{2}_{\mu_{5}^{3}}}{ t_{0}^{\frac{n}{4}} \langle y \rangle^{100n} },
\end{array}
\nonumber
\end{equation}
and, by Shur's lemma combined with the high regularity of $\psi$ we see that $ \| (1- \tilde{\chi}^{2}_{\mu_{3}^{3}}) e^{i t_{0} \bilap} \psi  \|_{L^{2}}
=O(\mu^{c}_{1})$ for $t_{0} \gg 1$. Combining this inequality with (\ref{Eqn:BounduplusvH2}) we see that (\ref{Eqn:FirstEst}) holds.

Next we prove (\ref{Eqn:SecondEst}). We have

\begin{equation}
\begin{array}{ll}
| \langle e^{i(t_{0}-t^{'}) \bilap} F(u(t^{'})), \, e^{it_{0} \bilap}  \psi \rangle |_{L^{2}} &
\lesssim \|  e^{-it' \bilap} F(u(t^{'}))  \|_{L^{\tilde{p}^{'}}}  \| \psi \|_{L^{\tilde{p}}} \\
& \lesssim \frac{1}{|t^{'}|^{\frac{n}{4} \left( 1 -\frac{2}{\tilde{p}'} \right)  }},
\end{array}
\label{Eqn:Ineq0}
\end{equation}
by (\ref{Eqn:BoundFu}). Hence (\ref{Eqn:SecondEst}) holds.

(\ref{Eqn:ThirdEst}) is more difficult to establish. We write

\begin{equation}
\begin{array}{l}
\langle e^{i(t_{0} -t^{'}) \bilap} F(u(t^{'})), \,  \tilde{\chi}^{2}_{\mu_{5}^{3}} e^{i(t_{0} - t^{''}) \bilap} F(u(t^{''}))  \rangle_{L^{2}} \\
=  \langle e^{i(t_{0}-t^{'}) \bilap} F(u(t^{'})), \,  e^{i(t_{0} - t^{''}) \bilap} F(u(t^{''}))  \rangle_{L^{2}}  \\
 - \langle e^{i(t_{0} -t^{'}) \bilap} F(u(t^{'})), \, (1- \tilde{\chi}^{2}_{\mu_{5}^{3}}) e^{i(t_{0} -t^{''}) \bilap} F(u(t'')) \rangle_{L^{2}} \\
= A + B
\end{array}
\nonumber
\end{equation}
so that we can use the fact that  $ 1- \tilde{\chi}^{2}_{\mu_{5}^{3}} $ is a compactly and nice decaying function. $A$ is
treated in a similar way to (\ref{Eqn:Ineq0}):

\begin{equation}
\begin{array}{ll}
A & \lesssim \frac{1}{|t^{'} - t^{''}|^{\frac{n}{4} \left( 1 - \frac{2}{\tilde{p}^{'}} \right) }} \cdot \\
\end{array}
\nonumber
\end{equation}
We see that $B$ can be written as $B = \int K(t',t'',t_{0},y,z) F(u(t',y)) \bar{F}(u(t'',z)) \, dy  \, dz $ with
$K$ the kernel defined by

\begin{equation}
\begin{array}{ll}
K(t',t'',t_{0},y,z) & := \frac{1}{(t_{0} -t')^{\frac{n}{4}}} \frac{1}{(t'' -t_{0})^{\frac{n}{4}}} \int
\left( \int e^{ i |\xi|^{4}} e^{ i \xi \cdot \frac{x-y}{ (t_{0} -t^{'})^{\frac{1}{4}}} } \, d \xi \right)
\left( \int e^{ i |\xi|^{4}}  e^{i \xi \cdot \frac{x-z}{ (t^{''} - t_{0})^{\frac{1}{4}}}  } \, d \xi \right) ( 1-
\tilde{\chi}^{2}_{\mu_{5}^{3}}) \cdot
\end{array}
\label{Eqn:DefK}
\end{equation}
Let $c$ be a small positive constant that is allowed to change from one line to another. We claim that
for all $(y,z) \in \mathbb{R}^{n} \times \mathbb{R}^{n}$ such that $y \neq z$ the following kernel estimate holds

\begin{equation}
\begin{array}{ll}
| K(t',t^{''},t_{0},y,z) & \lesssim \frac{1}{|t'-t''|^{c}}:
\end{array}
\label{Eqn:EstK}
\end{equation}
this estimate is delicate to prove and we postpone it to Section \ref{Section:EstK}. Hence

\begin{equation}
\begin{array}{ll}
| \langle e^{i(t_{0}-t') \triangle^{2}} F(u(t')), (1- \tilde{\chi}^{2}_{\mu_{5}^{3}})  e^{i (t_{0} - t'') \triangle^{2}} F(u(t'')) \rangle  | & \lesssim
\frac{1}{|t' - t''|^{c}} \| F(u(t')) \|_{L^{1}}  \| F(u(t'')) \|_{L^{1}}
\end{array}
\label{Eqn:EstL1L1}
\end{equation}

\begin{rem}
We notice that if we were to estimate $B$ by passing to the Fourier domain
then the strategy would be doomed to fail since the Fourier transform of
$ 1- \tilde{\chi}_{\mu_{5}^{3}} $ would depend on the number of points of concentration
$J:=J(\mu_{3})$ that can be really large compare with $|t^{''} - t'|$.  Hence it is
necessary to work in the spatial domain.
\end{rem}

\begin{rem}
Observe that if $\tilde{\chi}_{\mu_{5}^{3}}$ were equal to zero then, by
(\ref{Eqn:Dispers}), we would have found

\EQQARR{
\left| \langle \; e^{i(t_0 -t') \triangle^{2}} F(u(t')), e^{i(t_0 - t^{''}) \triangle^{2}} F(u(t''))  \rangle \right|
& = \left| \langle \; e^{i(t'' - t') \triangle^{2}} F(u(t')), F(u(t''))  \rangle  \right| \\
& \lesssim \frac{1}{|t^{''} - t'|^{\frac{n}{4}}} \| F(u(t')) \|_{L^{1}} \| F(u(t'')) \|_{L^{1}} \cdot
}
It is an open problem to know what the best value of $c$ is in (\ref{Eqn:EstL1L1}). Maybe it is $c = \frac{n}{4}$.

\end{rem}
We also have

\begin{equation}
\begin{array}{ll}
| \langle e^{i(t_{0}-t') \triangle^{2}} F(u(t')), (1- \tilde{\chi}_{\mu_{5}^{3}}^{2})  e^{i (t_{0} - t'') \triangle^{2}} F(u(t''))  \rangle  | & \lesssim
\| F(u(t')) \|_{L^{2}}  \| F(u(t'')) \|_{L^{2}} \cdot
\end{array}
\nonumber
\end{equation}
Hence by interpolation

\begin{equation}
\begin{array}{ll}
| \langle e^{i(t_{0}-t') \triangle^{2}} F(u(t')), (1 - \tilde{\chi}_{\mu_{5}^{3}}^{2} )  e^{i (t_{0} - t'') \triangle^{2}} F(u(t'')) \rangle  | & \lesssim
\frac{1}{|t' - t''|^{c \left( 1 - \frac{2}{\tilde{p}^{'}} \right) } } \| F(u(t')) \|_{L^{\tilde{p}}} \| F(u(t'')) \|_{L^{\tilde{p}}} \\
& \lesssim \mu_{1}^{c},
\end{array}
\nonumber
\end{equation}
where we used  (\ref{Eqn:BoundFu}) at the last line.





\end{proof}

\subsection{Final asymptotic spatial localization}
\label{subsec:finalspatial}

In this subsection we prove the final asymptotic spatial localization.  We prove the $L^{2}$ decay of the nonradiative part of the
solution outside a neighborhood of points $x_{j}:=x_{j}(t)$  such that their number does only depend on time $t$.

Let $1 \gg \mu_{0} \gg \mu_{1}$. Let $\mu_{2}$,  $ \mu_{3} $ and $\mu_{4}$ be constants chosen such that

\begin{itemize}
\item $\mu_{1} \gg \mu_{2} \gg \mu_{3} \gg \mu_{4} $
\item all the inequalities in this section are true
\item $\mu_{4}$ is a nondecreasing function of $\mu_{1}$
\end{itemize}
Let also $c$,$C$ denote a small, large constant whose value can change from one line to the other line. We prove the following proposition:

\begin{prop}{\textbf{`` Final asymptotic spatial localization ''}}
Given $\bar{\mu}_{1}>0$ on can find $\tilde{\mu}_{2}:= \tilde{\mu}_{2}(\bar{\mu}_{1}) >0$, $J \geq 0$
and $x_{1}(t)$, ... $x_{J}(t)$  such that

\begin{equation}
\begin{array}{ll}
\overline{\lim}_{t \rightarrow \infty} \int_{\inf_{1 \leq j \leq J} |x-x_{j}(t)| \geq \tilde{\mu}^{-1}_{2}} |v(t,x)|^{2} \, dx & \leq \bar{\mu}^{2}_{1}
\end{array}
\label{Eqn:FinalSpatProve}
\end{equation}
\label{prop:finalspat}
\end{prop}
The proof relies upon the following lemma:

\begin{lem}
Assume that for some $x_{0} \in \mathbb{R}^{n}$ and $R \geq 0$ we have

\begin{equation}
\begin{array}{lll}
\mu_{1}^{2} \leq & \int_{|x-x_{0}| \leq R} |u(t_{0},x)|^{2} \, dx \leq \mu_{0}^{2}
\end{array}
\end{equation}
Then we can find $R^{'}:=R^{'}(R,\mu_{4})$ such that

\begin{equation}
\begin{array}{ll}
\int_{|x-x_{0}| \leq R^{'}} |u(t_{0},x)|^{2} \, dx & \geq \int_{|x-x_{0}| \leq R} |u(t_{0},x)|^{2} \, dx + \mu^{2}_{4}
\end{array}
\end{equation}
\label{Lem:IncrMass}
\end{lem}

\begin{proof}

We can assume that $x_{0}=0$ without loss of generality. If the statement were not true then we would have

\begin{equation}
\begin{array}{ll}
\int_{|x| \geq R}  |u(t_{0},x)|^{2} \, dx & \leq \mu^{2}_{4}
\end{array}
\label{Eqn:Abs1}
\end{equation}
and

\begin{equation}
\begin{array}{ll}
\int_{\mathbb{R}^{n}} |u(t_{0},x)|^{2} \, dx  & \lesssim \mu^{2}_{0}
\end{array}
\label{Eqn:Abs2}
\end{equation}
Let $R^{'}:=R^{'}(\mu_{4})$ be large enough such that all the inequalities below are true. Let $I=[t_{0}, t_{0} + \mu^{-1}_{3}]$. Before
proceeding, we prove a local-in-time local-in-space mass bound

\begin{lem}
We have
\begin{equation}
\begin{array}{ll}
\sup_{t \in I} \int_{|x| \geq R^{'}} |u(t,x)|^{2} \, dx & \lesssim \mu^{2}_{4}
\end{array}
\label{Eqn:Ineq1}
\end{equation}
and

\begin{equation}
\begin{array}{ll}
\inf_{t \in I} \int_{|x| \leq R} |u(t,x)|^{2} \, dx & \gtrsim \mu^{2}_{1}
\end{array}
\label{Eqn:Ineq2}
\end{equation}

\end{lem}

\begin{proof}

A computation (using appropriately the divergence theorem, forcing derivatives of expressions at most second order derivatives of the
solution $u$ to appear) shows that a.e

\begin{equation}
\begin{array}{ll}
\partial_{t} |u|^{2} & = -  2 \Im (\bilap u \bar{u}) \\
& = - 2 \sum_{j=1}^{n} \sum_{i=1}^{n} \Im \partial^{2}_{x_{j} x_{i}} \left(  \partial^{2}_{x_{j} x_{i}} u \bar{u}  \right)
+ 4 \sum_{j=1}^{n} \sum_{i=1}^{n} \Im \partial_{x_{j}} \left( \partial^{2}_{x_{j} x_{i}} u \partial_{x_{i}} \bar{u}   \right)
\end{array}
\label{Eqn:DivMass}
\end{equation}
Now, let $\omega$ be a smooth function such that $\omega(x)=1$ if $|x| \geq 1$ and $\omega(x)=0$ if $|x| \leq \frac{1}{2}$. Let $\omega_{R^{'}}(x) :=
\omega \left( \frac{x}{R^{'}} \right) $. Multiplying (\ref{Eqn:DivMass}) by $\omega_{R^{'}}$, integrating w.r.t  $x$ and $t$, using the boundedness of Riesz
transforms, Remark \ref{rem:inhomhom} and (\ref{Eqn:Abs1}), we see that,
if $t \in I$

\begin{equation}
\begin{array}{ll}
\int_{|x| \geq R^{'}} |u(t,x)|^{2} \, dx - \mu_{4}^{2}  & \lesssim \mu^{2}_{4} + \sum_{j=1}^{n} \sum_{i=1}^{n} \int_{t_{0}}^{t} \left| \partial^{2}_{x_{j} x_{i}}
\omega_{R^{'}} \Im  \left( \partial^{2}_{x_{j}x_{i}} u \bar{u} \right) \right| \, dx \, dt^{'}  \\
&  + \sum_{j=1}^{n} \sum_{i=1}^{n} \int_{t_{0}}^{t} |   \partial^{2}_{x_{j} x_{i}} u \partial_{x_{i}} \bar{u} \partial_{x_{j}} \omega_{R^{'}} |  \, dx \, dt  \\
& \lesssim \frac{1}{R^{'}} \mu^{-C}_{3} \\
& \lesssim  \mu^{2}_{4}
\end{array}
\end{equation}
Then (\ref{Eqn:Ineq1}) holds. The proof of (\ref{Eqn:Ineq2}) is similar and left to the reader.

\end{proof}

Let $\tilde{\omega}$ be a smooth function such that $\tilde{\omega}(x)=1$ if $|x| \leq 1$ and $\tilde{\omega}(x)=1- \alpha$ if $|x| \geq 2$, with $\alpha$ so small that $\alpha \ll \frac{\mu_{4}}{(R^{'})^{\frac{n}{Q} (p-1) - 1}}$. Let
$\tilde{\omega}_{R^{'}}(x):= \tilde{\omega} \left( \frac{x}{R^{'}} \right)$. Then it is not difficult to
see that we also have

\begin{equation}
\begin{array}{ll}
\| | \partial_{x_{i}} \tilde{\omega}_{R^{'}}|^{\frac{1}{p-1}} \|^{p-1}_{L^{Q}} & \lesssim \mu_{4} \\
|\alpha| = k: \, \| \partial^{\alpha} \tilde{\omega}_{R^{'}} \|_{L^{\infty}} & \lesssim \frac{\mu_{4}^{k}}{R'^{k}}
\end{array}
\label{Eqn:ControlomegaR}
\end{equation}
We see from Sobolev embedding that for $ Q < \bar{Q} \leq \frac{2n}{n-4}$

\begin{equation}
\| u \|_{ L_{t}^{\infty} L^{\bar{Q}}(I)} \lesssim \| u \|_{L_{t}^{\infty} H^{2}(I)} \lesssim 1.
\label{Eqn:SobEmbedBarQ}
\end{equation}
Let $w:=\tilde{\omega}_{R^{'}} u$. A computation shows that

\begin{equation}
\begin{array}{ll}
i \partial_{t} w + \bilap w  & = F(w) + (F(u) \tilde{\omega}_{R^{'}} - F( u \tilde{\omega}_{R^{'}})) + \bilap \tilde{\omega}_{R^{'}} u +
2 \nabla (\triangle \tilde{\omega}_{R^{'}}) \cdot \nabla u \\
& + 2 \nabla \cdot (\nabla \tilde{\omega}_{R^{'}} \triangle u) - \triangle \tilde{\omega}_{R^{'}} \triangle u  + \triangle (\nabla \tilde{\omega}_{R^{'}} \cdot \nabla u)
\end{array}
\end{equation}
By (\ref{Eqn:Strich}) and (\ref{Eqn:StrichGain}) we see that

\begin{equation}
\begin{array}{ll}
\| w \|_{L_{t}^{2} L_{x}^{\frac{2n}{n-4}}(I)} & \lesssim \| w(t_{0}) \|_{L^{2}} + \| F(w) \|_{L_{t}^{2} L_{x}^{\frac{2n}{n+4}} (I)}
+ \| F(u) \tilde{\omega}_{R^{'}} - F( u \tilde{\omega}_{R^{'}}) \|_{L_{t}^{2} L_{x}^{\frac{2n}{n+4}} (I)} \\
& + \| \bilap \tilde{\omega}_{R^{'}} u \|_{ L_{t}^{1} L_{x}^{2} (I) } + \| \nabla (\triangle \tilde{\omega}_{R^{'}}) \cdot \nabla u  \|_{L_{t}^{1} L_{x}^{2}(I)}
+ \left\| \triangle \tilde{\omega}_{R} \triangle u \right\|_{L_{t}^{1} L_{x}^{2}(I)}
 \\
& +  \sum _{i=1}^{n} \| \partial_{x_{i}} \tilde{\omega}_{R^{'}} \triangle u \|_{L_{t}^{2} L_{x}^{\frac{2n}{n+2}} (I)} +
\| \nabla ( \nabla \tilde{\omega}_{R^{'}} \cdot \nabla u ) \|_{L_{t}^{2} L_{x}^{\frac{2n}{n+2}} (I) } \\
\end{array}
\nonumber
\end{equation}
We have 

\begin{equation}
\begin{array}{ll}
\| F(w) \|_{L_{t}^{2} L_{x}^{\frac{2n}{n+4}}(I)} & \lesssim \| w \|^{p-1}_{L_{t}^{\infty} L_{x}^{\frac{n(p-1)}{4}} (I)}
\| w \|_{L_{t}^{2} L_{x}^{\frac{2n}{n-4}}(I)} \\
& \lesssim \mu_{0}^{c} \| w \|_{L_{t}^{2} L_{x}^{\frac{2n}{n-4}} (I)},
\end{array}
\nonumber
\end{equation}
by interpolation of (\ref{Eqn:Abs2}) and (\ref{Eqn:SobEmbedBarQ}). Moreover

\begin{equation}
\begin{array}{ll}
\| F(u \tilde{\omega}_{R^{'}}) - \tilde{\omega}_{R^{'}} F(u) \|_{L_{t}^{2} L_{x}^{\frac{2n}{n+4}} (I)} & \lesssim
\| \mathbf{1}_{|x| \geq R^{'}} |u|^{p-1} u \|_{L_{t}^{2} L_{x}^{\frac{2n}{n+4}} (I) } \\
& \lesssim \mu_{3}^{-C} \| u \|_{L_{t}^{q_{0}} L_{x}^{r_{0}} (I) } \| \mathbf{1}_{|x| \geq R^{'}} u \|^{p-1}_{L_{t}^{\infty} L_{x}^{Q} (I)} \\
& \lesssim \mu_{3}^{-C} \mu_{4}^{c},
\end{array}
\nonumber
\end{equation}
the last inequality following from (\ref{Eqn:LocalEst}) and the interpolation of (\ref{Eqn:Ineq1}) and (\ref{Eqn:SobEmbedBarQ}). We have

\begin{equation}
\begin{array}{ll}
\| \bilap \tilde{\omega}_{R^{'}} u \|_{L_{t}^{1} L_{x}^{2} (I)} & \lesssim \mu_{3}^{-C} \frac{1}{(R^{'})^{4}} \| u \|_{L_{t}^{\infty} L_{x}^{2} (I)} \\
& \lesssim \frac{\mu_{3}^{-C}}{(R^{'})^{4}}
\end{array}
\nonumber
\end{equation}
The other `` $L_{t}^{1} L_{x}^{2}$ '' terms are treated in a similar way. Next

\begin{equation}
\begin{array}{ll}
\| \partial_{x_{i}} \tilde{\omega}_{R^{'}}  \triangle u \|_{L_{t}^{2} L_{x}^{\frac{2n}{n+2}} (I)} & \lesssim \mu^{-C}_{3}
\| | \partial_{x_{i}} \tilde{\omega}_{R^{'}} |^{\frac{1}{p-1}} \|^{p-1}_{L^{Q}}  \| \triangle u \|_{L_{t}^{q_{0}} L_{x}^{r_{0}}(I)} \\
& \lesssim \mu_{3}^{-C} \mu_{4},
\end{array}
\nonumber
\end{equation}
by (\ref{Eqn:ControlomegaR}) and (\ref{Eqn:LocalEst}). The other $ L_{t}^{2} L_{x}^{\frac{2n}{n+2}} (I) $ term is treated in a similar way. Therefore,
by a continuity argument, we have $ \| \tilde{\omega}_{R^{'}} u \|_{L_{t}^{2} L_{x}^{\frac{2n}{n-4}} (I)} \lesssim \mu_{0}$ which implies, by the pigeonhole
principle, that there exists $t \in I$ such that

\begin{equation}
\begin{array}{ll}
\| \mathbf{1}_{|x| \leq R^{'}} u(t) \|_{L_{x}^{\frac{2n}{n-4}}} & \lesssim \mu_{0} \mu_{3}^{\frac{1}{2}} \\
& \lesssim  \mu_{3}^{c}.
\end{array}
\label{Eqn:Boundu1}
\end{equation}
Now, since

\begin{equation}
\begin{array}{ll}
\| \mathbf{1}_{|x| \leq R^{'}} e^{it \bilap} u_{+}   \|_{L_{x}^{\frac{2n}{n-4}}} \rightarrow 0
\end{array}
\label{Eqn:DispComp}
\end{equation}
as $t$ goes to infinity, we see that, if $t_{0}$ is large enough

\begin{equation}
\begin{array}{ll}
\| \mathbf{1}_{|x| \leq R^{'}} v(t) \|_{L_{x}^{\frac{2n}{n-4}}} & \lesssim \mu_{3}^{c}.
\end{array}
\label{Eqn:LowerBd}
\end{equation}
But this leads to a contradiction. Indeed we see that if $t_{0}$ is sufficiently large, then, by Proposition \ref{prop:spat}, we can find
$J:=J(\mu_{1}) \geq 0$  and $x_{1}(t,\mu_{1})$,..., $x_{J}(t,\mu_{1})$ such that

\begin{equation}
\begin{array}{ll}
\int_{\inf_{1 \leq j \leq J } |x-x_{j}(t,\mu_{1})| \geq \mu^{-1}_{2} } |v(t,x)|^{2} \, dx \lesssim \mu^{3}_{1}
\end{array}
\nonumber
\end{equation}
and substracting this inequality to (\ref{Eqn:Ineq2}) ( taking again into account that \\
 $ \| \mathbf{1}_{|x| \leq R'} e^{i t \bilap} u_{+} \|_{L^{2}}
\ll \mu_{1}^{2}$ by H\"older and (\ref{Eqn:DispComp})  as $ t \rightarrow \infty$ ) we see that

\begin{equation}
\begin{array}{ll}
\int_{|x| \leq R^{'}} \chi_{\inf_{1 \leq j \leq J} |x-x_{j}(t,\mu_{1})| \leq \mu^{-1}_{2} } |v(t,x)|^{2} \, dx \gtrsim \mu^{2}_{1}
\end{array}
\nonumber
\end{equation}
and by H\"older

\begin{equation}
\begin{array}{ll}
\| \mathbf{1}_{|x| \leq R^{'}} v(t) \|_{L^{\frac{2n}{n-4}}} & \gtrsim J^{-\frac{4}{n}} (\mu_{1}) \mu^{4}_{2} \mu^{2}_{1},
\end{array}
\end{equation}
which contradicts (\ref{Eqn:LowerBd}).
\end{proof}

With this lemma in mind, one can prove the final asymptotic spatial localization, i.e Proposition \ref{prop:finalspat}. The proof is given in
\cite{taocompact}; in order to make our argument complete, we rewrite it.

Let $\bar{\mu}_{0}$, $\tilde{\mu}_{0}$, $\bar{\mu}_{1}$,  $\tilde{\mu}_{1}$ and $\tilde{\mu}_{2}$ (with $\tilde{\mu}_{2} \ll \tilde{\mu}_{1}$
and $\bar{\mu}_{1} \leq \bar{\mu}_{0}$) small enough such that all the inequalities below hold. We already know from the asymptotic partial spatial
localization (see Proposition \ref{prop:spat}) that one can find numbers $J_{0}:=J_{0}(\bar{\mu}_{0})$, $J_{1}:=J_{1}(\bar{\mu}_{1})$ and points $x_{1}(t)$, .... ,$x_{J_{0}}(t)$, $z_{1}(t)$,...,$z_{J_{1}}(t)$ such that

\begin{equation}
\begin{array}{ll}
\int_{\inf_{1 \leq j \leq J_{0}} |x-x_{j}(t)| \geq  \tilde{\mu}^{-1}_{0}} |v(t,x)|^{2} \, dx & \leq \bar{\mu}^{2}_{0}
\end{array}
\label{Eqn:xj}
\end{equation}
and

\begin{equation}
\begin{array}{ll}
\int_{\inf_{1 \leq j^{'} \leq J_{1}}  |x-z_{j^{'}}(t)| \geq \tilde{\mu}^{-1}_{1}}  |v(t,x)|^{2} \, dx & \leq \frac{\bar{\mu}^{2}_{1}}{2}
\end{array}
\label{Eqn:zjprime}
\end{equation}
We aim at proving that in fact

\begin{equation}
\begin{array}{ll}
\int_{\inf_{1 \leq j \leq J_{0}} |x-x_{j}(t)| \geq \tilde{\mu}^{-1}_{2}} |v(t,x)|^{2} \, dx & \leq \bar{\mu}^{2}_{1}
\end{array}
\label{Eqn:ToproveFinalSpat}
\end{equation}
We write $ \int_{ \inf_{1 \leq j \leq J_{0}} |x-x_{j}(t)| \geq \tilde{\mu}^{-1}_{2} }  |v(t,x)|^{2} \, dx  =
A + B $ with

\begin{equation}
\begin{array}{ll}
A & := \int_{  \{ \inf_{1 \leq j \leq J_{0}} |x-x_{j}(t)| \geq  \tilde{\mu}^{-1}_{2} \}
\cap  \{ \inf_{1 \leq j^{'} \leq J_{1}}  |x-z_{j^{'}}(t)|  \geq \tilde{\mu}^{-1}_{1}  \}  } |v(t,x)|^{2} \, dx
\end{array}
\nonumber
\end{equation}
and

\begin{equation}
\begin{array}{ll}
B & := \int_{ \{ \inf_{1 \leq j \leq J_{1}} |x-x_{j}(t)| \geq \tilde{\mu}^{-1}_{2} \} \cap \{ \cup_{j^{'}=1}^{J_{1}} |x-z_{j^{'}}(t)|
\leq  \tilde{\mu}^{-1}_{1} \} } |v(t,x)|^{2} \, dx
\end{array}
\nonumber
\end{equation}
$A$ is easy to estimate: we have $A \leq \frac{\bar{\mu}^{2}_{1}}{2} $, by (\ref{Eqn:zjprime}). Let $j^{'} \in [1...,J_{1}]$. Assume that

\begin{equation}
\begin{array}{ll}
\int_{\left\{ |x-z_{j^{'}}(t)| \leq \tilde{\mu}^{-1}_{1} \right\} \cap \{ \inf_{1 \leq j \leq J_{1}} |x-x_{j}(t)| \geq \tilde{\mu}^{-1}_{2} \}} |v(t,x)|^{2} \, dx & \geq \frac{\bar{\mu}^{2}_{1}}{4 J_{1}}
\end{array}
\nonumber
\end{equation}
By iterating Lemma \ref{Lem:IncrMass} $ \sim \frac{\bar{\mu}^{2}_{0} }{\mu^{2}_{4}}$ times, we see that

\begin{equation}
\begin{array}{ll}
\int_{ |x- z_{j^{'}}(t)| < \frac{1}{4 \bar{\mu}_{2}}  } |v(t,x)|^{2} \, dx & \geq \bar{\mu}^{2}_{0}
\end{array}
\end{equation}
Hence, in view of (\ref{Eqn:xj}), $\left\{ |x - z_{j'}(t)| \leq \tilde{\mu}^{-1}_{1}  \right\} \subset
\left\{ \inf_{1 \leq j \leq J_1} |x - x_{j}(t)| \leq \frac{\tilde{\mu}_{2}^{-1}}{2} \right\}$. So $B \leq \frac{\bar{\mu}^{2}_{1}}{4}$ and (\ref{Eqn:ToproveFinalSpat}) holds.

\section{Perturbation argument}

In this section we prove the following perturbation argument:

\begin{prop}{\textbf{"Perturbation Argument"}}
Let $I=[a,b]$ be a bounded interval and $t_{0} \in I$. Let $\mu_{0} > 0$. Assume that $(u,v)$ are solutions of (\ref{Eqn:BiSchrod}) and that $u$ satisfies
(\ref{Eqn:Boundu}). There exists $\mu_{1}:= \mu_{1}(|I|, \mu_{0})$ such that if

\begin{equation}
\begin{array}{ll}
\| e^{i (t-t_{0}) \bilap} (u(t_{0}) -v(t_{0})) \|_{X(I)} & \leq \mu_{1}
\end{array}
\label{Eqn:SmallLinPart}
\end{equation}
then

\begin{equation}
\begin{array}{ll}
\| u  - v \|_{X(I)} & \leq \mu_{0}
\end{array}
\label{Eqn:Perturb}
\end{equation}
Assume furthermore that

\begin{equation}
\| u(t_{0}) - v(t_{0}) \|_{L_{x}^{2}} \lesssim 1.
\label{Eqn:SmallLinPart2}
\end{equation}
Then

\begin{equation}
\| u - v \|_{L_{t}^{\infty} L_{x}^{2}(I)} \lesssim_{|I|} 1
\label{Eqn:Perturb2}
\end{equation}

\label{prop:Perturb}
\end{prop}

\begin{proof}

Notice that we already now that from Proposition \ref{Prop:LocalEst} that

\begin{equation}
\begin{array}{ll}
\| u \|_{L_{t}^{q_{0}} L_{x}^{r_{0}} (I)} & \lesssim \langle |I| \rangle^{\alpha}.
\end{array}
\label{Eqn:BounduI}
\end{equation}
Then, let $w:=u-v$. A simple computation shows that

\begin{equation}
\begin{array}{ll}
i w_{t} + \bilap w & = F(v+w)-F(v)
\end{array}
\nonumber
\end{equation}
The proof is made of two steps: short time perturbation argument and long time perturbation argument (see \cite{colliand} for a similar
argument).
\begin{itemize}

\item short time perturbation argument. We can assume without loss of generality $\mu_{0} \ll 1$. We shall prove the following result: \\ \\

\underline{Result}:\\
Let $J=[\tilde{a},\tilde{b}] \subset I$. There exist four constants $ 0 < c \ll 1 $, $ 0 < \epsilon \ll 1$, $ \gamma \gg 1$,
and $C \gg 1$  such that if $\mu \leq \epsilon$,

\begin{equation}
\begin{array}{ll}
|J| & \leq \frac{c \mu^{\gamma}}{ \langle |I| \rangle^{\gamma} }
\end{array}
\label{Eqn:UpperSizeJ}
\end{equation}
and

\begin{equation}
\begin{array}{ll}
\| e^{i(t-\tilde{a}) \bilap} w(\tilde{a}) \|_{X(J)} & \leq \mu,
\end{array}
\nonumber
\end{equation}
then

\begin{equation}
\begin{array}{ll}
\| w \|_{X(J)} & \leq C \mu, \\
\| F(v+w) - F(v) \|_{L_{t}^{2} L_{x}^{\frac{2n}{n+4}} (J) } & \leq C \mu, \; \text{and} \\
\| \nabla F(v+w) - \nabla F(v) \|_{L_{t}^{2} L_{x}^{\frac{2n}{n+2}} (J)} & \leq C \mu \cdot
\end{array}
\label{Eqn:Ineq}
\end{equation}

\begin{proof}

By (\ref{Eqn:Strich}), Remark \ref{rem:inhomhom},(\ref{Eqn:Boundu}), the estimate

\begin{equation}
\begin{array}{l}
\| P_{\leq 1} f \|_{L_{x}^{Q}} \lesssim  \| f \|_{L^{2}},
\end{array}
\nonumber
\end{equation}
we see that we have, for some $\beta >0$,

\begin{equation}
\begin{array}{ll}
\max{( \| P_{\leq 1} w \|_{L_{t}^{\infty} L_{x}^{Q}(J)}, \| w \|_{L_{t}^{q_{0}} L_{x}^{r_{0}}(J)})} - \mu &
\lesssim \| F(v+w) - F(v) \|_{L_{t}^{2} L_{x}^{\frac{2n}{n+4}} (J)} \\
& \lesssim  |J|^{\beta} \| w \|_{L_{t}^{q_{0}} L_{x}^{r_{0}} (J) }
\left( \| u \|^{p-1}_{L_{t}^{\infty} L_{x}^{Q} (J)} + \| w \|^{p-1}_{L_{t}^{\infty} L_{x}^{Q} (J)}   \right) \\
& \lesssim |J|^{\beta} \| w \|_{X(J)} \left( 1 + \| w \|^{p-1}_{X(J)}  \right) \cdot
\end{array}
\label{EqnSchemeLowFreq}
\end{equation}
By (\ref{Eqn:StrichGain}), (\ref{Eqn:LocalEst}) and the estimate

\begin{equation}
\begin{array}{l}
\| P_{> 1} f \|_{L_{x}^{Q}} \lesssim  \| f \|_{\dot{H}^{2}},
\end{array}
\nonumber
\end{equation}
we get

\begin{equation}
\begin{array}{l}
\max{( \| P_{> 1} w \|_{L_{t}^{\infty} L_{x}^{Q}(J)}, \| w \|_{L_{t}^{q_{0}} \dot{W}^{2,r_{0}} (J)})} - \mu \\
 \lesssim \| \nabla F(v+w) - \nabla F(v) \|_{L_{t}^{2} L_{x}^{\frac{2n}{n+2}}(J)} \\
 \lesssim \left\| (F^{'}(v+w) - F^{'}(w)) \cdot \nabla v \right\|_{L_{t}^{2} L_{x}^{\frac{2n}{n-2}}(J)}
+  \left\| F^{'}(v+w) \nabla w \right\|_{L_{t}^{2} L_{x}^{\frac{2n}{n-2}}(J)} \\
 \lesssim |J|^{\beta} \left( \| v \|^{p-1}_{L_{t}^{\infty} L_{x}^{Q}(J)} \| \nabla v \|_{L_{t}^{q_{0}} L_{x}^{\tilde{r}_{0}}(J)}
+ \| u \|^{p-1}_{L_{t}^{\infty} L_{x}^{Q}(J)} \| \nabla v \|_{L_{t}^{q_{0}} L_{x}^{\tilde{r}_{0}}(J)} \right) \\
 \lesssim |J|^{\beta} \| \triangle v \|_{L_{t}^{q_0} L_{x}^{r_0}(J)} \left(
\| u \|^{p-1}_{L_{t}^{\infty} L_{x}^{Q} (J) } + \| w \|^{p-1}_{L_{t}^{\infty} L_{x}^{Q}(J)} \right)  \\
 \lesssim |J|^{\beta} \left( \| w \|_{X(J)} + \| w \|^{p}_{X(J)} + \langle |I| \rangle^{\alpha} (1 + \| w \|^{p-1}_{X(J)}) \right)  \cdot
\end{array}
\nonumber
\end{equation}
Hence if (\ref{Eqn:UpperSizeJ}) holds, then (\ref{Eqn:Ineq}) holds.

\end{proof}

\item Long time perturbation argument.\\

For $\mu_{1}$ to be chosen shortly we define $\{ \mu_{k} \}_{k \geq 1}$ in the following fashion:

\begin{equation}
\begin{array}{ll}
\mu_{k+1} = \mu_{1} + \sum_{j=1}^{k} 2 C \mu_{j}.
\end{array}
\label{Eqn:Constructionmuk}
\end{equation}
Let $\{ J_{j} \}_{ K \geq  j \geq 1}$ be a partition of $[t_{0},b]$ \footnote{The same argument would work on $[a,t_{0}]$}
such that

\begin{equation}
\begin{array}{ll}
|J_{j}| & = \frac{c \mu_{1}^{\gamma}}{ \langle |I| \rangle^{\gamma}},
\end{array}
\nonumber
\end{equation}
except maybe the last one. Choose $\mu_{1} \ll_{|I|} 1$ so small that $\mu_{k} \ll_{|I|} \epsilon$ for $k \leq K$. We claim that

\begin{equation}
\begin{array}{ll}
\| w \|_{X(J_{j})} & \leq C \mu_{j} \\
\| F(v+w) - F(v) \|_{L_{t}^{2} L_{x}^{\frac{2n}{n+4}} (J_{j}) } & \leq C \mu_{j} \\
\| \nabla F(v+w) - \nabla F(v) \|_{L_{t}^{2} L_{x}^{\frac{2n}{n+2}} (J_{j})} & \leq C \mu_{j}.
\end{array}
\label{Eqn:PropJk}
\end{equation}
This is proved by induction. Assume that (\ref{Eqn:PropJk}) holds for $1 \leq j \leq k < K$. Then
from (\ref{Eqn:DuhamelBiSchrod}), (\ref{Eqn:Strich}), and (\ref{Eqn:StrichGain}), we see that

\begin{equation}
\begin{array}{ll}
\|  e^{i(t-a_{k+1}) \bilap} w(a_{k+1}) \|_{X(J_{k+1})} &  \leq
\| e^{i (t-a_{k+1}) \bilap} w(t_{0}) \|_{X(I)} +  \sum_{j=1}^{k} \| F(v+w) - F(v) \|_{L_{t}^{2} L_{x}^{\frac{2n}{n+4}} (J_{j})} \\
& + \sum_{j=1}^{k} \| \nabla F(v+w) - \nabla F(v) \|_{L_{t}^{2} L_{x}^{\frac{2n}{n+2}} (J_{j}) } \\
& \leq \mu_{1} + \sum_{j=1}^{k} 2C \mu_{j} \\
& \leq \epsilon.
\end{array}
\end{equation}
We can now apply the previous result to get (\ref{Eqn:PropJk}) for $k=k+1$. \\
Hence, summing over $j$, we see that (\ref{Eqn:Perturb}) holds by making $\mu_{1}$ smaller if necessary.\\
Now assume that (\ref{Eqn:SmallLinPart2}) holds. By repeating the same scheme as (\ref{EqnSchemeLowFreq})
we see that

\begin{equation}
\| w \|_{L_{t}^{\infty} L_{x}^{2}(I)} \lesssim 1 + |I|^{\beta} \| w \|_{X(I)} \left( 1 + \| w \|^{p-1}_{X(I)}) \right)  \lesssim_{|I|} 1.
\nonumber
\end{equation}

\end{itemize}

\end{proof}

\section{Kernel estimate}
\label{Section:EstK}

In this section we prove the kernel estimate (\ref{Eqn:EstK}). \\

\subsection{General notation}

In the proof of (\ref{Eqn:EstK}), we use the following notation. \\
\\
Given $d \geq 1$, let $\eta$, $\tilde{\eta}$ be two smooth radial nonnegative functions such that

\EQQARR{
y \in \mathbb{R}^{d}: \; \eta(y) = 1 \;  \text{if}  \;  |y| \leq 2 \; \text{and} \;  \eta(y) = 0 \;  \text{if} \;  |y| \geq 4 \cdot \\
\\
\supp (\tilde{\eta}) \subset \mathcal{B}(0,1) \; \text{and} \; \int_{\mathbb{R}^{d}} \tilde{\eta} \; dy =1
}
If $f$ is a function then we let $\comp{f} := 1 - f$.

\subsection{Preliminaries and bipolar coordinates}

Let $\phi_{x}(\xi) := |\xi|^{4} + \xi \cdot x$ and
$\xi_{st} := - \frac{x}{4^{\frac{1}{3}} |x|^{\frac{2}{3}}} $ the stationary point of $\phi_x$, i.e the point
$\xi_{st}$ such that $\nabla \phi_{x}(\xi_{st})=0$. Let

\begin{align*}
I(x) := \int_{\mathbb{R}^{n}}  e^{i \phi_{x}(\xi)} \; d \xi \cdot
\end{align*}
Recall from \cite{artzi} that if $\alpha \in \mathbb{N}^{n}$ then

\begin{align}
| \p^{\alpha} I (x)| \lesssim \frac{1}{ \langle  x  \rangle^{\frac{n - |\alpha|}{3}}} \cdot
\label{Eqn:EstI}
\end{align}
We introduce the modified fundamental solution $\tilde{I}$  by writing

\begin{align}
I(x) = e^{ i \xi_{st} \cdot x } \tilde{I}(x),
\label{Eqn:Phaseout}
\end{align}
with

\begin{align}
\tilde{I}(x) := \int_{\mathbb{R}^{n}} e^{i \left( \phi_{x}(\xi) - \xi_{st} \cdot x \right)} \, d \xi \cdot
\label{Eqn:DeftildeI}
\end{align}

We shall prove in Section \ref{Sec:OscInt} that for all $\beta \in \mathbb{N}$ \footnote{
Here $\partial_{\rho}$ means the radial derivative}

\begin{align}
|\partial^{\beta}_{\rho} \tilde{I}(x)| \lesssim \frac{1}{ \langle x  \rangle^{\frac{n + \beta}{3}}} \cdot
\label{Eqn:EsttildeI}
\end{align}
We have

\EQQARR{
K & = \frac{1}{(t_{0} - t')^{\frac{n}{4}}} \frac{1}{(t'' -t_{0})^{\frac{n}{4}}} \int_{\mathbb{R}^{N}} e^{i \tilde{\phi}(x)}
\tilde{I} \left( \frac{x-y}{(t_{0} - t')^{\frac{1}{4}}}  \right) \tilde{I} \left( \frac{x-z}{(t'' - t_{0})^{\frac{1}{4}}} \right)
\comp{(\tilde{\chi}^{2}_{\mu_{5}^{3}})} (x) \, dx,
\label{Eqn:Kt}
}
with

\begin{align*}
\tilde{\phi}(x) := - 4^{-\frac{1}{3}}  \left( \frac{|x-y|^{\frac{4}{3}}}{(t_{0}- t')^{\frac{1}{3}}} + \frac{|x-z|^{\frac{4}{3}}}{(t''-t_{0})^{\frac{1}{3}}} \right) \cdot
\end{align*}
We shall estimate without loss of generality $K$ in the following case:

\EQQARR{
y=0, \quad \text{and} \quad a \geq 1.
}
with $ a := \frac{t_0 - t'}{t^{''}-t_0}$. Indeed, one can check that the estimate of $K$ in this case is invariant under the transformation $ x \rightarrow x - \bar{x} $ with $\bar{x} \in \mathbb{R}^{2}$.
Hence by elementary changes of variables, all the other cases boil down to this one. Notice that this implies that $z \neq 0$. \\
Observe that the phase $e^{i\tilde{\phi}}$ depends only on the two variables $\rho := |x|$ and $\sigma := |x-z|$. Hence it is useful to make a change of
variable that emphasizes these two variables. To this end we use the bipolar coordinates (see e.g  \cite{foshi, foshiklain}) w.r.t. the origin
$O$ and $z$. Given $x \in \mathbb{R}^{n}$, we let $\rho:= |x|$ and $\sigma:=|x-z|$ be the bipolar coordinates w.r.t $O$ and $z$. Recall that

\EQQARRLAB{
\label{Eqn:Formula}
\int_{\mathbb{R}^{n}} f(x)  \, dx & \approx \int_{(\rho,\sigma) \in \Rec} f(\rho,\sigma,w') \frac{\rho \sigma}{|z|} \left( \frac{\mathcal{A}(\rho,\sigma, |z|)}{|z|} \right)^{n-3}  d \rho \, d \sigma \, d S_{\omega'} \cdot
}
Here $\Rec$ denotes the following half closed rectangle

\EQQARR{
\mathcal{R} :\ \left\{ (\rho,\sigma) \in \mathbb{R}^{2}: \; \rho,\sigma \geq 0; \quad  \rho \leq  \sigma + |z|, \; \sigma \leq \rho + |z|, \;
|z| \leq  \rho + \sigma \cdot \right\},
}
with boundary $\p \Rec $ made of three sides

\EQQARR{
\p \Rec :1 & := \left\{ (\rho,\sigma) \in \Rec: \; \rho - \sigma = |z|  \right\},   \\
\p \Rec :2 & := \left\{ (\rho,\sigma) \in \Rec: \; \rho + \sigma =  |z|  \right\}, \; \text{and} \\
\p \Rec  :3 & := \left\{ (\rho,\sigma) \in \Rec: \; \rho - \sigma =  -|z|  \right\} \cdot  \\
}
Here $\mathcal{A}$ denotes the area of the triangle with vertices $0$, $x$, and $z$; its value is given by
the Heron formula

\EQQARRLAB{
\mathcal{A} = \frac{1}{4} (\rho+\sigma+|z|)^{\frac{1}{2}} (\rho+\sigma-|z|)^{\frac{1}{2}} (\rho-\sigma+|z|)^{\frac{1}{2}}
(-\rho + \sigma + |z|)^{\frac{1}{2}} \cdot
\label{Eqn:Heron}
}
Here $w'$ represents the angular variable which parametrizes the $(n-2)-$ dimensional sphere that is obtained as the intersection
of the $(n-1)-$ dimensional spheres $\{ x \in \mathbb{R}^{n}: \, |x| = \rho \}$ and
$\{ x \in \mathbb{R}^{n}: \, |x-z| = \sigma \}$. Define the regions $\Rec_{a}$ and $\Rec_{b}$ by

\EQQARR{
\Rec_{a}:= \left\{ (\rho,\sigma) \in \Rec: \; \rho \gg |z| \quad \text{and} \quad \sigma \gg |z| \right\}, \; \text{and} \\
\\
\Rec_{b}:= \Rec / \Rec_{a} \cdot
}
Observe that $\rho \approx \sigma$ on $\Rec_{a}$  and $\rho \approx |z|$ or $\sigma \approx |z|$ on  $\Rec_{b}$.\\
\\
We then estimate $K$ by passing to the bipolar coordinates. The main advantage of this change of variable is that it considerably simplifies the estimates of the derivatives of the phase. The two main disadvantages are the following ones:

\begin{itemize}
\item it tends to $``\text{bound}'' $ the region of integration so one has to handle the boundary when we integrate by part the phase.
\item the integrand has less regularity so it is more difficult to handle. Observe that when we apply the formulas (\ref{Eqn:Formula})
and (\ref{Eqn:Heron}) the integrand is more singular as we approach $\p \Rec$. More precisely, the derivatives
w.r.t $\rho$ and $\sigma$ of $\mathcal{A}^{n-3}$ have very bad decay in the region of integration close to $\p \Rec$. Observe also that the integrand
is not differentiable more than $\left[  \frac{n-3}{2} \right]$ times for a large number of dimensions on $\p \Rec$. Hence it is preferable
not to integrate the phase by parts w.r.t to $\rho$ or $\sigma$ in this region.
\end{itemize}
In order to deal with the second disadvantage we use the following strategy:

\begin{itemize}

\item on $\Rec_{b}$ we integrate by part the phase w.r.t. $\rho$ or $\sigma$ in the region far from the stationary point $(0,0)$ and
far from the $\p \Rec$. It occurs that $K$ is integrable for high dimensions
thanks to the weights in (\ref{Eqn:EsttildeI}) in the region far from $(0,0)$ and close to $\p \Rec$; hence one can estimate $K$ directly. However
for lower dimensions $K$ is still not integrable in this region despite the presence of weights and one has to integrate the phase
by parts just a few times w.r.t $\rho$ or $\sigma$ to get integrability of $K$.

\item on $\Rec_{a}$ and far from $(0,0)$ we cannot estimate $K$ directly even if we are close to $\p \Rec$, since the region
is too large to get integrability of $K$. So we integrate the phase by parts w.r.t $ \rho' := \rho + \sigma$. Notice that when the derivative w.r.t
$\rho'$ hits the factor $(\rho - \sigma + |z|)^{\frac{n-3}{2}}$ or the factor $(-\rho + \sigma + |z|)^{\frac{n-3}{2}}$ of $\mathcal{A}^{n-3}$ it is equal
to zero. Notice also that the derivative w.r.t $\rho'$ of the factor $(\rho + \sigma - |z|)^{\frac{n-3}{2}}$ or
$(\rho + \sigma + |z|)^{\frac{n-3}{2}}$ of $\mathcal{A}^{n-3}$ has good decay since both factors are approximately equal to $\rho'$.
Hence this procedure kills the singularity when the derivative hits $\mathcal{A}^{n-3}$ by integration by parts. However there
is a drawback to this strategy: one does not optimize the oscillations of the phase to the most, in particular in the regions
where integration by parts w.r.t $\rho$ (resp. w.r.t $\sigma$) yields better decay than integration by parts w.r.t $\sigma$ (resp. w.r.t
$\rho$). So we cannot integrate the phase by parts as much as we want. In fact, we will integrate the phase by parts at a well-chosen distance
from $(0,0)$ just enough to get integrability of $K$.

\end{itemize}

When we use the bipolar coordinates, one has to estimate integrals $J$ of the form

\EQQARR{
J :=  \underset{(\rho,\sigma) \in \Rec} \int e^{i \tilde{\phi}} f(\rho,\sigma,\omega')
\frac{\rho \sigma}{|z|} \left( \frac{\mathcal{A}}{|z|} \right)^{n-3}
  d \rho \; d \sigma \; d S_{\omega'} \cdot
}
Let $[J] :=  f \frac{\rho \sigma}{|z|} \left( \frac{\mathcal{A}}{|z|} \right)^{n-3}$. Let $k \in \{ \rho, \sigma \}$.

\subsubsection{$\Rec_{a}$ and $\Rec_{b}$}
We need to differentiate $\Rec_{a}$ from $\Rec_{b}$. To this end let $\psi$ be the convolution of
$\mathbf{1}_{\Rec_{a}}$ and $(\rho,\sigma) \rightarrow \frac{1}{|z|^{2}} \tilde{\eta} \left( \frac{(\rho,\sigma)}{|z|} \right)$. The following holds (with $\alpha \in \mathbb{N}$)

\EQQARR{
| \p_{k}^{\alpha} \psi | \lesssim \frac{1}{|k|^{\alpha}} \cdot
}
The same estimate holds for $\comp{\psi}$. \\
Hence we can write $J=J^{a} + J^{b}$ with

\EQQARR{
J^{a} := \underset{(\rho,\sigma) \in \Rec} \int e^{i \tilde{\phi}} [J] \; \psi  \; d \rho \; d \sigma \; d S_{\omega'}  \\
J^{b} := \underset{(\rho,\sigma) \in \Rec} \int e^{i \tilde{\phi}} [J] \; \comp{\psi} \; d \rho \; d \sigma \; d S_{\omega'}
}

\subsubsection{$J^{a}$: integration by parts w.r.t $\rho'$} In order to deal with $J^{a}$, we write
$J^{a} = J^{a}_{cl} + J^{a}_{far}$ with

\EQQARR{
J^{a}_{cl} :=  \underset{(\rho,\sigma) \in \Rec}
\int e^{i \tilde{\phi}} \eta \left(\frac{\rho'}{\epsilon} \right) [J^{a}] \; d \rho \; d \sigma \; d S_{\omega'}, \; \text{and} \\
J^{a}_{far} :=  \underset{(\rho,\sigma) \in \Rec}
\int  e^{i \tilde{\phi}}  \comp{\eta} \left(\frac{\rho'}{\epsilon} \right) [J^{a}] \; d \rho \; d \sigma \; d S_{\omega'},
}
and $\epsilon > 0$ to be determined.

\subsubsection{$J^{b}$: integration by parts w.r.t $\rho$ and w.r.t  $\sigma$}

In order to deal with $J^{b}$, it is worth determining the regions of the plane for which the integration by parts of
the phase w.r.t $\rho$ yields better decay estimate than the integration by parts w.r.t $\sigma$. Roughly speaking, we
integrate by parts w.r.t $\rho$ in these regions and we integrate by parts w.r.t $\sigma$ in the complement of these regions. To this end we consider at first
sight two integrals:

\begin{itemize}

\item the integral appearing from the integration by parts w.r.t $\rho$ that contains the term
$ \p_{\rho} \left( \frac{1}{ _{ \p_{\rho} \tilde{\phi}}} \right) $

\item the integral appearing from the integration by parts w.r.t $\sigma$ that contains the term
$\p_{\sigma} \left( \frac{1}{_{ \p_{\sigma} \tilde{\phi}}} \right)$

\end{itemize}
We have

\EQQARR{
\left| \p_{\rho} \left( \frac{1}{_{\p_{\rho} \tilde{\phi}}} \right)  \right| \lesssim \left| \p_{\sigma} \left( \frac{1}{_{ \p_{\sigma} \tilde{\phi}}} \right)  \right|
& \Leftrightarrow |\rho| \gtrsim a^{\frac{1}{4}} |\sigma|
}
In order to emphasize this region, let $\Omega_{\rho} $ be an homogeneous function of degree $0$ and smooth away from the origin, such that
$\Omega_{\rho} =1$ on
$\left\{ (\rho,\sigma) \in \mathbb{R}^{2}: \, |\rho| \gtrsim  a^{\frac{1}{4}} |\sigma| \right\}$ and $ \Omega_{\rho} =0 $ on
$ \left\{ (\rho,\sigma) \in \mathbb{R}^{2}: \, |\rho| \ll  a^{\frac{1}{4}} |\sigma| \right\}$.
The following holds:
\EQQARR{
\text{if} \quad |\rho| \approx a^{\frac{1}{4}} |\sigma|: \quad |\partial^{\alpha}_{k} \Omega_{\rho}|  \lesssim \frac{1}{|k|^{\alpha}}; \quad
\text{if not}: \quad \partial^{\alpha}_{k} \Omega_{\rho} =0 \cdot
}
The same estimate holds for $\Omega_{\sigma} := \comp{\Omega_{\rho}}$.\\
Hence one can write $J^{b}$ as the sum of the $J^{b}_{k}$ with

\EQQARR{
J^{b}_{k}:= \underset{(\rho,\sigma) \in \Rec} \int  e^{i \tilde{\phi}} \Omega_{k} [J^{b}] \; d \rho \; d \sigma \; d S_{\omega'}, \;
}
Given $\epsilon > 0$ we can write $J^{b}_{k}$ as the sum of two terms

\EQQARR{
J^{b}_{k:cl} :=  \underset{(\rho,\sigma) \in \Rec}
\int e^{i \tilde{\phi}} \eta \left(\frac{k}{\epsilon} \right) [J^{b}_{k}] \; d \rho \; d \sigma \; d S_{\omega'}, \; \text{and} \\
J^{b}_{k:far} :=  \underset{(\rho,\sigma) \in \Rec}
\int  e^{i \tilde{\phi}}  \comp{\eta} \left(\frac{k}{\epsilon} \right) [J^{b}_{k}] \; d \rho \; d \sigma \; d S_{\omega'} \cdot
}

\subsubsection{$J^{b}_{k:far}$: neighborhood of $\p \Rec$ and interior}
In order to deal with $J^{b}_{k:far}$ we need to differentiate the neighborhood of $\p \Rec$ and its interior. To this end we make a
partition of $\Rec$ that emphasizes the neighborhood of $\p \Rec$ and the interior. Given $\beta>0$, write

\EQQARR{
1 = \bar{\Omega}_{1} + \bar{\Omega}_{2} + \bar{\Omega}_{3} + \bar{\Omega}_{4},
}
with

\EQQARR{
\bar{\Omega}_{1} := \eta \left( \frac{\rho - \sigma - |z|}{\beta} \right), \\
\bar{\Omega}_{2} :=  \comp{\eta} \left( \frac{\rho - \sigma - |z|}{\beta} \right) \eta \left( \frac{\rho + \sigma - |z|}{\beta} \right), \quad \text{and} \\
\bar{\Omega}_{3} := \comp{\eta} \left( \frac{\rho - \sigma - |z|}{\beta} \right) \comp{\eta} \left( \frac{\rho + \sigma - |z|}{\beta} \right)
\eta \left( \frac{\rho - \sigma + |z|}{\beta} \right) \cdot
}
Hence one can write $J^{b}_{k:far}$ as the sum of

\EQQARR{
J^{b}_{k:far,\p \Rec:l} := \underset{(\rho,\sigma) \in \Rec} \int  e^{i \tilde{\phi}} [J^{b}_{k:far}] \; \bar{\Omega}_{l} \; d \rho \; d \sigma \; d S_{\omega'},
}
with $l \in \{1,2,3 \}$, and

\EQQARR{
J^{b}_{k:far,\p \Rec:far} :=  \int  e^{i \tilde{\phi}} [J^{b}_{k:far}] \; \bar{\Omega}_{4} \mathbf{1}_{\Rec}  \; d \rho \; d \sigma \; d S_{\omega'}
\cdot
}
The following hold:

\EQQARR{
l \in \{ 1,2,3 \}: \; |\p^{\alpha}_{k} \bar{\Omega}_{l} |   \lesssim \frac{1}{\beta^{\alpha}}, \; \text{and} \\
| \p^{\alpha}_{k} ( \bar{\Omega}_{4} \mathbf{1}_{\Rec} ) | \lesssim \frac{1}{\beta^{\alpha}} \cdot
}
We shall estimate $J^{b}_{k:far,\p \Rec:far}$ by integrating the phase by parts w.r.t $k$.

\subsubsection{Estimates for $ \p_{k}^{1} \left( \comp{(\tilde{\chi}^{2}_{\mu_{5}^{3}})} \right)  $ }
We also need some estimates for $ \p_{k}^{1} \left( \comp{(\tilde{\chi}^{2}_{\mu_{5}^{3}})} \right)$. We have

\EQQARRLAB{
\underset{(\rho,\sigma) \in \Rec } \int \frac{\sigma \rho }{|z|} \left( \frac{\mathcal{A}}{|z|} \right)^{n-3}
\left| \p_{k}^{1} \left( \comp{(\tilde{\chi}^{2}_{\mu_{5}^{3}})} \right) \right|  d \rho \;  d \sigma \; d S_{\omega'}
\approx \int_{\mathbb{R}^{n}} \left| \p_{k} \left( \comp{(\tilde{\chi}^{2}_{\mu_{5}^{3}})} \right) \right|  \; dx \lesssim \mu_{5}^{c},
\label{Eqn:AreaTrick}
}
with $c$ a small positive constant.

\subsubsection{Proposition}

We finally state the following proposition, the proof of which is left to the reader.

\begin{prop}
Let $(\lambda_1,\lambda_2,\alpha_1,\alpha_2,\gamma) \in \mathbb{R}^{5}$. Let $(x,y,z) \in \mathbb{R}^{3}$ and
 $\epsilon \in (0, \infty]$. Then the following holds:

\EQQARR{
\text{if} \quad \alpha_{1} + \alpha_{2} + \gamma > 1 \quad \text{then} \\ \\
\int_{|x| \geq \epsilon}
\frac{1}{ \langle \frac{x}{\lambda_1} \rangle^{\alpha_1}}
\frac{1}{ \langle \frac{x}{\lambda_2} \rangle^{\alpha_2}} \frac{1}{|x|^{\gamma}} \, dx
\approx
\int_{|x| \approx \epsilon}
\frac{1}{ \langle \frac{x}{\lambda_1} \rangle^{\alpha_1}}
\frac{1}{ \langle \frac{x}{\lambda_2} \rangle^{\alpha_2}} \frac{1}{|x|^{\gamma}} \, dx
\approx
\frac{1}{ \langle \frac{\epsilon}{\lambda_1} \rangle^{\alpha_1}
\langle \frac{\epsilon}{ \lambda_2} \rangle^{\alpha_2} \epsilon^{\gamma -1}};  \\ \\
\text{if} \quad \alpha_1 - \gamma > 1 \quad \text{then} \quad \int_{|x| \leq \epsilon}
\frac{|x|^{\gamma}}{\langle \frac{x}{\lambda_1} \rangle^{\alpha_1}} \, dx \lesssim |\lambda_1|^{\gamma +1}; \\
\\
\text{if} \quad \alpha_1 + \alpha_2 - \gamma < 1 \quad \text{then} \quad \int_{|x| \leq \epsilon} \frac{|x|^{\gamma}}{\langle \frac{x}{\lambda_1} \rangle^{\alpha_1}
\langle \frac{x}{\lambda_2} \rangle^{\alpha_2}
} \, dx  \approx \int_{|x| \approx \epsilon} \frac{|x|^{\gamma}}{\langle \frac{x}{\lambda_1} \rangle^{\alpha_1}
\langle \frac{x}{\lambda_2} \rangle^{\alpha_2} } \, dx \approx
\frac{\epsilon^{\gamma+1}}{ \langle \frac{\epsilon}{\lambda_1} \rangle^{\alpha_1} \langle \frac{\epsilon}{\lambda_1} \rangle^{\alpha_2} } \cdot
\\ \\
\text{if} \quad |z| \ll |y| \quad \text{then} \quad
\underset{|x-y| \leq |z|} \int \frac{|x|^{\gamma}}{ \langle \frac{x}{\lambda_1} \rangle^{\alpha_1}} \, dx \lesssim
\frac{|y|^{\gamma} |z|}{ \langle \frac{y}{\lambda_1} \rangle^{\alpha_1}} \cdot
}

\label{Prop:BasicEst}
\end{prop}
We are now ready to estimate $K$. Write $\coef K = J^{a} + \underset{k} \sum J^{b}_{k} $ with

\EQQARR{
J  :=  \underset{(\rho,\sigma) \in \Rec} \int e^{i \tilde{\phi}} \;
\tilde{I} \left( \frac{(\rho,\sigma,\omega')}{(t_0 - t')^{\frac{1}{4}}} \right)
\tilde{I} \left( \frac{(\rho,\sigma,\omega') - z}{(t''-t_0)^{\frac{1}{4}}} \right)
\frac{\rho \sigma }{|z|} \left( \frac{\mathcal{A}}{|z|} \right)^{n-3}  \comp{(\tilde{\chi}^{2}_{\mu_{5}^{3}})}  \; d \rho \; d \sigma \; d S_{\omega'} \; \cdot
}
So it is sufficient to estimate $J^{a}$ and the $J^{b}_{k}$ s.

\subsection{Estimate of $J_{\rho}^{b}$}

Write $J_{\rho}^{b}$ as the sum of $J_{\rho:cl}^{b}$ and $J^{b}_{\rho:far}$.

\subsubsection{ Estimate of $J_{\rho:cl}^{b}$}

Notice that on the region of integration of $J_{\rho:cl}^{b}$, $\mathcal{A} \lesssim |\sigma| |z|$. We have

\EQQARRLAB{
\label{Eqn:Jrhocl1}
|J_{\rho:cl}^{b}| \quad \lesssim  \underset{|\rho| \lesssim \epsilon} \int \;
\underset{ a^{\frac{1}{4}} |\sigma| \lesssim |\rho|} \int
\frac{|\sigma|^{n-2}}{ \left\langle \frac{\sigma}{(t'' - t_0)^{\frac{1}{4}}} \right\rangle^{\frac{n}{3}}} \, d \sigma
\frac{1} { \left\langle \frac{\rho}{(t_0 - t')^{\frac{1}{4}}} \right\rangle^{\frac{n}{3}}}  \, d \rho \quad
\lesssim  \frac{\epsilon^n} {a^{\frac{n-1}{4}} \left\langle \frac{\epsilon}{(t_0 -t')^{\frac{1}{4}}} \right\rangle^{\frac{2n}{3}}} \cdot
}

\subsubsection{ Estimate of $J^{b}_{\rho:far}$} Write $J^{b}_{\rho:far}$ as the sum of
the $J^{b}_{\rho:far,\p \Rec:l}$ and $J^{b}_{\rho:far,\p \Rec : far}$.  \\
\\
\frame{Estimate of $J^{b}_{\rho:far,\p \Rec : far}$}  \\
\\
We integrate the phase by parts w.r.t $\rho$. Notice that during the process one has to estimate integrals that depend on the derivatives of $\tilde{\phi}$. This integrals are estimated by pointwise bounds of the derivatives of $\tilde{\phi}$. Since the dependance is not so easy to write, it is more convenient
to introduce classes of functions that satisfy the same pointwise bounds and to estimate the integrals depending on arbitrarily functions $f$
lying in these classes. \\ \\
We define (with $(p,q) \in \mathbb{N}^{2}$)

\EQQARR{
\mathcal{Q}_{p,q}:= \left\{ f \in C^{\infty}(\mathbb{R} - \{ 0 \}): \;
|\p^{\alpha} f(\rho)| \lesssim \frac{(t_0  -t')^{\frac{p}{3}}}{|\rho|^{\frac{p}{3} + q + \alpha }}, \; \alpha \in \mathbb{N} \right\} \cdot
}
Given $\vec{r}:=(r_1,...,r_5) \in \mathbb{N}^{5}$  and $j \in \{1,...,5\}$ we define $\overrightarrow{r_{j,+}}:= (r_1,...,r_j + 1,..,r_5)$. We also define
$\overrightarrow{0}:= (0,..,0)$. Let $\mathcal{P}_{p}:= \left\{ \vec{r} \in \mathbb{N}^{5}: \; \sum \limits_{j=1}^{5} r_j =p \right\}$. Given  $f \in \mathcal{Q}_{p,r_1}$ we define $ K_{\rho,\vec{r}} (f)
:= \int e^{i \tilde{\phi}} X_{\rho,\vec{r}}(f) \; d \rho \; d \sigma$ with

\EQQARR{
X_{\rho,\vec{r}}(f) & :=  f(\rho) \,
\p_{\rho}^{r_2} \left( \comp{\eta} \left( \frac{\rho}{\epsilon} \right) \; \Omega_{\rho} \; \; \comp{\psi}  \right)
\p_{\rho}^{r_3} \left( \tilde{I} \left( \frac{\rho}{(t_0 -t')^{\frac{1}{4}}} \right) \right) \;
\tilde{I} \left( \frac{\sigma}{(t^{''} - t_0)^{\frac{1}{4}}} \right) \\
&  \p_{\rho}^{r_4} \left(  \bar{\Omega}_{4} \mathbf{1}_{\Rec} \frac{\rho \sigma}{|z|} \left( \frac{\mathcal{A}}{|z|} \right)^{n-3}
\right)
\p_{\rho}^{r_5} \left( \comp{(\tilde{\chi}^{2}_{\mu_{5}^{3}})} \right) \cdot
}
Assume that $r_5=0$.
Integrating by parts w.r.t $\rho$, we see that there exist $g \in \mathcal{Q}_{p+1,r_{1} +1}$ and $h \in \mathcal{Q}_{p+1,r_{1}} $

\EQQARR{
K_{\rho,\vec{r}}(f) & =
\int  \frac{\p_{\rho} e^{i \tilde{\phi}}}{ _{ i \p_{\rho} \tilde{\phi}}} X_{\rho,\vec{r}} (f) \, d \rho \, d \sigma  \\
& =
- \int e^{i \tilde{\phi}} \p_{\rho}
\left( \frac{1}{ _{i \p_{\rho} \tilde{\phi}}} X_{\rho,\vec{r}}(f) \right) \, d \rho \, d \sigma \\
& = i \left(  K_{\rho, \overrightarrow{r_{1,+}}} (g) +  \sum \limits_{j=2}^{5} K_{\rho, \overrightarrow{r_{j,+}}} (h) \right),
}
Since $K_{\rho,\vec{0}}(1)=J^{b}_{\rho:far,\p \Rec : far}$, we see, by iterating over $p$, that we are reduced to estimate

\begin{itemize}
\item $K_{\rho,\vec{r}}(f)$ for $\vec{r} \in \mathcal{P}_{\bar{p}}$ such that $r_5=0$ and  $f \in \mathcal{Q}_{\bar{p},r_1}$ and
\item $K_{\rho,\vec{r}}(f)$ for $\vec{r} \in \mathcal{P}_p$, $1 \leq p \leq \bar{p}$, such that $r_5 =1$ and $f \in \mathcal{Q}_{p,r_1}$.
\end{itemize}
Notice again that on the region of integration of $K_{\rho,\vec{r}}(f)$, $\mathcal{A} \lesssim |\sigma| |z|$. It is worth choosing $\epsilon$ by considering only the term $K_{\rho,\vec{r}}(f)$ with $\vec{r}:=(\bar{p},0,..,0)$ for $\bar{p}$ large enough to assure integrability of
(\ref{Eqn:Modeps1rhoother}): the estimate of this term is the same as the one for which the integration by parts hits the derivative of the phase $\bar{p}$ times.\\
We have

\EQQARRLAB{
\label{Eqn:Modeps1rho}
|K_{\rho,\vec{r}}(f)| \quad & \lesssim \underset{|\rho| \gtrsim \epsilon} \int \; \frac{(t_0 -t')^{\frac{\bar{p}}{3}}}{|\rho|^{\frac{4 \bar{p}}{3}}} \frac{1}{\left\langle \frac{\rho}{(t_0 - t')^{\frac{1}{4}}} \right\rangle^{\frac{n}{3}}}
\underset{a^{\frac{1}{4}} |\sigma| \lesssim |\rho|} \int
\frac{|\sigma|^{n-2}}{ \left\langle \frac{\sigma}{(t^{''}-t_0)^{\frac{1}{4}}} \right\rangle^{\frac{n}{3}}}
\, d \sigma \, d \rho \lesssim \frac{(t_0 -t')^{\frac{\bar{p}}{3}} }{a^{\frac{n-1}{4}} \epsilon^{\frac{4 \bar{p}}{3}-n}
\left\langle \frac{\epsilon}{(t_0 - t')^{\frac{1}{4}}} \right\rangle^{\frac{2n}{3}} } \cdot \\
}
So optimizing in $\epsilon$ in (\ref{Eqn:Jrhocl1}) and (\ref{Eqn:Modeps1rho}) we find
$\epsilon \approx (t_0 - t')^{\frac{1}{4}}$. Choose $\beta = \epsilon$. We shall see shortly that this choice of $\beta$
is convenient. We now estimate the other terms with this value of $\epsilon$ and $\beta$. \\
If $r_{5}=0$ then

\EQQARRLAB{
\label{Eqn:Modeps1rhoother}
|K_{\rho,\vec{r}}  (f)| &  \lesssim  \underset{|\rho| \gtrsim \epsilon} \int \;
\frac{(t_0 -t')^{\frac{\bar{p}}{3}}}{|\rho|^{\frac{\bar{p}}{3}}}
 \frac{1}{|\rho|^{r_1+ r_2}}
\frac{1}{(t_0 - t')^{\frac{r_3}{4}}}
\frac{1}{\left\langle \frac{\rho}{(t_0 - t')^{\frac{1}{4}}} \right\rangle^{\frac{n + r_3}{3}}} \frac{1}{\epsilon^{r_4}}
\underset{ a^{\frac{1}{4}} |\sigma| \lesssim |\rho|} \int
\frac{ |\sigma|^{n-2}}{ \left\langle \frac{\sigma}{(t^{''}-t_0)^{\frac{1}{4}}} \right\rangle^{\frac{n}{3}}}
\, d \sigma \, d \rho  \\
& \lesssim
 \underset{|\rho| \approx \epsilon} \int \;
\frac{(t_0 -t')^{\frac{\bar{p}}{3}}}{|\rho|^{\frac{\bar{p}}{3}}}
 \frac{1}{\epsilon^{\bar{p}}} \frac{|\rho|^{n-1}}{a^{\frac{n-1}{4}}\left\langle \frac{\rho}{(t_0 -t')^{\frac{1}{4}}} \right\rangle^{\frac{2n + r_3}{3}}}
\; d \rho \\
& \lesssim  \frac{(t_0 - t')^{\frac{\bar{p}}{3}} }{ a^{\frac{n-1}{4}} \epsilon^{\frac{4 \bar{p}}{3}-n}
\left\langle \frac{\epsilon}{(t_0 - t')^{\frac{1}{4}}} \right\rangle^{\frac{2n}{3}} } :
}
this is the same estimate as (\ref{Eqn:Modeps1rho}); and if $r_5 =1 $, then from (\ref{Eqn:AreaTrick}) we see that

\EQQARRLAB{
|K_{\rho,\vec{r}} (f)|  & \lesssim  \underset{ \substack{(\rho,\sigma) \in \Rec \\ \rho \gtrsim \epsilon }} \int
\frac{(t_0 - t')^{\frac{\bar{p}}{3}}}{\rho^{^{\frac{\bar{p}}{3}}}} \frac{1}{\rho^{r_1 +  r_2 }}
\frac{1}{(t_0 - t')^{\frac{r_3}{4}}} \frac{1}{\epsilon^{r_4}}
\frac{\sigma \rho }{|z|} \left( \frac{\mathcal{A}}{|z|} \right)^{n-3}
\left| \p_{\rho}^{1} \left( \comp{(\tilde{\chi}^{2}_{\mu_{5}^{3}})} \right) \right|
  d \sigma \; d \rho \; d S_{\omega'} \\
& \lesssim  (t_0 - t')^{\frac{1}{4}} \cdot
\label{Eqn:Estj5Eq1}
}
\\
\frame{Estimate of $J^{b}_{\rho:far,\p \Rec:l}$  for $n > 9$} \\
\\
If $n > 9$, then, ignoring the compact support of $\comp{(\tilde{\chi}^{2}_{\mu_{5}^{3}})}$ \footnote{Here we should not take into account $\comp{(\tilde{\chi}^{2}_{\mu_{5}^{3}})}$ to assure integrability. Indeed, despite the fact that it is compactly supported, its support depends on the size $J(M,\mu_{3})$. So it may be really large compared with $t_0 - t'$ or $t^{''} - t_0$ in the case where $ t_0 - t' \approx t^{''} - t_0  \approx \mu_{1}^{-1}$},   $ |J^{b}_{\rho:far,\p \Rec:l}| $ is bounded by an integrable expression thanks to the weights in (\ref{Eqn:EsttildeI}); more precisely

\EQQARR{
|J^{b}_{\rho:far,\p \Rec:l}|  &  \lesssim  \epsilon^{\frac{n-3}{2}} \int_{|\rho| \gtrsim \epsilon}
\underset{a^{\frac{1}{4}} |\sigma| \lesssim |\rho|}  \int
\frac{|\sigma|^{\frac{n-1}{2}}}
{ \left\langle \frac{\rho}{ (t_0 - t')^{\frac{1}{4}} } \right\rangle^{\frac{n}{3}} \left\langle \frac{\sigma}{ (t^{''} - t_0)^{\frac{1}{4}}} \right\rangle^{\frac{n}{3}} } \, d \sigma \, d \rho \\
&  \lesssim (t_0 -t')^{\frac{n}{4}-} (t^{''} - t_0)^{\frac{n}{4}-}, 
}
using $\mathcal{A} \lesssim \epsilon^{\frac{1}{2}} |\sigma|^{\frac{1}{2}} |z|$. \\
\\
\frame{Estimate of $J^{b}_{\rho:far,\p \Rec:l}$  for $ 5 \leq  n \leq 9$}\\
\\
Observe that if $ 5 \leq n \leq 9$ the integral above is infinite. Hence, in order to get
integrability and use the oscillation of the phase to our advantage, we integrate by parts a small number of times w.r.t $\rho$. \\
Observe  that if we apply $\p_{\rho}$ to $\mathcal{A}$ on the region of integration of $J^{b}_{\rho:far,\p \Rec:l}$, then

\EQQARRLAB{
\left.
\begin{array}{l}
\frac{n-3}{2} \in \mathbb{N} \; \text{and} \; \alpha \in \mathbb{N}  \\
\text{or} \\
\frac{n-3}{2} \notin \mathbb{N} \; \text{and} \;
\alpha \leq \left[ \frac{n-3}{2} \right]
\end{array}
\right\}
: \;
\left| \p_{\rho}^{\alpha} \mathcal{A}^{n-3} \right|  \lesssim |\sigma|^{-\alpha} ( \epsilon^{\frac{1}{2}} |\sigma|^{\frac{1}{2}} |z| )^{n-3}
+ \epsilon^{- \alpha}  ( \epsilon^{\frac{1}{2}} |\sigma|^{\frac{1}{2}} |z|)^{n-3} \cdot
\label{Eqn:DerA}
}
The second term of the right-hand side of (\ref{Eqn:DerA}) appears when the derivative hits the factor of $\mathcal{A}$ bounded by
$\epsilon^{\frac{1}{2}}$; the first term appears when the derivative hits the other factors of $\mathcal{A}$. (\ref{Eqn:DerA}) is constantly used
in the sequel.\\
We define $\breve{A}_{l}$ by

\EQQARR{
l=1: \;  \breve{\mathcal{A}}_{1} := \mathcal{A}(-\rho + \sigma + |z|)^{-\frac{1}{2}} \\
l=2: \;  \breve{\mathcal{A}}_{2} := \mathcal{A}(\rho + \sigma - |z|)^{-\frac{1}{2}} \\
l=3: \; \breve{\mathcal{A}}_{3} := \mathcal{A}(\rho - \sigma + |z|)^{-\frac{1}{2}}
}
In the sequel we use the estimate $\breve{\mathcal{A}}_{l} \lesssim |\sigma|^{\frac{1}{2}} |z|$ on the region of integration.\\
\\
\underline{$n \in \{8,9\}$}: Integrating by parts once w.r.t $\rho$, we have integrals that are bounded by the finite integral $X$
defined by

\EQQARRLAB{
X  & := \epsilon^{\frac{n-3}{2}}
\int_{|\rho| \gtrsim \epsilon}
\frac{(t_0-t')^{\frac{1}{3}}}{|\rho|^{\frac{1}{3}}} \frac{1}{\left\langle \frac{\rho}{ (t_0 - t')^{\frac{1}{4}} } \right\rangle^{\frac{n}{3}}}
\underset{ a^{\frac{1}{4}} |\sigma| \lesssim |\rho|} \int \frac{\langle \sigma \rangle^{\frac{n-1}{2}}}
{  \left\langle \frac{\sigma}{ (t^{''} - t_0)^{\frac{1}{4}}} \right\rangle^{\frac{n}{3}} } \, d \sigma \, d \rho \\
& \lesssim (t_0 -t')^{\frac{n}{4}-} (t^{''} - t_0)^{\frac{n}{4}-} \cdot
\label{Eqn:X}
}
\\
\underline{$n = 7$}: Integrating by parts twice w.r.t $\rho$, we have integrals that are bounded by
finite integrals that are similar to $X$. Hence we get the same bounds that have the same form as (\ref{Eqn:X}).\\
\\
\underline{$n=6$}: Integrating by parts once w.r.t $\rho$, we have integrals that are bounded by finite integrals that are
similar to $X$ if the derivative does not hit the factor of $\mathcal{A}$ bounded by $\epsilon^{\frac{1}{2}}$ or
$\comp{(\tilde{\chi}^{2}_{\mu_{5}^{3}})}$. One can also estimate
thanks to  (\ref{Eqn:AreaTrick}) the integral which appears if the derivative hits $\comp{(\tilde{\chi}^{2}_{\mu_{5}^{3}})}$. It remains
to estimate the integral which appears when the derivative hits the factor of $\mathcal{A}$ bounded by $\epsilon^{\frac{1}{2}}$. Integrating another time
w.r.t $\rho$, there are again integrals bounded by finite integrals similar to $X$ and the integral which appears when the
derivative hits twice the factor of $\mathcal{A}$ bounded by $\epsilon^{\frac{1}{2}}$. We only deal with the case $l=2$. The other cases are
treated similarly and left to the reader. The integral is bounded by $Y$ defined by

\EQQARR{
Y &  :=  \underset{ \substack{ (\rho, \sigma) \in \Rec \\ |\rho| \gtrsim \epsilon} } \int
\frac{(t_0-t')^{\frac{2}{3}}}{|\rho|^{\frac{2}{3}}} \frac{1}{ \left \langle \frac{\rho}{(t_0-t')^{\frac{1}{4}}} \right \rangle^{\frac{n}{3}} }
\underset{\substack{ | \rho + \sigma - |z|| \lesssim \epsilon \\ a^{\frac{1}{4}} |\sigma| \lesssim |\rho|}}
\int
\left(
\begin{array}{l}
\frac{|\sigma| |\rho| }{|z|} \;
\frac{|\rho + \sigma - |z||^{-\frac{1}{2}}}{ \left \langle \frac{\sigma}{(t^{''}-t_0)^{\frac{1}{4}}} \right \rangle^{\frac{n}{3}}}
\left( \frac{\breve{\mathcal{A}_{2}}}{|z|  } \right)^{n-3}
\end{array}
\right)
d \sigma \;  d \rho  \\
& \lesssim \underset{|\rho| \gtrsim \epsilon}  \int
\frac{(t_0-t')^{\frac{2}{3}}}{|\rho|^{\frac{2}{3}}} \frac{\epsilon^{\frac{1}{2}} }{ \left \langle \frac{\rho}{(t_0-t')^{\frac{1}{4}}} \right \rangle^{\frac{n}{3}} }
\underset{a^{\frac{1}{4}} |\sigma| \lesssim |\rho|}  \max \frac{|\sigma|^{\frac{n-1}{2}}}
{\left \langle \frac{\sigma}{(t^{''} - t_0)^{\frac{1}{4}}} \right \rangle^{\frac{n}{3}} } \; d \rho \\
& \lesssim (t_0 -t')^{\frac{n}{4}-} (t^{''} - t_0)^{\frac{n}{4}-}
}
\\
\underline{$n=5$}: Integrating by parts twice w.r.t $\rho$, one can estimate
integrals by bounds that have the same form as (\ref{Eqn:X}), using a similar procedure
to the case $n=6$. One has also to estimate the boundary term that appears when we apply the second
integration by parts to the integral that appears when the derivative hits the factor of $\mathcal{A}$
bounded by $\epsilon^{\frac{1}{2}}$, more precisely

\EQQARR{
\underset{ \substack{(\rho, \sigma) \in \partial \Rec} } \int
e^{i \tilde{\phi}} \frac{(t_0-t')^{\frac{2}{3}}}{\rho^{\frac{2}{3}}}
[J^{b}_{\rho:far, \p \Rec:l}] \left( \frac{\breve{\mathcal{A}}_{l}}{\mathcal{A}} \right)^{n-3}
n_{\rho} \; d s \;  d S_{\omega'}
}
(Here $n_{\rho}$ is the $\rho$ component of the normal $n$). It is bounded by $Y$, defined by

\EQQARR{
Y & :=  \int_{|\rho| \gtrsim \epsilon}
\frac{(t_0-t')^{\frac{2}{3}}}{|\rho|^{\frac{2}{3}}} \frac{1}{ \left \langle \frac{\rho}{(t_0-t')^{\frac{1}{4}}} \right \rangle^{\frac{n}{3}} }
\underset{a^{\frac{1}{4}} |\sigma| \lesssim |\rho|} \max \left( \frac{|\sigma|^{\frac{n-1}{2}}}
{ \left \langle \frac{\sigma}{(t^{''} - t_0)^{\frac{1}{4}}} \right \rangle^{\frac{n}{3}}  } \right)  \; d \rho \\
& \lesssim (t_0 -t')^{\frac{n}{4}-} (t^{''} - t_0)^{\frac{n}{4}-}
}



\subsubsection{Conclusion}We conclude that there exists a small positive constant $c > 0$ such that

\EQQARRLAB{
\coefinv |J_{\rho}^{b}| \lesssim \frac{1}{|t''-t'|^{c}} \cdot
\label{Eqn:EstKrho}
}
\\

\subsection{Estimate of $J^{b}_{\sigma}$}

Write $J^{b}_{\sigma}$ as the sum of $J^{b}_{\sigma:cl}$  and $J^{b}_{\sigma:far}$.

\subsubsection{Estimate of $J^{b}_{\sigma:cl}$} Notice again that on the region of integration of
$J^{b}_{\sigma:cl}$, $\mathcal{A} \lesssim |\sigma| |z|$. We have

\EQQARRLAB{
|J^{b}_{\sigma:cl}| \quad \lesssim \int_{|\sigma| \lesssim \epsilon} \frac{|\sigma|^{n-2}} {\left\langle \frac{\sigma}{(t^{''} - t_0)^{\frac{1}{4}}}
\right\rangle^{\frac{n}{3}}} \int_{|\rho| \lesssim a^{\frac{1}{4}} |\sigma|} \frac{1}{\left\langle \frac{\rho}{(t_0 - t')^{\frac{1}{4}}}
\right\rangle^{\frac{n}{3}}} \; d \rho \; d \sigma
\lesssim \frac{(t_0 - t')^{\frac{1}{4}} \epsilon^{n-1} }{ \left\langle \frac{\epsilon}{(t^{''} - t_0)^{\frac{1}{4}}} \right\rangle^{\frac{n}{3}}} \cdot
\label{Eqn:Jsigmacl1}
}

\subsubsection{Estimate of $J^{b}_{\sigma:far}$} Write $J^{b}_{\sigma:far}$ as the sum of the
$J^{b}_{\sigma:far,\p \Rec:l}$ and $J^{b}_{\sigma:far,\p \Rec:far}$.  \\
\\
\frame{Estimate of $J^{b}_{\sigma:far,\p \Rec:far}$}  \\
\\
We estimate $J^{b}_{\sigma:far, \p \Rec:far}$  by integrating the phase w.r.t $\sigma$.
To this end we define (with $(p,q) \in \mathbb{N}^{2}$)

\EQQARR{
\mathcal{Q}_{p,q} & := \left\{ f \in C^{\infty} \left( \mathbb{R} - \{0\} \right): \; |\p^{\alpha} f(\sigma)|
\lesssim \frac{(t_0 - t')^{\frac{p}{3}}}{(a |\sigma|)^{\frac{p}{3}}} \frac{1}{|\sigma|^{q + \alpha}}, \alpha \in \mathbb{N} \right\} \cdot
}
Given $\vec{r}= (r_1,...,r_5) \in \mathbb{N}^{5}$  and $f \in \mathcal{Q}_{p,r_1}$ we define $ K_{\sigma,\vec{r}} (f)
:= \int e^{i \tilde{\phi}} X_{\sigma,\vec{r}}(f) \; d \rho \; d \sigma$ with

\EQQARR{
X_{\sigma,\vec{r}}(f) & :=  f(\sigma) \,
\p_{\sigma}^{r_2} \left( \comp{\eta} \left( \frac{\sigma}{\epsilon} \right) \, \Omega_{\sigma} \, \comp{\psi}
 \right)
 \tilde{I} \left( \frac{\rho}{(t_0 -t')^{\frac{1}{4}}} \right)  \times \\
& \p_{\sigma}^{r_3} \left( \tilde{I} \left( \frac{\sigma}{(t^{''} - t_0)^{\frac{1}{4}}} \right) \right) \p_{\sigma}^{r_4}
\left( \bar{\Omega}_{4} \mathbf{1}_{\Rec} \frac{\rho \sigma}{|z|} \left( \frac{\mathcal{A}}{|z|} \right)^{n-3}
\right) \p_{\sigma}^{r_5} \left( \comp{(\tilde{\chi}^{2}_{\mu_{5}^{3}})} \right) \cdot
}
Assume that $r_5 = 0$. Integrating by parts w.r.t $\sigma$, we see that there exist $g \in \mathcal{Q}_{p+1,r_{1}+1}$ and $h \in \mathcal{Q}_{p+1,r_1} $

\EQQARR{
K_{\sigma,\vec{r}}(f) & =
 \int  \frac{\p_{\sigma} e^{i \tilde{\phi}}}{ _{ i \p_{\sigma} \tilde{\phi}}}
X_{\sigma,\vec{r}} (f) \, d \sigma \, d \rho  \\
& = - \int e^{i \tilde{\phi}} \p_{\sigma}
\left( \frac{1}{ _{i \p_{\sigma} \tilde{\phi}}} X_{\sigma,\vec{r}} \right) \, d \sigma \, d \rho \\
& = i \left(  K_{\sigma, \overrightarrow{r_{1,+}}} (g) + \sum \limits_{j=2}^{5} K_{\sigma, \overrightarrow{r_{j,+}}} (h) \right) \cdot
}
Hence, since $K_{\sigma,\vec{0}}(1)=J_{\sigma:far, \p R:far}$, we see, by iterating over $p$, that we are reduced to estimate

\begin{itemize}
\item $K_{\sigma,\vec{r}}(f)$ for $\vec{r} \in \mathcal{P}_{\bar{p}}$ such that $r_5=0$  and
$f \in \mathcal{Q}_{\bar{p},r_1}$ and
\item $K_{\sigma,\vec{r}}(f)$ for $\vec{r} \in \mathcal{P}_p$, $1 \leq p \leq \bar{p}$, such that $r_5 =1$ and $f \in \mathcal{Q}_{p,r_1}$,
\end{itemize}
Notice again that on the region of integration of $K_{\sigma,\vec{r}}(f)$, $\mathcal{A} \lesssim |\sigma| |z|$.
It is worth choosing $\epsilon$ by considering only the term $K_{\sigma,\overrightarrow{r}}(f)$ with
$\overrightarrow{r}:=(\bar{p},0,..,0)$ for $\bar{p}$ large enough to assure the integrability of
(\ref{Eqn:ComputKsigma}): the estimate of this term is the same as the one for which the integration by parts hits the derivative of the phase $\bar{p}$ times. \\
We have

\EQQARRLAB{
\left| K_{\sigma,\overrightarrow{r}}(f) \right| & \lesssim \int_{|\sigma| \gtrsim \epsilon} \frac{(t_0 - t')^{\frac{\bar{p}}{3}}}{a^{\frac{\bar{p}}{3}}
|\sigma|^{\frac{4 \bar{p}}{3}}} \frac{|\sigma|^{n-2}}{ \left\langle \frac{\sigma}{(t^{''} - t_0)^{\frac{1}{4}}} \right\rangle^{\frac{n}{3}}}
\int_{|\rho| \lesssim a^{\frac{1}{4}} |\sigma|} \frac{1}{ \left\langle \frac{\rho}{(t_0 - t')'^{\frac{1}{4}}} \right\rangle^{\frac{n}{3}}} \; d \rho
\; d \sigma \\
& \\
& \lesssim \frac{(t_0 - t')^{\frac{\bar{p}}{3} + \frac{1}{4}}}{a^{\frac{\bar{p}}{3}}
\epsilon^{\frac{4\bar{p}}{3} - n +1}
\left\langle \frac{\epsilon}{(t^{''}-t_0)^{\frac{1}{4}}} \right\rangle^{\frac{n}{3}}
} \cdot
\label{Eqn:Modeps2sigma1}
}
So optimizing in $\epsilon$ the upper bound of (\ref{Eqn:Jsigmacl1}) and (\ref{Eqn:Modeps2sigma1}), we
find $\epsilon \sim (t^{''} -t_0)^{\frac{1}{4}}$. \\
Choose $\beta = \epsilon$. We now estimate the other terms with this value of $\epsilon$ and $\beta$. \\
If $r_5 =0$ then

\EQQARRLAB{
|K_{\sigma,\vec{r}} (f)| &  \lesssim  \underset{|\sigma| \gtrsim \epsilon} \int \;
\frac{(t_0 -t')^{\frac{\bar{p}}{3}}}{(a |\sigma|)^{\frac{\bar{p}}{3}}}
 \frac{1}{|\sigma|^{r_1+ r_2 }}
\frac{1}{(t^{''} - t_0)^{\frac{r_3}{4}}}
\frac{|\sigma|^{n-2}}{ \left\langle \frac{\sigma}{(t^{''}-t_0)^{\frac{1}{4}}} \right\rangle^{\frac{n + r_3}{3}}}
\frac{1}{\epsilon^{r_4}}
\underset{ |\rho| \lesssim a^{\frac{1}{4}} |\sigma|} \int
\frac{1}{\left\langle \frac{\rho}{(t_0 - t')^{\frac{1}{4}}} \right\rangle^{\frac{n}{3}}}
\; d \rho \; d \sigma  \\
& \lesssim  \underset{|\sigma| \approx \epsilon} \int \;
\frac{(t_0 -t')^{\frac{\bar{p}}{3}}}{(a |\sigma|)^{\frac{\bar{p}}{3}}} \frac{1}{|\sigma|^{\bar{p}}}
 \frac{ |\sigma|^{n-2}}{\left\langle
 \frac{\sigma}{(t^{''} -t_0)^{\frac{1}{4}}} \right\rangle^{\frac{n+r_3}{3}}} (t_0 - t')^{\frac{1}{4}}
\; d \sigma \\
& \lesssim \frac{(t^{''} -t_{0})^{\frac{\bar{p}}{3} + \frac{1}{4}}  }
{ a^{\frac{\bar{p}}{3}}  \epsilon^{ \frac{ 4 \bar{p}}{3} - n + 1 } \left\langle \frac{\epsilon}{(t^{''} - t_0)^{\frac{1}{4}}} \right\rangle^{\frac{n}{3}} }:
\label{Eqn:ComputKsigma}
}
this is the same estimate as (\ref{Eqn:Modeps2sigma1}); and if $r_5 =1$ then from (\ref{Eqn:AreaTrick}) we see that

\EQQARRLAB{
|K_{\sigma,\vec{r}} (f)|  & \lesssim \underset{ \substack{(\rho,\sigma) \in \Rec \\ \sigma \gtrsim \epsilon }} \int
\frac{(t_0 - t')^{\frac{\bar{p}}{3}}}{(a \sigma)^{\frac{\bar{p}}{3}}} \frac{1}{\sigma^{r_1 +  r_2 }}
\frac{1}{(t^{''} - t_0)^{\frac{r_3}{4}}} \frac{1}{\epsilon^{r_4}}
\frac{\sigma \rho }{|z|} \left( \frac{\mathcal{A}}{|z|} \right)^{n-3}
\left| \p_{\sigma}^{1} \left( \comp{(\tilde{\chi}^{2}_{\mu_{5}^{3}})} \right) \right|
 d \rho \; d \sigma \; d S_{\omega'}  \\
& \\
 & \lesssim (t^{''} - t_0)^{\frac{1}{4}} \cdot
\label{Eqn:Est2j5Eq1}
}

\frame{Estimate of $J_{\sigma:far,\p \Rec:l}^{b}$ for $n >9$} \\
\\
If $n > 9$, then, ignoring again the compact support of $\comp{(\tilde{\chi}^{2}_{\mu_{5}^{3}})}$,
$ |J_{\sigma:far,\p \Rec:l}^{b}| $ is bounded by an integrable expression thanks to the weights in (\ref{Eqn:EsttildeI}); more precisely

\EQQARR{
|J^{b}_{\sigma:far,\p \Rec:l}| &  \lesssim  \epsilon^{\frac{n-3}{2}}
\int_{|\sigma| \gtrsim \epsilon}
\frac{1} {\left\langle \frac{\sigma}{ (t^{''} - t_0)^{\frac{1}{4}}} \right\rangle^{\frac{n}{3}} }
\underset{a^{\frac{1}{4}}|\sigma| \gtrsim |\rho| }  \int
\frac{|\rho|^{\frac{n-1}{2}}} { \left\langle \frac{\rho}{ (t_0 - t')^{\frac{1}{4}} } \right\rangle^{\frac{n}{3}}}  \, d \rho \, d \sigma \\
&  \lesssim (t_0 - t')^{\frac{n}{4}-} (t''- t_0)^{\frac{n}{4}-},
}
using $\mathcal{A} \lesssim \epsilon^{\frac{1}{2}} |\rho|^{\frac{1}{2}} |z|$. \\
\\
\frame{Estimate of $J^{b}_{\sigma:far,\p \Rec:l}$ for $ 5 \leq n \leq 9$} \\
\\
If $ 5 \leq n \leq 9$ the integral above is not integrable. Hence, again, in order to get integrability, we integrate the phase
by parts a small number of times w.r.t $\sigma$.\\
Observe that the estimate $\breve{\mathcal{A}}_{l} \lesssim |\rho|^{\frac{1}{2}} |z|$ on the region
of integration of $J_{\sigma:far,\p \Rec:l}^{b}$, with $\breve{\mathcal{A}}_{l}$ defined in the previous subsection. Observe that
if we apply $\p_{\sigma}$ to $\mathcal{A}$ on the region of integration of $J_{\sigma:far,\p \Rec:l}^{b}$, then

\EQQARRLAB{
\left.
\begin{array}{l}
\frac{n-3}{2} \in \mathbb{N} \; \text{and} \; \alpha \in \mathbb{N}  \\
\text{or} \\
\frac{n-3}{2} \notin \mathbb{N} \; \text{and} \; \alpha \leq \left[ \frac{n-3}{2} \right]
\end{array}
\right\}: \;
\left| \p_{\sigma}^{\alpha} \mathcal{A}^{n-3} \right|  \lesssim \rho^{-\alpha} ( \epsilon^{\frac{1}{2}} \rho^{\frac{1}{2}} |z| )^{n-3}
+ \epsilon^{-\alpha}  ( \epsilon^{\frac{1}{2}} \rho^{\frac{1}{2}} |z|)^{n-3} \cdot
\label{Eqn:DerAsigm}
}
The second term of the right-hand side of (\ref{Eqn:DerAsigm}) appears when the derivative hits the factor of $\mathcal{A}$ bounded by
$\epsilon^{\frac{1}{2}}$; the first term appears when the derivative hits the other factors of $\mathcal{A}$. These estimate are constantly used
in the sequel.\\
\\
\underline{$n \in \{8,9\}$}: Integrating by parts once w.r.t $\sigma$, we have integrals that are bounded by the finite integral $X$
defined by

\EQQARRLAB{
X & := \epsilon^{\frac{n-3}{2}} \int_{|\sigma| \gtrsim \epsilon} \frac{(t_0-t')^{\frac{1}{3}}}{(a |\sigma|)^{\frac{1}{3}}}
\frac{1}{\left \langle \frac{\sigma}{(t'' - t_0)^{\frac{1}{4}}} \right \rangle^{\frac{n}{3}}}
\underset{ a^{\frac{1}{4}} |\sigma| \gtrsim |\rho| } \int
\frac{\langle \rho \rangle^{\frac{n-1}{2}}} {\left \langle \frac{\rho}{(t_0-t')^{\frac{1}{4}}} \right \rangle^{\frac{n}{3}}}
\; d \rho \;  d \sigma \\
& \lesssim (t_0 - t')^{\frac{n}{4}-} (t'' -t_0)^{\frac{n}{4}-}
\label{Eqn:Xsigm}
}
\\
\underline{$n = 7$}: Integrating by parts twice w.r.t $\sigma$, we have integrals that are bounded by
finite integrals that are similar to $X$. Hence we get bounds that have the same form as (\ref{Eqn:Xsigm}).\\
\\
\underline{$n = 6$}: Integrating by parts once w.r.t $\sigma$, we have integrals that are bounded by finite integrals that are
similar to $X$ if the derivative does not hit the factor of $\mathcal{A}$ bounded by $\epsilon^{\frac{1}{2}}$ or $\comp{(\tilde{\chi}^{2}_{\mu_{5}^{3}})}$.  One
can also estimate thanks to (\ref{Eqn:AreaTrick}) the integral which appears if the derivative hits $\comp{(\tilde{\chi}^{2}_{\mu_{5}^{3}})}$. It remains
to estimate the integral which appears when the derivative hits the factor of $\mathcal{A}$ bounded by $\epsilon^{\frac{1}{2}}$. Integrating another time
w.r.t $\sigma$, there are again integrals bounded by finite integrals similar to $X$ and the integral which appears when the
derivative hits twice the factor of $\mathcal{A}$ bounded by $\epsilon^{\frac{1}{2}}$. We only deal with the case $l=2$. The other cases are
treated similarly and left to the reader. The integral is bounded by $Y$ defined by

\EQQARR{
Y &  :=  \underset{ \substack{ (\rho, \sigma) \in \Rec \\ |\sigma| \gtrsim \epsilon} } \int
\frac{(t''-t_0)^{\frac{2}{3}}}{(a |\sigma|)^{\frac{2}{3}}} \frac{1}{ \left \langle \frac{\sigma}{(t''-t_0)^{\frac{1}{4}}} \right \rangle^{\frac{n}{3}} }
\underset{\substack{ | \rho + \sigma - |z|| \lesssim \epsilon \\ a^{\frac{1}{4}} |\sigma| \gtrsim |\rho|}}
\int
\left(
\begin{array}{l}
\frac{|\sigma| |\rho| }{|z|} \;
\frac{|\rho + \sigma - |z||^{-\frac{1}{2}}}{ \left \langle \frac{\rho}{(t_0-t')^{\frac{1}{4}}} \right \rangle^{\frac{n}{3}}}
\left( \frac{\breve{\mathcal{A}}_{2}}{|z|  } \right)^{n-3}
\end{array}
\right)
d \sigma \;  d \rho  \\
& \lesssim \underset{|\sigma| \gtrsim \epsilon}  \int
\frac{(t_0-t')^{\frac{2}{3}}}{(a |\sigma|)^{\frac{2}{3}}} \frac{\epsilon^{\frac{1}{2}} }{ \left \langle \frac{\sigma}{(t''-t_0)^{\frac{1}{4}}} \right
\rangle^{\frac{n}{3}} }
\underset{ a^{\frac{1}{4}}|\sigma| \gtrsim |\rho|}  \max \frac{|\rho|^{\frac{n-1}{2}}}
{\left \langle \frac{\rho}{(t_0 - t')^{\frac{1}{4}}} \right \rangle^{\frac{n}{3}} } \; d \sigma \\
& \lesssim (t_0 -t')^{\frac{n}{4}-} (t^{''} - t_0)^{\frac{n}{4}-}
}
\\
\underline{$n=5$}: Integrating by parts twice w.r.t $\sigma$, one can estimate
integrals by bounds that have the same form as (\ref{Eqn:Xsigm}), using a similar procedure
to the case $n=6$. One has also to estimate the boundary term that appears when we apply the second
integration by parts to the integral that appears when the derivative hits the factor of $\mathcal{A}$
bounded by $\epsilon^{\frac{1}{2}}$ at the first integration by parts, more precisely

\EQQARR{
\underset{ \substack{(\rho, \sigma) \in \partial \Rec} } \int
e^{i \tilde{\phi}} \frac{(t_0-t')^{\frac{2}{3}}}{(a \sigma)^{\frac{2}{3}}}
[J^{b}_{\sigma:far, \p \Rec:l}] \left( \frac{\breve{\mathcal{A}_l}}{\mathcal{A}} \right)^{n-3}
n_{\sigma} \; d s \;  d S_{\omega'}
}
(Here $n_{\sigma}$ is the $\sigma$ component of the normal $n$). It is bounded by $Y$, defined by

\EQQARR{
Y & :=  \int_{|\sigma| \gtrsim \epsilon} \frac{(t_0-t')^{\frac{2}{3}}}{(a|\sigma|)^{\frac{2}{3}}} \frac{1}{ \left \langle \frac{\sigma}{(t''-t_0)^{\frac{1}{4}}} \right \rangle^{\frac{n}{3}} }
\underset{a^{\frac{1}{4}} |\sigma| \gtrsim |\rho|} \max \left( \frac{|\rho|^{\frac{n-1}{2}}}
{ \left \langle \frac{\rho}{(t_0 - t')^{\frac{1}{4}}} \right \rangle^{\frac{n}{3}}  } \right)  \; d \sigma \\
&  \lesssim (t_0 -t')^{\frac{n}{4}-} (t^{''} - t_0)^{\frac{n}{4}-}.
}

\subsubsection{Conclusion}We conclude that there exists a small positive constant $c>0$ such that

\EQQARR{
\coefinv |J_{\sigma}^{b}| \lesssim \frac{1}{|t''-t'|^{c}} \cdot
}

\subsection{Estimate of $J^{a}$}
Write $J^{a}$ as the sum of $J^{a}_{cl}$ and $J^{a}_{far}$.

\subsubsection{Estimate of $J^{a}_{cl}$}

Notice again that on the region of integration of $J_{cl}^{a}$, $\mathcal{A} \lesssim |\sigma| |z|$. We have

\EQQARRLAB{
\label{Eqn:Acl}
|J_{cl}^{a}| \quad \lesssim  \underset{ \substack{|\rho| \lesssim \epsilon \\ |\rho| \approx |\sigma| }} \int \;
\frac{|\sigma|^{n-2}}{ \left\langle \frac{\sigma}{(t'' - t_0)^{\frac{1}{4}}} \right\rangle^{\frac{n}{3}}}
\frac{1} { \left\langle \frac{\rho}{(t_0 - t')^{\frac{1}{4}}} \right\rangle^{\frac{n}{3}}}  \;  d \sigma \; d \rho \quad
\lesssim  \frac{\epsilon^n} { \left\langle \frac{\epsilon}{(t_0 -t')^{\frac{1}{4}}} \right\rangle^{\frac{n}{3}}
\left\langle \frac{\epsilon}{(t''-t_0)^{\frac{1}{4}}} \right\rangle^{\frac{n}{3}} } \cdot
}

\subsubsection{Estimate of $J^{a}_{far}$}

We integrate the phase by parts w.r.t $\rho'$. \\
We define (with $(p,q) \in \mathbb{N}^{2}$)

\EQQARR{
\mathcal{Q}_{p,q}:= \left\{ f \in C^{\infty}(\mathbb{R}^{2} - \{ (0,0) \}): \;
|\p^{\alpha}_{\rho'} f(\rho,\sigma)| \lesssim \frac{(t_0  -t')^{\frac{p}{3}}}{|\rho|^{\frac{p}{3} + q + \alpha }}, \; \alpha \in \mathbb{N} \right\} \cdot
}
Given $\vec{r}= (r_1,...,r_5) \in \mathbb{N}^{5}$ and  $f \in \mathcal{Q}_{p,r_1}$ we define $ K_{\rho',\vec{r}} (f)
:= \int e^{i \tilde{\phi}} X_{\rho',\vec{r}}(f) \; d \rho \; d \sigma$ with

\EQQARR{
X_{\rho',\vec{r}}(f) & :=  f(\rho,\sigma) \,
\p_{\rho'}^{r_2} \left( \comp{\eta} \left( \frac{\rho'}{\epsilon} \right) \; \psi  \;
 \frac{\rho \sigma}{|z|} \left( \frac{\mathcal{A}}{|z|} \right)^{n-3}
\right)
\p_{\rho'}^{r_3} \left(
\tilde{I} \left( \frac{\rho}{(t_0 -t')^{\frac{1}{4}}} \right) \right) \;
\p_{\rho'}^{r_4} \left(  \tilde{I} \left( \frac{\sigma}{(t^{''} - t_0)^{\frac{1}{4}}}
\right) \right)
\p_{\rho'}^{r_5} \left( \comp{(\tilde{\chi}^{2}_{\mu_{5}^{3}})} \right) \cdot
}
Assume that $r_5=0$. Integrating by parts w.r.t $\rho'$, we see that there exist $g \in \mathcal{Q}_{p+1,r_{1} +1}$ and
$h \in \mathcal{Q}_{p+1,r_{1}} $ such that

\EQQARR{
K_{\rho',\vec{r}}(f) & =
\int  \frac{\p_{\rho'} e^{i \tilde{\phi}}}{ _{ i \p_{\rho'} \tilde{\phi}}} X_{\rho',\vec{r}} (f) \, d \sigma \, d \rho  \\
& =
- \int e^{i \tilde{\phi}} \p_{\rho'}
\left( \frac{1}{ _{i \p_{\rho'} \tilde{\phi}}} X_{\rho',\vec{r}}(f) \right) \, d \sigma \, d \rho \\
& = i \left(  K_{\rho', \overrightarrow{r_{1,+}}} (g) +  \sum \limits_{j=2}^{5} K_{\rho', \overrightarrow{r_{j,+}}} (h) \right),
}
Since $K_{\rho',\vec{0}}(1)=J^{a}_{far}$, we see, by iterating over $p$, that we are reduced to estimate

\begin{itemize}
\item $K_{\rho',\vec{r}}(f)$ for $\vec{r} \in \mathcal{P}_{\bar{p}}$ such that $r_5=0$ and  $f \in \mathcal{Q}_{\bar{p},r_1}$ and
\item $K_{\rho',\vec{r}}(f)$ for $\vec{r} \in \mathcal{P}_p$, $1 \leq p \leq \bar{p}$, such that $r_5 =1$ and $f \in \mathcal{Q}_{p,r_1}$.
\end{itemize}
Notice that on the region of integration of $K_{\rho',\vec{r}}(f)$,  $| \p^{\alpha}_{\rho'} \mathcal{A} | \lesssim |\rho|^{-\alpha} (|\sigma| |z|) $. \\
\\
\underline{$n > 6$}: Let $\bar{p} := \left[ \frac{n+3}{4} \right] +1 $ so that (\ref{Eqn:Modeps2rhoprother}) is integrable.
It is worth choosing $\epsilon$ by considering only the term $K_{\rho',\overrightarrow{r}}(f)$ with
$\overrightarrow{r}:=(\bar{p},0,..,0)$: the estimate of this term is the same as the one for which the integration by parts hits the derivative of the phase $\bar{p}$ times. We get

\EQQARRLAB{
\label{Eqn:Modeps2rhopr}
| K_{\rho',\vec{r}}  (f)|
\lesssim   \underset{|\rho| \gtrsim \epsilon} \int \; \frac{(t_0 -t')^{\frac{\bar{p}}{3}}}{|\rho|^{\frac{4 \bar{p}}{3}}}
\frac{1}{\left\langle \frac{\rho}{(t_0 - t')^{\frac{1}{4}}} \right\rangle^{\frac{n}{3}}}
\underset{ |\rho - \sigma| \leq |z|} \int
\frac{ |\sigma|^{n-1}}{|z|}
\frac{1}{ \left\langle \frac{\sigma}{(t^{''}-t_0)^{\frac{1}{4}}} \right\rangle^{\frac{n}{3}}}
\, d \sigma \, d \rho  \lesssim \frac{(t_0 -t')^{\frac{\bar{p}}{3}} }{\epsilon^{\frac{4 \bar{p}}{3} - n}\left \langle \frac{\epsilon}{(t_0 -t')^{\frac{1}{4}}} \right \rangle^{\frac{n}{3}}
\left \langle \frac{\epsilon}{(t^{''} -t_0)^{\frac{1}{4}}} \right \rangle^{\frac{n}{3}}
}
}
Optimizing $\epsilon$ in (\ref{Eqn:Modeps2rhopr}) and (\ref{Eqn:Acl}) we find $\epsilon \approx  (t_0 - t')^{\frac{1}{4}}$.
It remains to estimate the other terms with this value of $\epsilon$. From $\vec{r} \in \mathcal{P}_{\bar{p}}$
and (\ref{Eqn:AreaTrick}) we see that if $r_{5}=0$ then

\EQQARRLAB{
\label{Eqn:Modeps2rhoprother}
|K_{\rho',\vec{r}}  (f)|  \\
\lesssim  \underset{|\rho| \gtrsim \epsilon} \int \;
\frac{(t_0 -t')^{\frac{\bar{p}}{3}}}{|\rho|^{\frac{\bar{p}}{3}}}
 \frac{1}{|\rho|^{r_1+ r_2}}
\frac{1}{(t_0 - t')^{\frac{r_3}{4}}}
\frac{1}{\left\langle \frac{\rho}{(t_0 - t')^{\frac{1}{4}}} \right\rangle^{\frac{n + r_3}{3}}}
\underset{ |\rho - \sigma| \leq |z|} \int
\frac{ |\sigma|^{n-1}}{|z|}
\frac{1}{(t'' - t_0)^{\frac{r_4}{4}}}
\frac{1}{ \left\langle \frac{\sigma}{(t^{''}-t_0)^{\frac{1}{4}}} \right\rangle^{\frac{n + r_4}{3}}}
\, d \sigma \, d \rho  \\
\lesssim \underset{|\rho| \approx \epsilon} \int \;
\frac{(t_0 -t')^{\frac{\bar{p}}{3}}}{\rho^{\frac{\bar{p}}{3} + r_1 + r_2 + r_3} \left\langle \frac{\rho}{(t_0 -t')^{\frac{1}{4}}} \right\rangle^{\frac{n + r_3}{3}}}
\frac{|\rho|^{n-1}}{
(t'' - t_0)^{\frac{r_4}{4}}
\left\langle \frac{\rho}{(t'' -t_0)^{\frac{1}{4}}} \right\rangle^{\frac{n + r_4}{3}}
}
\; d \rho \\
\lesssim (t_0 - t')^{\frac{n}{4}-} (t'' - t_0)^{\frac{n}{4}-} ;
}
and if $r_5 =1$ then

\EQQARRLAB{
\label{Eqn:Modeps2rhoprother2}
|K_{\rho',\vec{r}} (f)|  & \lesssim  \underset{ \substack{(\rho,\sigma) \in \Rec \\ \rho \gtrsim \epsilon }} \int
\frac{(t_0 - t')^{\frac{\bar{p}}{3}}}{\rho^{^{\frac{\bar{p}}{3}}}} \frac{1}{\rho^{r_1 +  r_2 }}
\frac{1}{(t_0 - t')^{\frac{r_3}{4}}} \frac{1}{(t'' - t_0)^{\frac{r_4}{4}}}
\frac{\sigma \rho }{|z|} \left( \frac{\mathcal{A}}{|z|} \right)^{n-3}
\left| \p_{\rho'}^{1} \left( \comp{(\tilde{\chi}^{2}_{\mu_{5}^{3}})} \right) \right|
  d \sigma \; d \rho \; d S_{\omega'} \\
& \lesssim  (t_0 - t')^{\frac{n}{4}-} (t'' - t_0)^{\frac{n}{4}-} \cdot
}
\\
\underline{$n \in \{ 5,6 \}$}: choose $\epsilon := (t_0 -t')^{\frac{1}{8}}$. If $r_5 = 0$ then dividing into the two cases
$ \epsilon \lesssim (t'' - t_0)^{\frac{1}{4}}$ and $ \epsilon \gg (t'' - t_0)^{\frac{1}{4}}$ we can easily estimate the integral
at the second line of (\ref{Eqn:Modeps2rhoprother}). If $r_5=1$ then the procedure to estimate $K_{\rho',\vec{r}} (f)$ is similar
to (\ref{Eqn:Modeps2rhoprother2}). Again we find bounds of the form $ (t_0 - t')^{\frac{n}{4}-} (t'' - t_0)^{\frac{n}{4}-}$.


\subsubsection{Conclusion} We conclude that there exists a small positive constant $c>0$ such that

\EQQARR{
\coefinv |J^{a}| \lesssim \frac{1}{|t''-t'|^{c}} \cdot
}

\section{Modified fundamental solution: estimates of its derivatives}

In this section we prove (\ref{Eqn:EsttildeI}). \\
\\
Let $(e_1,...,e_n)$ be the standard orthonormal basis of $\mathbb{R}^{n}$. Rotating if necessary
we may assume WLOG that $x = |x| e_n$. In view of (\ref{Eqn:EstI}) we may assume WLOG that $|x| \gg 1$.
\\
We then estimate $ (\partial^{\alpha} \tilde{I})^{|\xi_{st}|:far}(x) $
for $\alpha \in \mathbb{N}^{n} $ such that $|\alpha|=\beta$, with $ (\partial^{\alpha} \tilde{I})^{|\xi_{st}|:far}(x) $
defined by

\EQQARR{
(\p^{\alpha} \tilde{I})^{|\xi_{st}|:far} := \int_{\mathbb{R}^{n}} \comp{\eta} (|\xi| - |\xi_{st}|)
(\xi - \xi_{st})_{1}^{\alpha_1} ... (\xi - \xi_{st})_{n}^{\alpha_n} e^{i \tilde{\phi}_{x}(\xi)} \; d \xi,
}
with

\EQQARR{
\tilde{\phi}_{x}:= |\xi|^{4} + (\xi - \xi_{st}) \cdot x,
}
by integrating by parts the phase $e^{i \tilde{\phi}_x}$. Indeed observe that the stationary point
of $\tilde{\phi}_{x}$ is also $\xi_{st}$. Since we expect that the main contribution of
$ \partial^{\beta}_{\rho} \tilde{I}(x) $  to be around $\xi_{st}$, we first try to localize coarsely our analysis around $\xi_{st}$.
This procedure will not only allow us to considerably simplify the estimates of the derivatives of the phase around $\xi_{st}$
when we perform later our analysis at a finer scale around $\xi_{st}$ but also to avoid handling the boundary term at the origin
when we pass to the polar coordinates later.
Observe that if $ \left| |\xi| - |\xi_{st}| \right| \gtrsim 1$ then

\EQQARR{
|\nabla \tilde{\phi}_{x}(\xi)|  & = \left| \nabla \tilde{\phi}_{x} (\xi) - \nabla \tilde{\phi}_{x}(\xi_{st}) \right| \\
& \gtrsim \left| |\xi|^{3} - |\xi_{st}|^{3} \right| \\
& \gtrsim |\xi_{st}|^{2} \cdot
}
Hence

\EQQARR{
| |\xi| - |\xi_{st}| | \geq 1: \quad |\nabla \tilde{\phi}_{x}(\xi) | \gtrsim |\xi_{st}|^{2}  \\
\\
|\xi| \gg |\xi_{st}| : \quad |\nabla \tilde{\phi}_{x}(\xi) | \gtrsim |\xi|^{3} \cdot
}
Hence, by repeated integration by part using the formula $ e^{ i \tilde{\phi}_x} = \frac{\nabla e^{i \tilde{\phi}_x} \cdot \nabla \tilde{\phi}_x}{|\nabla \tilde{\phi}_x|^{2}}$, we easily get

\EQQARR{
N \in \mathbb{N}: \quad |(\p^{\alpha} \tilde{I})^{|\xi_{st}|:far}| \lesssim_{N} \frac{1}{|x|^{N}} \cdot
}
\\
Hence it suffices to prove

\EQQARRLAB{
\left| \p_{\rho}^{\beta}  \int_{\mathbb{R}^{n}} \eta \left( |\xi| - |\xi_{st}| \right) e^{i \tilde{\phi}_{x}(\xi)} \; d \xi \right|
\lesssim \frac{1}{|x|^{\frac{n + \beta}{3}}} \cdot
\label{Eqn:AuxEst}
}
Passing to the polar coordinates, an easy computation shows that (\ref{Eqn:AuxEst}) follows from the following estimates

\EQQARR{
|\tilde{I}_{\beta}^{\cls; \pi:k} (x)| \lesssim \frac{1}{|x|^{\frac{n + \beta}{3}}}.
}
Here $k \in \{cl,med,far \}$,

\EQQARR{
\tilde{I}_{\beta}^{\cls; \pi:k} (x) :=  |\sigma_{n-2}|  e^{i |x| |\xi_{st}|}
\int_{\mathbb{R}} \int_{0}^{\pi}   s^{n-1} \eta(s- |\xi_{st}|)
 \left( |\xi_{st}| + s \cos{\theta} \right)^{\beta}
e^{i \bar{\phi}_x}  ( \sin \theta )^{n-2} \psi^{k}(\theta) \, d \theta \, ds,
}
with $\theta$ denoting the geometric angle between $e_n$ and $\xi$, $\sigma_{n-2}$ denoting
the surface of the $n-2$ dimensional sphere,

\EQQARR{
\bar{\phi}_x := s^{4} + s |x| \cos \theta,
}
$\psi^{cl}$ (resp. $\psi^{far}$) denoting a smooth function
that is supported in a small neighborhood of $\pi$ (resp. $0$) and equals to $1$ near $\pi$ (resp. $0$), and
$\psi^{med} := 1 - \psi^{cl} - \psi^{far}$. \\
\\
We first estimate $\tilde{I}_{\beta}^{\cls; \; \medth} $ and $\tilde{I}_{\beta}^{\cls; \; \farth}$.\\
Integrating by parts w.r.t  $\theta$ (resp. w.r.t $s$) the phase $e^{i \bar{\phi}_{x}}$, we can estimate
$\tilde{I}_{\beta}^{\cls; \; \medth}$ (resp. $\tilde{I}_{\beta}^{\cls; \; \farth}$). We get

\EQQARRLAB{
N \in \mathbb{N}, \quad k \in \{cl, med\}: \quad \left| \tilde{I}_{\beta}^{\cls; \; \pi:k}  \right|  \lesssim_{N} \frac{1}{|x|^{N}} \cdot
\label{Eqn:DefIbet}
}
We then estimate $\tilde{I}_{\beta}^{\cls; \clth} $ by either integrating by parts the phase
$e^{i \bar{\phi}_{x} }$ w.r.t $s$ or $\theta$ using the formula.
Following the same strategy as in the previous section it is worth considering at first
sight two integrals:

\begin{itemize}

\item the integral appearing from the integration by parts w.r.t $s$ that contains the term
$ \p_{s} \left( \frac{1}{_{\p_{s} \bar{\phi}_{x}}} \right) $

\item the integral appearing from the integration by parts w.r.t  $\theta$ that contains the term
$\p_{\theta} \left( \frac{1}{_{\p_{\theta}\bar{\phi}_{x}}} \right)$

\end{itemize}
On the support of the integrand of $\tilde{I}_{\beta}^{\cls;\clth}$  we have

\EQQARR{
\left| \p_{s} \left( \frac{1}{_{\p_{s} \bar{\phi}_x}} \right) \right|  \lesssim
\left|  \p_{\theta} \left( \frac{1}{_{\p_{\theta} \bar{\phi}_x}}   \right) \right|
& \Leftrightarrow |\p_{s} \bar{\phi}_x| \gtrsim |\p_{\theta} \bar{\phi}_x| \left( \frac{|\xi_{st}|}{|x|} \right)^{\frac{1}{2}} \\
& \\
& \Leftrightarrow \left| 4 s^{3} - |x| + \frac{|x| (\pi - \theta)^{2}}{2} + o ( |x| (\pi - \theta)^{2} )  \right|
\gtrsim |\xi_{st}|^{\frac{3}{2}} |x|^{\frac{1}{2}} | \pi - \theta |  \\
& \\
& \Leftrightarrow (s,\theta) \in \Gamma_{s},
}
with

\EQQARR{
\Gamma_{s} := \left\{ (s,\theta): \quad \theta \in [0,\pi] \quad \text{and} \quad  \pi - \theta \lesssim \frac{|s - |\xi_{st}|| |\xi_{st}|^{2} }{|x|}   \right\} \cdot }
In order to emphasize this region we introduce  $\Omega$, 
an homogeneous function of degree zero w.r.t  $(|\xi_{st}|,\pi)$ \footnote{i.e $\Omega \left( \lambda (s - |\xi_{st}|, \theta - \pi) \right)
= \Omega  \left( s- |\xi_{st}|, \theta - \pi \right) $ for all $\lambda \neq 0$} and smooth away from $(|\xi_{st}|,\pi)$ such that
$\Omega =1$ on $\Gamma_{s}$ and $\Omega=0$ on $\Gamma_{\theta}$ with

\EQQARR{
\Gamma_{\theta}:= \left\{ (s,\theta): \quad \theta \in [0,\pi] \quad \text{and} \quad  \pi - \theta \gg \frac{\left|s - |\xi_{st}| \right| \; |\xi_{st}|^{2} }{|x|}   \right\} \cdot
}
The following holds (with $k \in \mathbb{N}$):

\EQQARRLAB{
| s- |\xi_{st}| | \approx \frac{|x| (\pi - \theta)}{ |\xi_{st}|^{2}}: \; |\p_{s}^{k} \Omega| \lesssim \frac{1}{|s - |\xi_{st}||^{k}}, \; \text{and} \;
  |\p_{\theta}^{k} \Omega| \lesssim \frac{1}{|\pi - \theta|^{k}}; \\
\\
| s- |\xi_{st}| | \gg \frac{|x| (\pi - \theta)}{ |\xi_{st}|^{2}} \; \text{or} \;
| s- |\xi_{st}| | \ll \frac{|x| (\pi - \theta)}{ |\xi_{st}|^{2}} \; \text{then} \;  \p_{s}^{k} \Omega = \p_{\theta}^{k} \Omega =0 \cdot
\label{Eqn:EstDerOmeg2}
}
The same estimates hold for $\comp{\Omega}$.\\
\\
Hence, in view of (\ref{Eqn:DefIbet}) we are reduced to estimate  $ \tilde{I}_{s}^{\cls; \; \clth} $
and $\tilde{I}_{\theta}^{\cls; \; \clth} $ defined by the following

\EQQARR{
\tilde{I}_{s}^{\cls; \; \clth}    :=
\int_{\mathbb{R}} \int_{0}^{\pi} e^{i \bar{\phi}_{x}} \Omega \;  s^{n-1} \eta(s- |\xi_{st}|)
 \left( |\xi_{st}| + s \cos{\theta} \right)^{\beta}
 ( \sin \theta )^{n-2} \psi^{cl}(\theta) \, d \theta \, ds, \quad \text{and}
}

\EQQARR{
\tilde{I}_{\theta}^{\cls; \; \clth}   :=
\int_{\mathbb{R}} \int_{0}^{\pi} e^{i \bar{\phi}_{x}}  \comp{\Omega} \;   s^{n-1}  \eta(s- |\xi_{st}|)
 \left( |\xi_{st}| + s \cos{\theta} \right)^{\beta}
 ( \sin \theta )^{n-2} \psi^{cl}(\theta) \, d \theta \, ds \cdot
}
\\
We first estimate $\tilde{I}_{s}^{\cls; \clth}$. \\
We have
$\tilde{I}_{s}^{\cls; \; \clth}(x) =  \tilde{I}_{s,cl}^{\cls; \; \clth}  +
\tilde{I}_{s,far}^{\cls; \; \clth} $ with

\EQQARR{
\tilde{I}_{s,cl}^{\cls; \; \clth} := \int_{\mathbb{R}} \int_{0}^{\pi} e^{i \bar{\phi}_{x}} \eta \left( \frac{s - |\xi_{st}|}{\epsilon} \right)
\Omega \;   s^{n-1}
\eta(s- |\xi_{st}|)
 \left( |\xi_{st}| + s \cos{\theta} \right)^{\beta} ( \sin \theta )^{n-2} \psi^{cl}(\theta) \, d \theta \, ds, \; \text{and} \\
\\
\tilde{I}_{s,far}^{\cls; \; \clth}
:= \int_{\mathbb{R}} \int_{0}^{\pi} e^{i \bar{\phi}_{x}}
 \comp{\eta} \left( \frac{s - |\xi_{st}|}{\epsilon} \right)
\Omega \;   s^{n-1} \eta(s- |\xi_{st}|)
 \left( |\xi_{st}| + s \cos{\theta} \right)^{\beta}
( \sin \theta )^{n-2} \psi^{cl}(\theta) \, d \theta \, ds \cdot
}
Observe that on the support of $\Omega$

\EQQARR{
(s,\theta) \in \supp{(\Omega)}: \; | |\xi_{st}| + s \cos{\theta} | \approx  \left| s - |\xi_{st}| \right|, \\
\\
(s,\theta) \in \supp{(\Omega)}: \; |\p_{s} \bar{\phi}_{x}| \approx |\xi_{st}|^{2} \left|s - |\xi_{st}|\right| \cdot \\
}
These observations are implicitely used in the sequel. We have

\EQQARR{
\left| \tilde{I}_{s,cl}^{\cls; \; \clth} \right| & \lesssim |\xi_{st}|^{n-1}
\int_{\left| s - |\xi_{st}| \right| \lesssim \epsilon}  \left| s - |\xi_{st}| \right|^{\beta}
\int_{|\pi - \theta| \lesssim \frac{|s - |\xi_{st}|| |\xi_{st}|^{2}}{|x|}} |\pi - \theta|^{n-2} \; d \theta \; ds  \\
\\
& \lesssim |\xi_{st}|^{n -1} \frac{ \epsilon^{n+ \beta}}{|x|^{\frac{n-1}{3}}} \cdot
}
We estimate $\tilde{I}_{s,far}^{\cls; \; \clth}$ by integration by parts of the phase
$e^{i \bar{\phi}_x}$ w.r.t $s$. If we integrate
by parts $\bar{p}$ times with $\bar{p} \gg 1$, we have to estimate many integrals. It is worth choosing $\epsilon$
by considering only the integral $K$ that contains only the term
$ \left( \p_{s} \left( \frac{1}{ _{\p_{s}\bar{\phi}_{x}}} \right) \right)^{\bar{p}}$, i.e

\EQQARR{
K := \int_{\mathbb{R}} \int_{0}^{\pi} e^{i \bar{\phi}_{x}} \left( \p_{s} \left( \frac{1}{_{\p_{s} \bar{\phi}_{x}}} \right) \right)^{\bar{p}}
\comp{\eta} \left( \frac{s - |\xi_{st}|}{\epsilon} \right)
\Omega \;   s^{n-1}
\eta(s- |\xi_{st}|)
 \left( |\xi_{st}| + s \cos{\theta} \right)^{\beta} ( \sin \theta )^{n-2} \psi^{cl}(\theta) \, d \theta \, ds
}
We have

\EQQARRLAB{
\label{Eqn:RHS}
|K|  & \lesssim |\xi_{st}|^{n-1} \int_{|s - |\xi_{st}|| \gtrsim \epsilon} \frac{| |\xi_{st}| - s|^{\beta} }{ |\xi_{st}|^{2\bar{p}}|s - |\xi_{st}||^{2\bar{p}}}
\int_{ |\pi - \theta| \lesssim \frac{|s - |\xi_{st}|| |\xi_{st}|^{2}}{|x|}} |\pi - \theta|^{n-2} \; d \theta \; d s
}
Hence

\EQQARR{
|K| & \lesssim |\xi_{st}|^{n-1} \frac{\epsilon^{n+ \beta}}{\epsilon^{2\bar{p}} |\xi_{st}|^{2\bar{p}} |x|^{\frac{n-1}{3}}}
}
Hence, optimizing in $\epsilon$  the upper bound of  $ | \tilde{I}_{s,cl}^{\cls; \; \clth} |  + |K|$, we see that
$\epsilon \approx |\xi_{st}|^{-1}$. \\
\\
We now estimate the other integrals with this value of $\epsilon$.\\
Given $p \in \mathbb{N}$, let \footnote{Here $\mathcal{V}(\pi)$ denotes a small neighborhood of $\pi$}

\EQQARR{
\mathcal{Q}_{p} := \left\{ f \in C^{\infty} \left( \{  0  < |s- |\xi_{st}|| \lesssim 1  \} \times  \mathcal{V}(\pi)  \right): \;
|\p^{\alpha}_{s} f (s,\theta)| \lesssim
\frac{1}{|\xi_{st}|^{2p} |s- |\xi_{st}||^{ 2p+ \alpha}}, \; \alpha \in \mathbb{N} \right\} \cdot
}
and \footnote{a function $f$ behaves like $g$ if it (and its derivatives) satisfy the same estimates or better estimates as
$g$ (and its derivatives)}

\EQQARR{
K_{p} := \left\{
\begin{array}{l}
\int_{\mathbb{R}} \int_{0}^{\pi} e^{i \bar{\phi}_{x}}
 f(s,\theta) \; \bar{\eta} \left( \frac{s -|\xi_{st}|}{\epsilon} \right) \; \tilde{\Omega} \; \tilde{\eta} (s - |\xi_{st}|)
s^{n-1} \;  (|\xi_{st}| + s \cos{\theta})^{\beta} ( \sin \theta )^{n-2} \psi^{cl}(\theta) \; d\theta \; ds, \\
f \in \mathcal{Q}_{p} \; \; and \; \;  \bar{\eta} \; ( resp. \; \tilde{\Omega} \; and \; \tilde{\eta}) \;
behaving \; like \; \comp{\eta} \; ( resp. \; \Omega \; and \; \eta)
\end{array}
\right\} \cdot
}
Integrating by parts the phase $e^{i \bar{\phi}_{x}}$  w.r.t $s$, we see, in view of (\ref{Eqn:EstDerOmeg2}), that

\EQQARR{
K_{p} \subset K_{p+1} \cdot
}
Since $\tilde{I}_{s,far}^{\cls; \; \clth} \in K_{0}$, we get

\EQQARR{
|\tilde{I}_{s,far}^{\cls; \; \clth}| \lesssim RHS \; \; \text{of} \; \; (\ref{Eqn:RHS}) \lesssim  |\xi_{st}|^{n-1}
\frac{\epsilon^{n+ \beta}}{\epsilon^{2 \bar{p}} |\xi_{st}|^{2 \bar{p}} |x|^{\frac{n-1}{3}}} \cdot
}
Hence

\EQQARR{
| \tilde{I}_{s}^{\cls; \; \clth} | \lesssim \frac{1}{|x|^{\frac{n+ \beta}{3} }} \cdot
}
We then estimate $\tilde{I}_{\theta}^{\cls; \; \clth}$. We have
$\tilde{I}_{\theta}^{\cls; \; \clth} =  \tilde{I}_{\theta,cl}^{\cls; \; \clth}  +
\tilde{I}_{\theta,far}^{\cls; \; \clth} $ with

\EQQARR{
\tilde{I}_{\theta,cl}^{\cls; \; \clth} := \int_{\mathbb{R}} \int_{0}^{\pi} e^{i \bar{\phi}_{x}} \eta \left( \frac{\theta - \pi}{\epsilon} \right) \comp{\Omega} \;   s^{n-1}
\eta(s- |\xi_{st}|)
 \left( |\xi_{st}| + s \cos{\theta} \right)^{\beta} ( \sin \theta )^{n-2} \psi^{cl}(\theta) \, d \theta \, ds, \;  \text{and} \\
\\
\tilde{I}_{\theta,far}^{\cls; \; \clth}
:= \int_{\mathbb{R}} \int_{0}^{\pi}  e^{i \bar{\phi}_{x}}
 \comp{\eta} \left( \frac{\theta - \pi}{\epsilon} \right)
\comp{\Omega} \;   s^{n-1} \eta(s- |\xi_{st}|)
 \left( |\xi_{st}| + s \cos{\theta} \right)^{\beta}
( \sin \theta )^{n-2} \psi^{cl}(\theta) \, d \theta \, ds \cdot
}
We have

\EQQARR{
| \tilde{I}_{\theta,cl}^{\cls; \; \clth} | & \lesssim |\xi_{st}|^{n-1}
\int_{|\pi - \theta| \lesssim \epsilon}
\int_{|s - |\xi_{st}|| \lesssim \frac{|x| |\pi - \theta|}{|\xi_{st}|^{2}}}
\max
\left(
\left| s - |\xi_{st}| \right|^{\beta}, \; |\xi_{st}|^{\beta} |\pi - \theta|^{2\beta}
\right)
 \; ds \;
|\pi - \theta|^{n-2}
\;  d \theta \\
& \\
& \lesssim \max \left( \epsilon^{\beta +n}, \epsilon^{2 \beta + n} \right) |\xi_{st}|^{\beta + n} \cdot
}
We estimate  $\tilde{I}_{\theta,far}^{\cls; \; \clth}$ by integration by parts, using the formula
$\p_{\theta} e^{i \bar{\phi}_{x}} = i \p_{\theta} \bar{\phi}_{x} e^{i \bar{\phi}_{x}}$. Again, if we integrate by parts
$\bar{p}$ times with $\bar{p} \gg 1$, we have to estimate many integrals. It is worth choosing $\epsilon$ by considering only
the integral $K$ that contains only the term $  \left( \p_{\theta} \left( \frac{1}{_{\p_{\theta} \bar{\phi}_{x}}} \right) \right)^{\bar{p}}$, i.e

\EQQARR{
K := \int_{\mathbb{R}} \int_{0}^{\pi}  e^{i \bar{\phi}_{x}} \left( \p_{\theta} \left( \frac{1}{_{ \p_{\theta} \bar{\phi}_{x}}} \right) \right)^{\bar{p}}
\comp{\eta} \left( \frac{\theta - \pi}{\epsilon} \right)
\Omega \;   s^{n-1}
\eta(s- |\xi_{st}|)
 \left( |\xi_{st}| + s \cos{\theta} \right)^{\beta} ( \sin \theta )^{n-2} \psi^{cl}(\theta) \, d \theta \, ds \cdot
}
We have

\EQQARR{
|K| & \lesssim |\xi_{st}|^{n-1} \int_{|\pi - \theta| \gtrsim \epsilon}  \frac{ |\pi - \theta|^{n-2}} {(|\xi_{st}| |x|)^{\bar{p}} (\pi - \theta)^{2 \bar{p} }}
\int_{|s - |\xi_{st}|| \lesssim \frac{|x| |\pi - \theta|}{|\xi_{st}|^{2}}}
\max \left( \left| s - |\xi_{st}| \right|^{\beta},  |\xi_{st}|^{\beta} |\pi - \theta|^{2 \beta} \right) \; ds \;  d \theta \\
& \\
& \lesssim  \frac{|\xi_{st}|^{\beta + n} \max{( \epsilon^{ \beta + n}, \epsilon^{ 2 \beta + n})}}{\epsilon^{2 \bar{p}} |\xi_{st}|^{4 \bar{p}} } \cdot
}
Hence, optimizing the upper bound of $|K| +  |\tilde{I}_{\theta,cl}^{\cls; \; \clth}| $, we find $\epsilon \approx |\xi_{st}|^{-2}$. \\
\\
We now estimate the other integrals with this value of $\epsilon$. \\
Observe that if $q(s,\theta):=  ( |\xi_{st}| + s \cos{\theta})^{\beta} $, then, on the support of the
integrand of $ \tilde{I}_{\theta,far}^{\cls; \; \clth} $, we have

\begin{equation}
\begin{array}{ll}
\alpha \leq \beta: \; | \p^{\alpha}_{\theta} q(s,\theta)| &
\lesssim \sum \limits_{i=1}^{\alpha} \left| |\xi_{st}| + s \cos{\theta} \right|^{\beta -i} |\xi_{st}|^{i} \\                                                      & \lesssim \sum \limits_{i=1}^{\alpha}  |\xi_{st}|^{i} \left|  |\xi_{st}| - s \right|^{\beta -i} +
|\xi_{st}|^{\beta} |\pi - \theta|^{2(\beta-i)}, \; \text{and} \\
& \\
\alpha > \beta: \; \p^{\alpha}_{\theta} q(s,\theta) & = 0 \cdot
\end{array}
\label{Eqn:EstDerq}
\end{equation}
Let $\vec{r} := (r_{l})_{l \in [1..3]} \in \mathbb{N}^{3}$, $\overrightarrow{r_{l,+}}:= (r_1,r_{l+1},r_3)$ for
$l \in \{ 1,2,3 \}$, and $\vec{0}:=(0)_{l \in [1..3]}$. Given $(p,q) \in \mathbb{N}^{2}$ let

\EQQARR{
\mathcal{Q}_{p,q} := \left\{ f \in C^{\infty} \left( 0 <  |\pi - \theta| \ll 1 \right): \; |\p^{\alpha}_{\theta} f (s,\theta)|
\lesssim \frac{1}{(|\xi_{st}| |x|)^{p} |\pi - \theta|^{p+q + \alpha}}  \right\} \cdot
}
Given $ f \in \mathcal{Q}_{p,r_1}$,  we define
$K_{\vec{r}}(f) :=  \int_{\mathbb{R}} \int_{0}^{\pi} e^{i \bar{\phi}_{x}} X_{\vec{r}}(f) \; d s \; d \theta$, with

\EQQARR{
X_{\vec{r}} (f) := f(s,\theta) \p^{r_2}_{\theta} \left( \comp{\eta} \left( \frac{\theta - \pi}{\epsilon} \right) \comp{\Omega}
(\sin \theta)^{n-2} \psi^{cl}(\theta) \right) \p^{r_3}_{\theta}  q(s,\theta)  \; s^{n-1}  \eta \left( s- |\xi_{st}| \right) \cdot }
Integrating by parts w.r.t $\theta$, we see that there exists $g \in \mathcal{Q}_{p+1, r_{1}+1}$ and
$ h  \in \mathcal{Q}_{p+1, r_{1}}$ such that

\EQQARR{
K_{\vec{r}}(f) & = \int_{\mathbb{R}^{2}} \frac{\p_{\theta} e^{i \bar{\phi}_{x}}}{  _{i \p_{\theta} \bar{\phi}}} X_{\vec{j}}(f) \; d \theta \; ds  \\
& = - \int_{\mathbb{R}} \int_{0}^{\pi} e^{i \bar{\phi}_{x}}  \p_{\theta} \left( \frac{1}{ _{ i \p_{\theta} \bar{\phi}_{x}}} X_{\vec{r}}(f) \right) \;
d \theta \; ds \\
& = i \left( K_{\overrightarrow{r_{1,+}}} (g) + \sum \limits_{l=2}^{3} K_{\overrightarrow{r_{l,+}}} (h) \right) \cdot
}
Hence, since $ K_{\vec{0}}(1) = \tilde{I}_{\theta,far}^{\cls; \; \clth}$, we see by iteration over $p$ that we are reduced to estimate
$K_{\vec{r}}(f)$ for $\vec{r}$ such that $  \sum \limits_{l=1}^{3} r_{l} = \bar{p}$ and $f \in \mathcal{Q}_{\bar{p},r_1}$. \\
Hence, in view of (\ref{Eqn:EstDerOmeg2}) and (\ref{Eqn:EstDerq})

\EQQARR{
| \tilde{I}_{\theta,far}^{\cls; \; \clth} | & \lesssim \frac{|\xi_{st}|^{n-1}}{(|\xi_{st}| |x|)^{\bar{p}}}  \sum \limits_{k=1}^{r_3}
  \int_{|\pi -\theta| \gtrsim \epsilon }  \frac{|\pi -\theta|^{n-2}}{|\pi- \theta|^{\bar{p} + r_1 + r_2}}
 \int_{|s- |\xi_{st}| | \lesssim \frac{|x|}{|\xi_{st}|^{2}} |\pi - \theta|} Z_{k,\beta}  \; ds   \; d \theta, \\
}
with

\EQQARR{
Z_{k,\beta} & :=
\max \left( | |\xi_{st}| - s|^{\beta - k} |\xi_{st}|^{k}, |\xi_{st}|^{\beta} |\pi - \theta|^{2(\beta-k)} \right) \cdot
}
Hence

\EQQARR{
| \tilde{I}_{\theta,far}^{\cls; \; \clth} | \lesssim  \frac{|\xi_{st}|^{\beta +n} \max(
\epsilon^{\beta + n}, \epsilon^{2 \beta +n})}{|\xi_{st}|^{4 \bar{p}} \epsilon^{2 \bar{p}}} \lesssim \frac{1}{|x|^{\frac{n+ \beta}{3} }} \cdot
}

\label{Sec:OscInt}

\end{document}